%% file: arxiv-main.tex
\documentclass{article}
\pdfoutput=1

\usepackage{preamble}
\usepackage[style=authoryear-comp, bibencoding=utf8]{biblatex}

\DeclareNameAlias{author}{given-family}

\title{\TITLE}
\author{Hoang Kim Nguyen \and Taichi Uemura}

\begin{document}

\maketitle

\begin{abstract}
  \input{abstract}
\end{abstract}

\input{sections}

\printbibliography

\end{document}

%% file: abstract.tex
We introduce $\infty$-type theories as an $\infty$-categorical
generalization of the categorical definition of type theories introduced by the second named author.
We establish analogous results to the previous work including the
construction of initial models of $\infty$-type theories, the
construction of internal languages of models of $\infty$-type theories, and the
theory-model correspondence for $\infty$-type theories.
Some structured $(\infty,1)$-categories are naturally regarded as
models of some $\infty$-type theories.
Thus, since every (1-categorical) type theory is in particular an $\infty$-type theory,
$\infty$-type theories provide a unified framework for
connections between type theories and $(\infty,1)$-categorical structures.
As an application we prove Kapulkin
and Lumsdaine's conjecture that the dependent type theory with
intensional identity types gives internal languages for
$(\infty,1)$-categories with finite limits.


%% file: sections.tex

\section{Introduction}
\label{sec:introduction}

\emph{Type theory} and \emph{higher category theory} are closely
related: dependent type theories with intensional identity types
provide a syntactic way of reasoning about
\((\infty,1)\)-categories. This is known as the family of
\emph{internal language conjectures} and has led for example to syntactic developments of classical material in homotopy theory such as the homotopy groups of spheres \parencite{brunerie2016thesis, brunerie2019spheres, licata2013spheres} and the Blakers-Massey Theorem \parencite{hou2016blakersmassey}, just to name a few. These proofs often lead to new perspectives on classical material and their nature makes them applicable to a wider class of \((\infty,1)\)-categories, importing ideas from the homotopy theory of spaces to other \((\infty,1)\)-categories, see for example \parencite{anel2018goodwillie} and \parencite{anel2020blakersmassey}. One of the main appeals of type theory for higher category theory and homotopy theory is thus the usage of this type theoretic language to reason in a synthetic way.
On the other hand, higher categories will be useful for the
study of type theories. For example, one can expect a conceptual proof
of Voevodsky's homotopy canonicity conjecture that any closed term of the type of natural
numbers is homotopic to a numeral using a higher dimensional analogue of
the Freyd cover \parencite{lambek1986higher}.

However, internal language conjectures are still open problems in \emph{homotopy type theory}
\parencite{hottbook}.
The advantage of type-theoretic languages, that a lot of equations \emph{strictly} hold in type theories so that
a lot of trivial homotopies in \((\infty,1)\)-categories can be
eliminated, is at the same time the main difficulty of
internal language conjectures. One has to justify interpreting strict
equality in type theories as homotopies in
\((\infty,1)\)-categories. This is an \(\infty\)-dimensional version
of the \emph{coherence problem} in the categorical semantics of type
theories.

An internal language conjecture should be formulated as an equivalence
between an \((\infty,1)\)-category of theories and an
\((\infty,1)\)-category of structured
\((\infty,1)\)-categories. Currently, only a few internal language
conjectures have been made precise. \Textcite{kapulkin2018homotopy}
made precise formulations of the simplest cases and conjectured that
the \((\infty,1)\)-category of theories over Martin-L{\"o}f type
theory with intensional identity types (and dependent function types
with function extensionality) is equivalent to the
\((\infty,1)\)-category of small \((\infty,1)\)-categories with finite
limits (and pushforwards). In this paper, we prove
\citeauthor{kapulkin2018homotopy}'s conjecture by introducing a novel
\(\infty\)-dimensional generalization of type theories which we call
\emph{\(\infty\)-type theories}.

The basic strategy for proving \citeauthor{kapulkin2018homotopy}'s
conjecture is to decompose the equivalence to be proved into smaller
pieces. An existing approach is to introduce \(1\)-categorical
presentations of \((\infty,1)\)-categories with finite limits. It had
already been shown by \textcite{szumilo2014two} that
\((\infty,1)\)-categories with finite limits are equivalent to
categories of fibrant objects in the sense of
\textcite{brown1973abstract}. \Textcite{kapulkin2019internal} then
proved that categories of fibrant objects are equivalent to
\citeauthor{joyal2017clans}'s tribes
\parencite{joyal2017clans}. Tribes are considered as \(1\)-categorical
models of the type theory, but a full proof of the equivalence between
tribes and theories has not yet been achieved.

Although this approach is natural for those who know the homotopical
interpretation of intensional identity types
\parencite{awodey2009homotopy,arndt2011homotopy,shulman2015inverse},
\(1\)-categorical models of intensional identity types are not
convenient to work with. A problem is that \(1\)-categories of
\(1\)-categorical models of type theories need not be rich enough to
calculate the \((\infty,1)\)-categories they present. It is also
unclear if this approach can be generalized to internal language
conjectures for richer type theories.

In this paper we seek another path. The key idea is to introduce a
notion of \emph{\(\infty\)-type theories}, an \(\infty\)-dimensional
generalization of type theories. Intuitively, an \(\infty\)-type
theory is a kind of type theory but equality is like homotopies rather
than strict one. Ordinary type theories are considered as truncated
\(\infty\)-type theories in the sense that all homotopies are
trivial.

Our proof strategy is as follows. Let \(\itth\) denote the type theory
with intensional identity types. We introduce an \(\infty\)-type
theory \(\itth_{\infty}\) which is analogous to \(\itth\) but without
truncation. Because \(\itth_{\infty}\) is already a higher dimensional
object, it is straightforward to interpret \(\itth_{\infty}\) in
\((\infty,1)\)-categories with finite limits. The internal language
conjecture is then reduced to a coherence problem between \(\itth\)
and \(\itth_{\infty}\): how to interpret \(\itth\) in models of
\(\itth_{\infty}\). Although this coherence problem is as difficult as
the original internal language conjecture, this reduction is an
important step. Since the problem is now formulated in the language of
\(\infty\)-type theories and related concepts, our proof strategy is
easily generalized to internal language conjectures for richer type
theories. When we extend \(\itth\) by some type constructors, we just
extend \(\itth_{\infty}\) in the same way.

A solution to coherence problems in the \(1\)-categorical semantics of
type theories given by \textcite{hofmann1995interpretation} is to
replace a ``non-split'' model, in which equality between types is up
to isomorphism, by an equivalent ``split'' model, in which equality
between types is strict. In our approach, models of \(\itth_{\infty}\)
are like non-split models of \(\itth\), so we consider replacing a
model of \(\itth_{\infty}\) by an equivalent model of
\(\itth\). Splitting techniques for \((\infty,1)\)-categorical
structures have not yet been fully developed except for some
presentable \((\infty,1)\)-categories considered by
\textcite{gepner2017univalence,shulman2019toposes}. Since we have to
split small \((\infty,1)\)-categories which are usually
non-presentable, their results cannot directly apply. However, as he
already mentioned in \parencite[Remark 1.4]{shulman2019toposes},
\citeauthor{shulman2019toposes}'s result on splitting presentable
\((\infty,1)\)-toposes can be used for splitting small
\((\infty,1)\)-categories by embedding them into presheaf
\((\infty,1)\)-toposes.

\paragraph{Organization}

In \cref{sec:preliminaries} we fix notations and remember some
concepts in \((\infty,1)\)-category theory. Relevant concepts to this
paper are \emph{\(\infty\)-cosmoi} \parencite{riehl2018elements},
\emph{compactly generated} \((\infty,1)\)-categories,
\emph{exponentiable} arrows, and \emph{representable maps} between
right fibrations.

We introduce the notion of an \(\infty\)-type theory in
\cref{sec:infty-type-theories}. It is defined as an
\((\infty,1)\)-category with a certain structure, generalizing the
categorical definition of type theories introduced by the second named
author \parencite{uemura2019framework}. There are two important
notions around \(\infty\)-type theories: \emph{models} and
\emph{theories}. The notion of models we have in mind is a
generalization of categories with families
\parencite{dybjer1996internal} and, equivalently, natural models
\parencite{awodey2018natural,fiore2012discrete}. The notion of
theories is close to the essentially algebraic definitions of theories
given by
\textcite{garner2015combinatorial,isaev2017algebraic,voevodsky2014b-system}.

In \cref{sec:theory-model-corr} we prove \(\infty\)-analogue of the
main results of the previous work
\parencite{uemura2019framework}. Given an \(\infty\)-type theory
\(\tth\), we construct a functor that assigns to each model of
\(\tth\) a \(\tth\)-theory called the \emph{internal language} of the
model. The internal language functor has a fully faithful left adjoint
which constructs a \emph{syntactic model} from a \(\tth\)-theory. We
further characterize the image of the left adjoint.

We study some concrete \(\infty\)-type theories in
\cref{sec:corr-betw-type}. The most basic example is
\(\etth_{\infty}\), the \(\infty\)-analogue of Martin-L{\"o}f type
theory with \emph{extensional} identity types. We show that
\(\etth_{\infty}\)-theories are equivalent to
\((\infty,1)\)-categories with finite limits (\cref{etth-dem-lex}),
which is an \(\infty\)-analogue of the result of
\textcite{clairambault2014biequivalence}. This is to be an
intermediate step toward \citeauthor{kapulkin2018homotopy}'s
conjecture, but it also has an interesting corollary. One can derive a
new universal property of the \((\infty,1)\)-category of small
\((\infty,1)\)-categories with finite limits from a universal property
of \(\etth_{\infty}\) (\cref{Lex-ump}). We also study a couple of
examples of \(\infty\)-type theories with dependent function
types. Finally in \cref{sec:intern-lang-left} we prove
\citeauthor{kapulkin2018homotopy}'s conjecture.

\section{Preliminaries}
\label{sec:preliminaries}

\subsection{\(\infty\)-categories}

For concreteness, we will work with \emph{\(\infty\)-categories}, also called \emph{quasicategories} in the literature,  \parencite{joyal2008notes, lurie2009higher, cisinski2019higher} as models for \((\infty,1)\)-categories. An \(\infty\)-category is a simplicial set satisfying certain horn filling conditions. We recollect some standard definitions and notations.

\begin{definition}
\begin{enumerate}
\item Given an \(\infty\)-category \(\cat\) and a simplicial set \(\sh\), we denote \(\Fun(\sh, \cat)\) the internal hom of simplicial sets, which is itself an \(\infty\)-category and models the \(\infty\)-category of functors and natural transformations.

\item For an \(\infty\)-category \(\cat\), we denote by \(\cat^\simeq\) the largest \(\infty\)-groupoid (Kan complex) contained in \(\cat\). Furthermore we write \(\lgpd (\cat,\catI):= \Fun(\cat,\catI)^\simeq\).

\item For an \(\infty\)-category \(\cat\), we denote by \(C^\triangleright\) the join \(\cat \star \Delta^0\).

\item We denote by \(\Cat_\infty\) the \(\infty\)-category of small \(\infty\)-categories. This is obtained as the homotopy coherent nerve of the simplicial category with objects given by small \(\infty\)-categories and hom simplicial sets given by \(\lgpd(\cat, \catI)\).

\item We denote by \(\CAT_\infty\) the \(\infty\)-category of (possibly large) \(\infty\)-categories obtained in a similar way.

\item We denote by \(\Space\) the \(\infty\)-category of small \(\infty\)-groupoids obtained as the homotopy coherent nerve of the simplicial category with objects small Kan complexes and hom simplicial sets given by the internal hom of simplicial sets.
\end{enumerate}
\end{definition}

Although we chose to work with \(\infty\)-categories, we will primarily use the language of the formal category theory of \(\infty\)-categories as expressed by \(\infty\)-cosmoi. Therefore, most of our constructions, statements and proofs are independent of the model.

\subsection{\(\infty\)-cosmoi}

An \emph{\(\infty\)-cosmos} \parencite{riehl2018elements} is, roughly,
a complete \((\infty,2)\)-category with enough structure to do formal
category theory. More concretely, an \(\infty\)-cosmos \(\cosmos\) is
a simplicially enriched category such that for any pair of objects
\(\cat,\catI \in \cosmos\), the hom simplicial set
\(\cosmos(\cat,\catI)\) is an \(\infty\)-category. \(\cosmos\) is also
equipped a class of morphisms called isofibrations, and all small
\((\infty,1)\)-categorical limits are constructible from products and
pullbacks of isofibrations. Moreover, \(\cosmos\) has cotensors with
small simplicial sets \(\sh\cotensor \cat\) characterized by the
equivalence (isomorphism, in fact) of \(\infty\)-categories
\[
\cosmos( \catI, \sh \cotensor \cat) \simeq \Fun ( \sh, \cosmos(\catI,\cat)).
\]

Given an \(\infty\)-category, we may take its \emph{homotopy category}, which is just an ordinary category. Applying this to the hom spaces of an \(\infty\)-cosmos gives rise to a 2-category. Adjunctions and equivalences in \(\infty\)-cosmoi are then defined in the usual way using this 2-category.

\begin{example}
We denote by \(\kCAT_\infty\) the \(\infty\)-cosmos of (possibly large) \(\infty\)-categories. That is, \(\kCAT_\infty\) is the simplicial category with objects (possibly large) \(\infty\)-categories and hom simplicial sets \(\Fun(\cat,\catI)\). The  cotensor \(\sh \cotensor \cat\) in \(\kCAT_\infty\) is given by the functor \(\infty\)-category \(\Fun(\sh ,\cat)\) and adjunctions and equivalences agree with the standard notions of \(\infty\)-categories.
\end{example}

Cartesian fibrations and right fibrations in \(\infty\)-cosmoi are
characterized by analogy with those in complete 2-categories. Here we
prefer to work with versions of these concepts that are invariant
under equivalence. The following definition coincides with
\citeauthor{riehl2018elements}'s when \(\fun\) is an isofibration.

\begin{definition}
A functor \(\fun\colon \cat \to \catI\) in an \(\infty\)-cosmos \(\cosmos\) is said to be a \emph{cartesian fibration} if the functor
\[
(\ev_1, \fun_\ast)\colon \Delta^1 \cotensor \cat \to \cat \times_{\catI} \Delta^1 \cotensor \catI
\]
has a right adjoint with invertible counit. A \emph{fibred functor}
between cartesian fibrations is a morphism in \(\cosmos^{\to}\)
that commutes with the right adjoint of \((\ev_{1}, \fun_{\ast})\). A
cartesian fibration is a \emph{right fibration} if
\((\ev_1, \fun_\ast)\) is an equivalence.

For a small \(\infty\)-category \(\cat\), we denote by
\(\CartFib_{\cat}\subset \Cat_\infty/\cat\) the \(\infty\)-category of
cartesian fibrations over \(\cat\) and fibred functors over
\(\cat\). We denote by \(\RFib_{\cat}\subset \CartFib_{\cat}\) the
full subcategory spanned by the right fibrations over \(\cat\). Note
that any functor between right fibrations over \(\cat\) is
automatically a fibred functor, so \(\RFib_{\cat}\) is a full
subcategory of \(\Cat_{\infty}/\cat\). We
write \(\RFib \subset \cat_{\infty}^{\to}\) for the full subcategory
spanned by the right fibrations.
\end{definition}

\subsection{Compactly generated \(\infty\)-categories}

\begin{definition}[{\textcite[Definition 5.5.7.1 and Theorem
    5.5.1.1]{lurie2009higher}}]
  An \(\infty\)-category \(\cat\) is said to be
  \emph{compactly generated} if it is an
  \(\omega\)-accessible localization of \(\Fun(\catI^{\op}, \Space)\),
  that is, a reflective full subcategory of \(\Fun(\catI^{\op},
  \Space)\) closed under filtered colimits,
  for some small \(\infty\)-category \(\catI\). The subcategory of \(\CAT_{\infty}\) spanned by the compactly generated \(\infty\)-categories and \(\omega\)-accessible right adjoints is denoted by \(\PrR_{\omega}\). We will moreover denote by \(\kPrR_{\omega} \subset \kCAT_{\infty}\) the locally full subcategory spanned by the compactly
  generated \(\infty\)-categories and \(\omega\)-accessible right
  adjoints.
\end{definition}

Recall \parencite[Proposition 5.5.7.6]{lurie2009higher} that \(\PrR_{\omega}\subset \CAT_{\infty}\) is closed under small limits. By definition, compactly generated \(\infty\)-categories are closed in \(\kCAT_{\infty}\) under cotensors with small simplicial sets. Hence, the subcategory
  \(\kPrR_{\omega} \subset \kCAT_{\infty}\) is an \(\infty\)-cosmos, and the inclusion
  \(\kPrR_{\omega} \to \kCAT_{\infty}\) preserves the structures of
  \(\infty\)-cosmoi and reflects equivalences.

\begin{example}
  The \(\infty\)-category \(\Space\) of small spaces is compactly
  generated. The \(\infty\)-category \(\Cat_{\infty}\) of small
  \(\infty\)-categories is compactly generated, and the functor
  \(\lgpd(\Delta^{n}, {-}) : \Cat_{\infty} \to \Space\) sending an
  \(\infty\)-category \(\cat\) to the space of \(n\)-cells of \(\cat\)
  is an \(\omega\)-accessible right adjoint. This is because
  \(\Cat_{\infty}\) is regarded as an \(\omega\)-accessible
  localization of \(\Fun(\Delta^{\op}, \Space)\) using the equivalence
  of quasicategories and complete Segal spaces
  \parencite{joyal2007quasi}.
\end{example}

\begin{example}
  For a small \(\infty\)-category \(\cat\), the \(\infty\)-category
  \(\CartFib_{\cat}\) of cartesian fibrations over \(\cat\) and fibred
  functors over \(\cat\) is compactly generated as
  \(\CartFib_{\cat} \simeq \Fun(\cat^{\op}, \Cat_{\infty})\). The
  forgetful functor \(\CartFib_{\cat} \to \Cat_{\infty}/\cat\) is an
  \(\omega\)-accessible right adjoint. To see this, observe that this
  forgetful functor is the right derived functor of the forgetful
  functor
  \[
    \SSet^{+}/\cat^{\sharp} \to \SSet/\cat
  \]
  which is a right Quillen functor with respect to the cartesian model
  structure and the slice model structure of the Joyal model structure
  on \(\SSet\)\parencite[Proposition 3.1.5.2]{lurie2009higher} or  \parencite[Proposition
  3.1.18]{nguyen2019thesis}. The functor
  \(\SSet^{+}/\cat^{\sharp} \to \SSet/\cat\) preserves filtered
  colimits, and filtered colimits are homotopy colimits in both model
  structures, from which it follows that the right derived functor
  preserves filtered colimits.
\end{example}

\begin{example}
  \label{LAdj-cg}
  We define a subcategory \(\LAdj \subset
  \Cat_{\infty}^{\to}\) to be the pullback
  \[
    \begin{tikzcd}
      \LAdj
      \arrow[rr]
      \arrow[d] & &
      \CartFib_{\Delta^{1}}
      \arrow[d] \\
      \Cat_{\infty}^{\to}
      \arrow[r,"\simeq"'] &
      \coCartFib_{\Delta^{1}}
      \arrow[r] &
      \Cat_{\infty}/\Delta^{1}.
    \end{tikzcd}
  \]
  By construction, \(\LAdj\) is compactly generated, and the forgetful
  functor \(\LAdj \to \Cat_{\infty}^{\to}\) is a conservative,
  \(\omega\)-accessible right adjoint. Since a functor
  \(\fun : \catII \to \Delta^{1}\) that is both a cocartesian
  fibration and a cartesian fibration can be identified with an
  adjunction between the fibers over \(0\) and \(1\), the
  \(\infty\)-category \(\LAdj\) can be described as follows:
  \begin{itemize}
  \item the objects are the functors \(\fun : \cat \to \catI\) that
    have a right adjoint \(\fun^{*}\);
  \item the morphisms
    \((\fun_{1} : \cat_{1} \to \catI_{1}) \to (\fun_{2} : \cat_{2} \to
    \catI_{2})\) are the squares
    \[
      \begin{tikzcd}
        \cat_{1}
        \arrow[r,"\funI"]
        \arrow[d,"\fun_{1}"'] &
        \cat_{2}
        \arrow[d,"\fun_{2}"] \\
        \catI_{1}
        \arrow[r,"\funII"'] &
        \catI_{2}
      \end{tikzcd}
    \]
    satisfying the Beck-Chevalley condition: the canonical natural
    transformation
    \[
      \begin{tikzcd}
        \cat_{1}
        \arrow[r,"\funI"]
        \arrow[dr,Rightarrow,start
        anchor={[xshift=1ex,yshift=-1ex]},end
        anchor={[xshift=-1ex,yshift=1ex]}] &
        \cat_{2} \\
        \catI_{1}
        \arrow[u,"\fun_{1}^{*}"]
        \arrow[r,"\funII"'] &
        \catI_{2}
        \arrow[u,"\fun_{2}^{*}"']
      \end{tikzcd}
    \]
    is invertible.
  \end{itemize}
\end{example}

We use \cref{LAdj-cg} to verify that an \(\infty\)-category whose
objects are small \(\infty\)-categories with a certain structure
defined by adjunction is compactly generated.

\begin{example}
  For a finitely presentable simplicial set \(\sh\), we define
  \(\Lex_{\infty}^{(\sh)}\) to be the pullback
  \[
    \begin{tikzcd}
      \Lex_{\infty}^{(\sh)}
      \arrow[r]
      \arrow[d] &
      [10ex]
      \LAdj
      \arrow[d] \\
      \Cat_{\infty}
      \arrow[r,"\cat \mapsto (\diagonal : \cat \to \sh \cotensor
      \cat)"'] &
      \Cat_{\infty}^{\to}.
    \end{tikzcd}
  \]
  \(\Lex_{\infty}^{(\sh)}\) is the \(\infty\)-category of small
  \(\infty\)-categories with limits of shape \(\sh\). We define the
  \(\infty\)-category \(\Lex_{\infty}\) of small left exact
  \(\infty\)-categories to be the wide pullback of
  \(\Lex_{\infty}^{(\sh)}\) over \(\Cat_{\infty}\) for all finitely
  presentable simplicial sets \(\sh\). By construction,
  \(\Lex_{\infty}^{(\sh)}\) and \(\Lex_{\infty}\) are compactly
  generated, and the forgetful functors to \(\Cat_{\infty}\) are
  conservative, \(\omega\)-accessible right adjoints.
\end{example}

We remark that codomain functors are always cartesian fibrations in
\(\kPrR_{\omega}\).

\begin{proposition}
  \label{PrR-codomain-fibration}
  For a compactly generated \(\infty\)-category \(\cat\),
  the functor \(\cod : \cat^{\to} \to \cat\) is a cartesian fibration
  in \(\kPrR_{\omega}\).
\end{proposition}
\begin{proof}
  Recall that finite limits commute with filtered colimits
  in any compactly generated \(\infty\)-category
  \(\cat\). This implies that \(\cat\) is finitely complete in the
  \(\infty\)-cosmos \(\kPrR_{\omega}\) (that is, the diagonal functor
  \(\cat \to \sh \cotensor \cat\) has a right adjoint for every
  finitely presentable simplicial set \(\sh\)). Hence, the codomain
  functor \(\cat^{\to} \to \cat\) is a cartesian fibration.
\end{proof}

\subsection{Exponentiable arrows}
\label{sec:exponentiable-arrows}

\begin{definition}
  An arrow \(\arr : \obj \to \objI\) in a left exact
  \(\infty\)-category \(\cat\) is said to be \emph{exponentiable} if
  the pullback functor \(\arr^{*} : \cat/\objI \to \cat/\obj\) has a
  right adjoint. If this is the case, we refer to the right adjoint of
  \(\arr^{*}\) as the \emph{pushforward along \(\arr\)} and denote it by
  \(\arr_{*} : \cat/\obj \to \cat/\objI\).
\end{definition}

\begin{definition}
  For an exponentiable arrow \(\arr : \obj \to \objI\) in a left exact
  \(\infty\)-category \(\cat\), the \emph{associated polynomial
    functor} \(\poly_{\arr} : \cat \to \cat\) is the composite
  \[
    \begin{tikzcd}
      \cat
      \arrow[r,"\obj^{*}"] &
      \cat/\obj
      \arrow[r,"\arr_{*}"] &
      \cat/\objI
      \arrow[r,"\objI_{!}"] &
      \cat
    \end{tikzcd}
  \]
  where \(\obj^{*}\) is the pullback along \(\obj \to \terminal\) and
  \(\objI_{!}\) is the forgetful functor.
\end{definition}

Recall that polynomials can be \emph{composed}
\parencite{gambino2013polynomial,weber2015polynomials}: given two
exponentiable arrows \(\arr_{1} : \obj_{1} \to \objI_{1}\) and
\(\arr_{2} : \obj_{2} \to \objI_{2}\), we have an exponentiable arrow
\(\arr_{1} \otimes \arr_{2}\) such that
\(\poly_{\arr_{1} \otimes \arr_{2}} \simeq \poly_{\arr_{1}} \circ
\poly_{\arr_{2}}\). We may also concretely define \(\arr_{1} \otimes
\arr_{2}\) as follows: \(\cod (\arr_{1} \otimes \arr_{2}) =
\poly_{\arr_{1}} \objI_{2}\); \(\dom (\arr_{1} \otimes \arr_{2})\) is
the pullback
\[
  \begin{tikzcd}
    \dom (\arr_{1} \otimes \arr_{2})
    \arrow[r]
    \arrow[d] &
    \obj_{2}
    \arrow[d, "\arr_{2}"] \\
    \poly_{\arr_{1}} \objI_{2} \times_{\objI_{1}} \obj_{1}
    \arrow[r, "\ev"'] &
    \objI_{2};
  \end{tikzcd}
\]
\(\arr_{1} \otimes \arr_{2}\) is the composite
\(\dom (\arr_{1} \otimes \arr_{2}) \to \poly_{\arr_{1}} \objI_{2}
\times_{\objI_{1}} \obj_{1} \to \poly_{\arr_{1}} \objI_{2} = \cod
(\arr_{1} \otimes \arr_{2})\).

\subsection{Representable maps of right fibrations}
\label{sec:repr-maps-right}

We review the notion of a \emph{representable map} of right
fibrations, which is a generalization of a representable map of
discrete fibrations over a \(1\)-category. We think of a representable
map of right fibrations as an \(\infty\)-categorical analogue of a
natural model of type theory \parencite{awodey2018natural} and a
category with families \parencite{dybjer1996internal}.

\begin{definition}
  We say a map \(\map : \sh \to \shI\) of right fibrations over an
  \(\infty\)-category \(\cat\) is \emph{representable} if it has a
  right adjoint.
\end{definition}

\begin{proposition}
  \label{rfib-slice}
  Let \(\proj : \sh \to \cat\) be a right fibration between
  \(\infty\)-categories. A functor \(\map : \shI \to \sh\) is a right
  fibration if and only if the composite \(\proj\map : \shI \to \cat\)
  is. Consequently, we have a canonical equivalence of
  \(\infty\)-categories
  \[
    \RFib_{\cat}/\sh \simeq \RFib_{\sh}.
  \]
\end{proposition}
\begin{proof}
  By definition.
\end{proof}

\begin{corollary}
  \label{rep-right-adjoint}
  A representable map \(\map : \sh \to \shI\) of right fibrations over
  an \(\infty\)-category \(\cat\) is exponentiable, and the
  pushforward along \(\map\) is given by the pullback along the right
  adjoint \(\rarep : \shI \to \sh\) of \(\map\).
  \[
    \begin{tikzcd}
      \RFib_{\cat}/\sh \simeq \RFib_{\sh}
      \arrow[rr,bend right,"\rarep^{*}"',start anchor=south east,end
      anchor=south west] &
      \rotatebox[origin=c]{270}{\(\adj\)} &
      \RFib_{\shI} \simeq \RFib_{\cat}/\shI
      \arrow[ll,bend right,"\map^{*}"',start anchor=north west,end
      anchor=north east]
    \end{tikzcd}
  \]
\end{corollary}

\begin{corollary}
  \label{rep-pullback}
  Representable maps of right fibrations over an \(\infty\)-category
  \(\cat\) are stable under pullbacks: if
  \[
    \begin{tikzcd}
      \sh_{1}
      \arrow[r,"\mapI"]
      \arrow[d,"\map_{1}"'] &
      \sh_{2}
      \arrow[d,"\map_{2}"] \\
      \shI_{1}
      \arrow[r,"\mapII"'] &
      \shI_{2}
    \end{tikzcd}
  \]
  is a pullback in \(\RFib_{\cat}\) and \(\map_{2}\) is representable,
  then \(\map_{1}\) is representable. Moreover, if this is the case,
  the square satisfies the Beck-Chevalley condition.
\end{corollary}
\begin{proof}
  By \cref{rfib-slice}, the functor \(\mapII\) is a right
  fibration. Thus, the right adjoint of \(\map_{2}\) lifts to a fibred
  right adjoint of \(\map_{1}\).
  \[
    \begin{tikzcd}
      \sh_{1}
      \arrow[rr,bend left,"\map_{1}"]
      \arrow[d,"\mapI"'] &
      \rotatebox[origin=c]{270}{\(\adj\)} &
      \shI_{1}
      \arrow[ll,dotted,bend left]
      \arrow[d,"\mapII"] \\
      [6ex]
      \sh_{2}
      \arrow[rr,bend left,"\map_{2}"] &
      \rotatebox[origin=c]{270}{\(\adj\)} &
      \shI_{2}
      \arrow[ll,bend left]
    \end{tikzcd}
  \]
\end{proof}

\begin{proposition}
  \label{representable-map-1}
  A map \(\map : \sh \to \shI\) of right fibrations over an
  \(\infty\)-category \(\cat\) is representable if and only if, for
  any section \(\elI : \cat/\objI \to \shI\), the pullback
  \(\elI^{*}\sh\) is a representable right fibration over \(\cat\).
\end{proposition}
\begin{proof}
  For a section \(\elI : \cat/\objI \to \shI\), an arrow
  \(\arr : \map\el \to \elI\) in \(\shI\) for some \(\el \in \sh\)
  corresponds to a square
  \begin{equation}
    \begin{tikzcd}
      \cat/\obj
      \arrow[r,"\el"]
      \arrow[d,"\arr"'] &
      \sh
      \arrow[d,"\map"] \\
      \cat/\objI
      \arrow[r,"\elI"'] &
      \shI.
    \end{tikzcd}
    \label{eq:4}
  \end{equation}
  \((\el, \arr)\) is a universal arrow from \(\map\) to \(\elI\) if
  and only if \cref{eq:4} is a pullback.
\end{proof}

\section{\(\infty\)-type theories}
\label{sec:infty-type-theories}

We introduce notions of an \emph{\(\infty\)-type theory}, a
\emph{theory} over an \(\infty\)-type theory and a \emph{model} of an
\(\infty\)-type theory, translating the previous work of the second
author \parencite{uemura2019framework} into the language of
\(\infty\)-categories. The idea is to extend the functorial semantics
of algebraic theories \parencite{lawvere2004functorial}. Algebraic
theories are identified with categories with finite products, and
models of an algebraic theory are identified with functors into the
category of sets preserving finite products. For type theories, it is
natural to identify models of a type theory with functors into
presheaf categories, because (extensions of) natural models
\parencite{awodey2018natural} and categories with families
\parencite{dybjer1996internal} are diagrams in presheaf
categories. Since representable maps of presheaves play a special role
in the natural model semantics, some arrows in the source category
should be specified to be sent to representable maps. This motivates
the following definitions.

\begin{definition}
  An \emph{\(\infty\)-category with representable maps} is a pair
  \((\cat,\reps)\) where \(\cat\) is an \(\infty\)-category and
  \(\reps \subseteq k(\Delta^1,\cat)\) is a subspace of the space of arrows
  of \(\cat\) satisfying the conditions below. Arrows in \(\reps\) are
  called \emph{representable arrows}.
  \begin{enumerate}
  \item \(\cat\) has finite limits.
  \item All the identities are representable and representable arrows
    are closed under composition.
  \item Representable arrows are stable under pullbacks.
  \item Representable arrows are exponentiable.
  \end{enumerate}
  A \emph{morphism of \(\infty\)-categories with representable maps}
  is a functor preserving representable arrows, finite limits and
  pushforwards along representable arrows.
\end{definition}

\begin{example}
  For a small \(\infty\)-category \(\cat\), the \(\infty\)-category
  \(\RFib_{\cat}\) of small right fibrations over \(\cat\) is an
  \(\infty\)-category with representable maps in which a map is
  representable if it has a right adjoint.
\end{example}

\begin{definition}
  An \emph{\(\infty\)-type theory} is an \(\infty\)-category with
  representable maps whose underlying \(\infty\)-category is small. A
  \emph{morphism of \(\infty\)-type theories} is a morphism of
  \(\infty\)-categories with representable maps.
  By an \emph{\(\nat\)-type theory} for \(1 \le \nat < \infty\), we mean an
  \(\infty\)-type theory whose underlying \(\infty\)-category is an
  \(\nat\)-category.
\end{definition}

\begin{example}
  The type theories in the sense of the previous work
  \parencite{uemura2019framework} are the \(1\)-type theories.
\end{example}

\begin{definition}
  Let \(\tth\) be an \(\infty\)-type theory.
  \begin{itemize}
  \item A \emph{model of \(\tth\)} consists of an \(\infty\)-category
    \(\model(\bas)\) with a terminal object and a morphism of
    \(\infty\)-categories with representable maps
    \(\model : \tth \to \RFib_{\model(\bas)}\).
  \item A \emph{theory over \(\tth\)} or a \emph{\(\tth\)-theory} is a
    left exact functor \(\theory : \tth \to \Space\).
  \end{itemize}
\end{definition}

\begin{example}
  \label{natural-model}
  We will construct in \cref{sec:infty-category-infty} a presentable
  \(\infty\)-category \(\TTh_{\infty}\) of \(\infty\)-type theories
  and their morphisms, so we have various \emph{free constructions} of
  \(\infty\)-type theories. For example, there is an \(\infty\)-type
  theory \(\tthG_{\infty}\) freely generated by one representable
  arrow \(\typeof : \El \to \Ty\). Indeed, the functor
  \(\TTh_{\infty} \to \Space\) that sends an \(\infty\)-type theory
  \(\tth\) to the space of representable arrows in \(\tth\) preserves
  limits and filtered colimits, and thus it is representable by
  presentability. The universal property of \(\tthG_{\infty}\) asserts
  that a morphism \(\tthG_{\infty} \to \cat\) of \(\infty\)-categories
  with representable maps is completely determined by the image of the
  representable arrow \(\typeof \in \tthG_{\infty}\). Thus, a model of
  \(\tthG_{\infty}\) consists of the following data:
  \begin{itemize}
  \item an \(\infty\)-category \(\model(\bas)\) with a terminal
    objects;
  \item a representable map
    \(\model(\typeof) : \model(\El) \to \model(\Ty)\) of right
    fibrations over \(\model(\bas)\).
  \end{itemize}
  In other words, a model of \(\tthG_{\infty}\) is an
  \(\infty\)-categorical analogue of a natural model
  \parencite{awodey2018natural,fiore2012discrete}. One may think of an object
  \(\ctx \in \model(\bas)\) as a context, a section
  \(\sh : \model(\bas)/\ctx \to \model(\Ty)\) as a type over \(\ctx\),
  and a section \(\el : \model(\bas)/\ctx \to \model(\El)\) as a term
  over \(\ctx\). The representability of \(\model(\typeof)\) is used
  for modeling context comprehension: for a section
  \(\sh : \model(\bas)/\ctx \to \model(\Ty)\), the representing object
  for \(\sh^{*}\model(\El)\) is though of as the context
  \((\ctx, x : \sh)\) with \(x\) a fresh variable.

  It is not simple to describe a \(\tthG_{\infty}\)-theory, but we
  could say that the \(\infty\)-category of
  \(\tthG_{\infty}\)-theories is an \(\infty\)-analogue of the
  category of generalized algebraic theories
  \parencite{cartmell1978generalised}. Indeed, the second named author
  showed in \parencite{uemura2022universal} that the category of
  generalized algebraic theories is equivalent to the category of left
  exact functors \(\tthG \to \Set\) where \(\tthG\) is the left exact
  category freely generated by an exponentiable arrow.
\end{example}

In
\cref{sec:infty-category-infty,sec:infty-categ-theor,sec:infty-categ-models}
below, we will construct an \(\infty\)-category \(\TTh_{\infty}\) of
\(\infty\)-type theories, an \(\infty\)-category \(\Th(\tth)\) of
\(\tth\)-theories and an \(\infty\)-category \(\Mod(\tth)\) of models
of \(\tth\). These \(\infty\)-categories are constructed inside the
\(\infty\)-cosmos \(\kPrR_{\omega}\) of compactly generated
\(\infty\)-categories and \(\omega\)-accessible right adjoints. In \cref{sec:univ-prop-modtth}
we give a universal property of \(\Mod(\tth)\) as an object of
\(\CAT_{\infty}/\Lex_{\infty}^{(\emptyset)}\) from which for example
it follows that the assignment \(\tth \mapsto \Mod(\tth)\) takes
colimits to limits. In \cref{sec:slice-infty-type} we see that a slice
of the underlying \(\infty\)-category of an \(\infty\)-type theory is
naturally equipped with a structure of \(\infty\)-type theory and has
a useful universal property.

\subsection{The \(\infty\)-category of \(\infty\)-type theories}
\label{sec:infty-category-infty}

We construct an \(\infty\)-category \(\TTh_{\infty}\) of \(\infty\)-type
theories and their morphisms.

\begin{definition}
  Let \(\Cat^{+}_{\infty}\) be the pullback
  \[
    \begin{tikzcd}
      \Cat^{+}_{\infty}
      \arrow[r]
      \arrow[d] &
      (\Space^\to)_{\le -1}
      \arrow[d,"\cod"] \\
      \Cat_{\infty}
      \arrow[r,"{\lgpd(\Delta^{1}, {-})}"'] &
      \Space
    \end{tikzcd}
  \]
  where \((\Space^\to)_{\le -1}\) denotes the full subcategory of
  \(\Space^{\to}\) spanned by the \((-1)\)-truncated maps of spaces
  which is an \(\omega\)-accessible localization of
  \(\Space^{\to}\). \(\Cat^{+}_{\infty}\) is the \(\infty\)-category
  of small \(\infty\)-categories equipped with a subspace of arrows. We
  define \(\Lex^{+}_{\infty}\) to be the full subcategory of
  \(\Lex_{\infty} \times_{\Cat_{\infty}} \Cat^{+}_{\infty}\) spanned
  by the left exact \(\infty\)-categories with a class of arrows
  closed under composition and stable under pullbacks.
\end{definition}

The inclusion
\(\Lex^{+}_{\infty} \to \Lex_{\infty} \times_{\Cat_{\infty}}
\Cat^{+}_{\infty}\) has a left adjoint by taking the closure of the specified subspace of arrows under
composition and pullbacks, and \(\Lex^{+}_{\infty}\) is closed in
\(\Lex_{\infty} \times_{\Cat_{\infty}} \Cat^{+}_{\infty}\) under
filtered colimits. Hence, \(\Lex^{+}_{\infty}\) is compactly
generated, and the inclusion
\(\Lex^{+}_{\infty} \to \Lex_{\infty} \times_{\Cat_{\infty}}
\Cat_{\infty}^{+}\) is an \(\omega\)-accessible right adjoint.

Let \((\cat, \reps)\) be an object of \(\Lex_{\infty}^{+}\). Since
\(\cat\) has finite limits, we have a functor \(\theta(\cat, \reps)\)
between isofibrations over \(\reps\) whose fiber over
\((\arr : \obj \to \objI) \in \reps\) is the pullback functor
\(\arr^{*} : \cat/\objI \to \cat/\obj\). An \(\infty\)-type theory is
nothing but an object \((\cat, \reps)\) of \(\Lex_{\infty}^{+}\) such
that \(\theta(\cat, \reps)\) has a fiberwise right adjoint. We show
that this condition is equivalent to the condition that the functor
has a right adjoint.

\begin{proposition}
  \label{adjoint-over-groupoid}
  Let
  \[
    \begin{tikzcd}
      \cat
      \arrow[rr,"\fun"]
      \arrow[dr] & &
      \catI
      \arrow[dl] \\
      & \sh
    \end{tikzcd}
  \]
  be a functor between isofibrations in \(\kCAT_{\infty}\) such that
  \(\sh\) is an \(\infty\)-groupoid. The following are equivalent:
  \begin{enumerate}
  \item \label{item:8} the functor \(\fun : \cat \to \catI\) has a
    right adjoint;
  \item \label{item:9} for every point \(\el \in \sh\), the functor
    between fibers \(\fun_{\el} : \cat_{\el} \to \catI_{\el}\) has a
    right adjoint.
  \end{enumerate}
\end{proposition}
\begin{proof}
  Suppose that each \(\fun_{\el} : \cat_{\el} \to \catI_{\el}\) has a
  right adjoint \(\funI_{\el}\) with counit
  \(\counit_{\el, \objI} : \fun_{\el}(\funI_{\el}(\objI)) \to
  \objI\). It suffices to see that \(\counit_{\el, \objI}\) is
  universal in \(\catI\). Let \(\obj \in \cat_{\el'}\) be an object in
  another fiber and consider the induced map
  \begin{equation*}
    \cat(\obj, \funI_{\el}(\objI)) \to \catI(\fun_{\el'}(\obj), \objI).
  \end{equation*}
  This is a map over \(\sh(\el', \el)\), and thus it suffices to show
  that this is fiberwise an equivalence. Since \(\sh\) is an
  \(\infty\)-groupoid and since \(\cat \to \sh\) and \(\catI \to \sh\)
  are isofibrations, the fibers over \(p \in \sh(\el', \el)\) are
  equivalent to the fibers over \(\id \in \sh(\el, \el)\), but the map
  between the fibers over \(\id\) is the equivalence
  \(\cat_{\el}(\obj, \funI_{\el}(\objI)) \simeq
  \catI_{\el}(\funI_{\el}(\obj), \objI)\).

  Suppose that \(\fun\) has a right adjoint \(\funI : \catI \to \cat\)
  with counit \(\counit : \fun\funI \To \id\). Since \(\sh\) is an
  \(\infty\)-groupoid, the natural transformation
  \[
    \begin{tikzcd}
      \catI
      \arrow[r,"\funI"]
      \arrow[dr,equal,""{name=a0}] &
      \cat
      \arrow[dr]
      \arrow[d,"\fun"]
      \arrow[to=a0,Rightarrow,"\counit"]\\
      & \catI
      \arrow[r] &
      \sh
    \end{tikzcd}
  \]
  is invertible. Then, since \(\cat \to \sh\) and \(\catI \to \sh\)
  are isofibrations, one can replace \(\funI\) and \(\counit\) by a
  functor \(\funI' : \catI \to \cat\) and a natural transformation
  \(\counit' : \fun \funI' \To \id\), respectively, over \(\sh\). Then
  \(\funI'\) and \(\counit'\) give a fiberwise right adjoint of
  \(\fun\).
\end{proof}

\begin{remark}
The proposition also holds more generally when \(\sh\) is an \(\infty\)-category. See \parencite[Proposition 7.3.2.1]{lurie2017algebra}.
\end{remark}

The functor \(\theta(\cat, \reps)\) is constructed as follows. Since
\(\cat\) has finite limits, the functor
\((\Delta^{1} \times \Delta^{1}) \cotensor \cat \to \Lambda^{2}_{2}
\cotensor \cat\) sending a square to its bottom and right edges has a
right adjoint. Composing the right adjoint and the functor
\((\Delta^{1} \times \Delta^{1}) \cotensor \cat \to \Lambda^{2}_{1}
\cotensor \cat\) sending a square to its bottom and left edges, we
have a functor
\begin{equation*}
  \theta' : \Lambda^{2}_{2} \cotensor \cat \to \Lambda^{2}_{1} \cotensor \cat
\end{equation*}
over \(\Delta^{\{1, 2\}} \cotensor \cat\). The functor
\(\theta(\cat, \reps)\) is then the pullback of \(\theta'\) along the
inclusion \(\reps \to \Delta^{\{1, 2\}} \cotensor \cat\). This
construction is functorial and preserves limits and filtered colimits,
yielding a functor
\(\theta : \Lex_{\infty}^{+} \to \Delta^{2} \cotensor \Cat_{\infty}\)
in \(\kPrR_{\omega}\).

\begin{definition}
  We define \(\TTh_{\infty}\) to be the pullback
  \[
    \begin{tikzcd}
      \TTh_{\infty}
      \arrow[rr]
      \arrow[d] & &
      \LAdj
      \arrow[d] \\
      \Lex_{\infty}^{+}
      \arrow[r,"\theta"'] &
      \Delta^{2} \cotensor \Cat_{\infty}
      \arrow[r] &
      \Delta^{\{0, 1\}} \cotensor \Cat_{\infty}.
    \end{tikzcd}
  \]
  By \cref{adjoint-over-groupoid}, the objects of \(\TTh_{\infty}\)
  are precisely the \(\infty\)-type theories. It is also
  straightforward to see that the morphisms of \(\TTh_{\infty}\)
  are precisely the morphisms of \(\infty\)-type theories.
\end{definition}

\subsection{The \(\infty\)-category of theories over an
  \(\infty\)-type theory}
\label{sec:infty-categ-theor}

\begin{definition}
  For an \(\infty\)-type theory \(\tth\), we define \(\Th(\tth)\) to
  be the full subcategory of \(\Fun(\tth, \Space)\) spanned by the
  functors preserving finite limits.
\end{definition}

By definition, \(\Th(\tth)\) is
  compactly generated, and the inclusion
  \(\Th(\tth) \to \Fun(\tth, \Space)\) is an \(\omega\)-accessible
  right adjoint. The \(\infty\)-category \(\Th(\tth)\) has the following alternative definitions:
\begin{itemize}
\item \(\Th(\tth)\) is the cocompletion of \(\tth^{\op}\) under
  filtered colimits;
\item \(\Th(\tth)\) is the \(\omega\)-free cocompletion of
  \(\tth^{\op}\), that is, the initial cocomplete \(\infty\)-category
  equipped with a functor from \(\tth^{\op}\) preserving finite
  colimits.
\end{itemize}

\subsection{The \(\infty\)-category of models of an \(\infty\)-type
  theory}
\label{sec:infty-categ-models}

We construct an \(\infty\)-category \(\Mod(\tth)\) of models of an
\(\infty\)-type theory \(\tth\). The following description of
\(\Mod(\tth)\) is based on unpublished work by John Bourke and the second named author on the
\(2\)-category of \(1\)-models of a \(1\)-type theory.

Let \(\tth\) be an \(\infty\)-type theory. Recall that a functor to a
slice \(\infty\)-category \(\fun' : \cat \to \catI/\objI\) corresponds
to a functor \(\fun : \cat^{\rcone} \to \catI\) that sends
\(\bas \in \cat^{\rcone}\) to \(\objI\). Then a model \(\model\) of
\(\tth\) can be regarded as a functor
\(\model : \tth^{\rcone} \to \Cat_{\infty}\) satisfying the following
conditions:
\begin{enumerate}
\item \label{item:3} \(\model(\bas)\) has a terminal object;
\item \label{item:4} for every object \(\obj \in \tth\), the functor
  \(\model(\obj) \to \model(\bas)\) is a right fibration;
\item \label{item:5} for every finite diagram \(\obj : \sh \to \tth\),
  the canonical functor
  \(\model(\lim_{\sh}\obj) \to
  \lim_{\sh^{\rcone}}\model\obj^{\rcone}\) is an equivalence;
\item \label{item:6} for every representable arrow
  \(\arr : \obj \to \objI\) in \(\tth\), the functor
  \(\model(\arr) : \model(\obj) \to \model(\objI)\) has a right
  adjoint \(\rarep_{\arr} : \model(\objI) \to \model(\obj)\);
\item \label{item:7} for every pair of arrows
  \(\arr : \obj \to \objI\) and \(\arrI : \objI \to \objII\) with
  \(\arrI\) representable, the canonical functor
  \(\model(\arrI_{*}\obj) \to \rarep_{\arrI}^{*}\model(\obj)\) is an
  equivalence (recall that the pushforward along \(\model(\arrI)\) in
  \(\RFib_{\model(\bas)}\) is given by the pullback along
  \(\rarep_{\arrI}\)).
\end{enumerate}
From this description, we will define \(\Mod(\tth)\) as a
subcategory of \(\Fun(\tth^{\rcone}, \Cat_{\infty})\).

\begin{definition}
  We define \(\Mod_{\labelcref{item:3}}(\tth)\) to be the pullback
  \[
    \begin{tikzcd}
      \Mod_{\labelcref{item:3}}(\tth)
      \arrow[r]
      \arrow[d] &
      \Lex_{\infty}^{(\emptyset)}
      \arrow[d] \\
      \Fun(\tth^{\rcone}, \Cat_{\infty})
      \arrow[r,"\ev_{\bas}"'] &
      \Cat_{\infty}.
    \end{tikzcd}
  \]
\end{definition}

\begin{definition}
  For an object \(\obj \in \tth\), we define
  \(\Mod_{\labelcref{item:4}}^{\obj}(\tth)\) to be the pullback
  \[
    \begin{tikzcd}
      \Mod_{\labelcref{item:4}}^{\obj}(\tth)
      \arrow[r]
      \arrow[d,hook] &
      [4ex]
      \RFib
      \arrow[d,hook] \\
      \Fun(\tth^{\rcone}, \Cat_{\infty})
      \arrow[r,"\ev_{(\obj \to {*})}"'] &
      \Cat_{\infty}^{\to}
    \end{tikzcd}
  \]
  and \(\Mod_{\labelcref{item:4}}(\tth)\) to be the wide pullback of
  \(\Mod_{\labelcref{item:4}}^{\obj}(\tth)\) over
  \(\Fun(\tth^{\rcone}, \Cat_{\infty})\) for all objects
  \(\obj \in \tth\).
\end{definition}

\begin{definition}
  For a finite diagram \(\obj : \sh \to \tth\), we define
  \(\Mod_{\labelcref{item:5}}^{(\sh, \obj)}(\tth)\) to be the pullback
  \[
    \begin{tikzcd}
      \Mod_{\labelcref{item:5}}^{(\sh, \obj)}(\tth)
      \arrow[r]
      \arrow[d,hook] &
      [12ex]
      \Cat_{\infty}^{\simeq}
      \arrow[d,hook] \\
      \Fun(\tth^{\rcone}, \Cat_{\infty})
      \arrow[r,"(\ev_{\lim_{\sh}\obj} \To
      \lim_{\sh^{\rcone}}\ev_{\obj^{\rcone}})"'] &
      \Cat_{\infty}^{\to}
    \end{tikzcd}
  \]
  and \(\Mod_{\labelcref{item:5}}(\tth)\) to be the wide pullback of
  \(\Mod_{\labelcref{item:5}}^{(\sh, \obj)}(\tth)\) over
  \(\Fun(\tth^{\rcone}, \Cat_{\infty})\) for all finite diagrams
  \((\sh, \obj : \sh \to \tth)\).
\end{definition}

\begin{definition}
  For a representable arrow \(\arr : \obj \to \objI\) in \(\tth\), we
  define \(\Mod_{\labelcref{item:6}}^{\arr}(\tth)\) to be the pullback
  \[
    \begin{tikzcd}
      \Mod_{\labelcref{item:6}}^{\arr}(\tth)
      \arrow[r]
      \arrow[d] &
      \LAdj
      \arrow[d] \\
      \Fun(\tth^{\rcone}, \Cat_{\infty})
      \arrow[r,"\ev_{\arr}"'] &
      \Cat_{\infty}^{\to}
    \end{tikzcd}
  \]
  and \(\Mod_{\labelcref{item:6}}(\tth)\) to be the wide pullback of
  \(\Mod_{\labelcref{item:6}}^{\arr}(\tth)\) over
  \(\Fun(\tth^{\rcone}, \Cat_{\infty})\) for all representable arrows
  \(\arr\) in \(\tth\).
\end{definition}

\begin{definition}
  We denote by \(\Mod_{-\labelcref{item:7}}(\tth)\) the wide pullback
  of \(\Mod_{\labelcref{item:3}}(\tth)\),
  \(\Mod_{\labelcref{item:4}}(\tth)\),
  \(\Mod_{\labelcref{item:5}}(\tth)\) and
  \(\Mod_{\labelcref{item:6}}(\tth)\) over
  \(\Fun(\tth^{\rcone}, \Cat_{\infty})\). By construction,
  \(\Mod_{-\labelcref{item:7}}(\tth)\) is the \(\infty\)-category of
  functors \(\model : \tth^{\rcone} \to \Cat_{\infty}\) satisfying
  \cref{item:3,item:4,item:5,item:6}.
\end{definition}

\begin{definition}
  For a pair of composable arrows \(\arr : \obj \to \objI\) and
  \(\arrI : \objI \to \objII\) in \(\tth\) with \(\arrI\)
  representable, we define
  \(\Mod_{\labelcref{item:7}}^{(\arr, \arrI)}(\tth)\) to be the
  pullback
  \[
    \begin{tikzcd}
      \Mod_{\labelcref{item:7}}^{(\arr, \arrI)}(\tth)
      \arrow[r]
      \arrow[d,hook] &
      [8ex]
      \Cat_{\infty}^{\simeq}
      \arrow[d,hook] \\
      \Mod_{-\labelcref{item:7}}(\tth)
      \arrow[r,"(\ev_{\arrI_{*}\obj} \To
      \rarep_{\arrI}^{*}\ev_{\obj})"'] &
      \Cat_{\infty}^{\to}
    \end{tikzcd}
  \]
  and \(\Mod(\tth)\) to be the wide pullback of
  \(\Mod_{\labelcref{item:7}}^{(\arr, \arrI)}(\tth)\) over
  \(\Mod_{-\labelcref{item:7}}(\tth)\) for all pairs \((\arr, \arrI)\)
  of composable arrows in \(\tth\) with \(\arrI\) representable.
\end{definition}

By
  construction, the \(\infty\)-category \(\Mod(\tth)\) is compactly
  generated, and the forgetful functor
  \(\Mod(\tth) \to \Fun(\tth^{\rcone}, \Cat_{\infty})\) is a
  conservative, \(\omega\)-accessible right adjoint. Moreover,
  the objects of \(\Mod(\tth)\) are the models of \(\tth\) and the
  morphisms in \(\Mod(\tth)\) are described as follows. Let
  \(\model\) and \(\modelI\) be models of \(\tth\) and
  \(\fun : \model \To \modelI : \tth^{\rcone} \to \Cat_{\infty}\) be a
  natural transformation. Then \(\fun\) is in \(\Mod(\tth)\) if and
  only if the following conditions hold:
  \begin{itemize}
  \item the component \(\fun(\bas) : \model(\bas) \to \modelI(\bas)\)
    preserves terminal objects;
  \item for any representable arrow \(\arr : \obj \to \objI\) in
    \(\tth\), the square
    \[
      \begin{tikzcd}
        \model(\obj)
        \arrow[r,"\fun(\obj)"]
        \arrow[d,"\model(\arr)"'] &
        \modelI(\obj)
        \arrow[d,"\modelI(\arr)"] \\
        \model(\objI)
        \arrow[r,"\fun(\objI)"'] &
        \modelI(\objI)
      \end{tikzcd}
    \]
    satisfies the Beck-Chevalley condition.
  \end{itemize}

\subsection{Universal property of \(\Mod(\tth)\)}
\label{sec:univ-prop-modtth}

We give a universal property of \(\Mod(\tth)\) seen as an object of
\(\CAT_{\infty} / \Lex_{\infty}^{(\emptyset)}\). A consequence is that
the assignment \(\tth \mapsto \Mod(\tth)\) takes colimits of
\(\infty\)-type theories to limits of \(\infty\)-categories over
\(\Lex_{\infty}^{(\emptyset)}\) (\cref{presentation-of-Mod}).

\begin{definition}
  For a functor \(\cat : \idxcat \to \Cat_{\infty}\), we denote by
  \(\Fun(\idxcat, \RFib)_{\cat}\) the full subcategory of
  \(\Fun(\idxcat, \Cat_{\infty})/\cat\) spanned by the natural
  transformations \(\proj : \sh \To \cat : \idxcat \to \Cat_{\infty}\)
  whose components are right fibrations. In other words,
  \(\Fun(\idxcat, \RFib)_{\cat}\) is the fiber of the functor
  \(\Fun(\idxcat, \cod) : \Fun(\idxcat, \RFib) \to \Fun(\idxcat,
  \Cat_{\infty})\) over the object
  \(\cat \in \Fun(\idxcat, \Cat_{\infty})\). We say a map
  \(\map : \sh \to \shI\) in \(\Fun(\idxcat, \RFib)_{\cat}\) is
  \emph{representable} if every component
  \(\map(\idx) : \sh(\idx) \to \shI(\idx)\) is a representable map of
  right fibrations over \(\cat(\idx)\) and if every naturality square
  \[
    \begin{tikzcd}
      \sh(\idx)
      \arrow[r,"\sh(\idxarr)"]
      \arrow[d,"\map(\idx)"'] &
      \sh(\idxI)
      \arrow[d,"\map(\idxI)"] \\
      \shI(\idx)
      \arrow[r,"\shI(\idxarr)"'] &
      \shI(\idxI)
    \end{tikzcd}
  \]
  satisfies the Beck-Chevalley condition.
\end{definition}

\begin{proposition}
  \label{diagram-rep-map-pullback}
  Representable maps in \(\Fun(\idxcat, \RFib)_{\cat}\) are closed under
  composition and stable under pullbacks.
\end{proposition}
\begin{proof}
  Let
  \[
    \begin{tikzcd}
      \sh_{1}
      \arrow[r,"\mapI"]
      \arrow[d,"\map_{1}"'] &
      \sh_{2}
      \arrow[d,"\map_{2}"] \\
      \shI_{1}
      \arrow[r,"\mapII"'] &
      \shI_{2}
    \end{tikzcd}
  \]
  be a pullback in \(\Fun(\idxcat, \RFib)_{\cat}\) and suppose that
  \(\map_{2}\) is representable. By \cref{rep-pullback}, every
  \(\map_{1}(\idx) : \sh_{1}(\idx) \to \shI_{1}(\idx)\) is a
  representable map of right fibrations over \(\cat(\idx)\), and the
  square
  \[
    \begin{tikzcd}
      \sh_{1}(\idx)
      \arrow[r,"\mapI(\idx)"]
      \arrow[d,"\map_{1}(\idx)"'] &
      \sh_{2}(\idx)
      \arrow[d,"\map_{2}(\idx)"] \\
      \shI_{1}(\idx)
      \arrow[r,"\mapII(\idx)"'] &
      \shI_{2}(\idx)
    \end{tikzcd}
  \]
  satisfies the Beck-Chevalley condition. It remains to show that, for
  any arrow \(\idxarr : \idx \to \idxI\) in \(\idxcat\), the square
  \[
    \begin{tikzcd}
      \sh_{1}(\idx)
      \arrow[r,"\sh_{1}(\idxarr)"]
      \arrow[d,"\map_{1}(\idx)"'] &
      \sh_{1}(\idxI)
      \arrow[d,"\map_{1}(\idxI)"] \\
      \shI_{1}(\idx)
      \arrow[r,"\shI_{1}(\idxarr)"'] &
      \shI_{1}(\idxI)
    \end{tikzcd}
  \]
  satisfies the Beck-Chevalley condition. Since a map of right
  fibrations over a fixed base is conservative, it suffices to show
  that the composite of squares
  \[
    \begin{tikzcd}
      \sh_{1}(\idx)
      \arrow[r,"\sh_{1}(\idxarr)"]
      \arrow[d,"\map_{1}(\idx)"'] &
      \sh_{1}(\idxI)
      \arrow[r,"\mapI(\idxI)"]
      \arrow[d,"\map_{1}(\idxI)"] &
      \sh_{2}(\idxI)
      \arrow[d,"\map_{2}(\idxI)"] \\
      \shI_{1}(\idx)
      \arrow[r,"\shI_{1}(\idxarr)"'] &
      \shI_{1}(\idxI)
      \arrow[r,"\mapII(\idxI)"'] &
      \shI_{2}(\idxI)
    \end{tikzcd}
  \]
  satisfies the Beck-Chevalley condition, but this is true by the
  Beck-Chevalley condition for \(\map_{2}\).
\end{proof}

\begin{proposition}
  \label{diagram-rep-map-pushforward}
  A representable map in \(\Fun(\idxcat, \RFib)_{\cat}\) is
  exponentiable, and the pushforward is given by the pullback along
  the right adjoint.
\end{proposition}
\begin{proof}
  The same as \cref{rep-right-adjoint}.
\end{proof}

\begin{proposition}
  \label{rmcat-of-diagrams}
  For a functor \(\cat : \idxcat \to \Cat_{\infty}\), the
  \(\infty\)-category \(\Fun(\idxcat, \RFib)_{\cat}\) together with
  the class of representable maps is an \(\infty\)-category with
  representable maps. For a functor \(\fun : \idxcat' \to \idxcat\),
  the functor
  \(\fun^{*} : \Fun(\idxcat, \RFib)_{\cat} \to \Fun(\idxcat',
  \RFib)_{\cat\fun}\) defined by the precomposition of \(\fun\) is a
  morphism of \(\infty\)-categories with representable maps.
\end{proposition}
\begin{proof}
  By definition.
\end{proof}

Let \(\tth\) be an \(\infty\)-type theory and
\(\cat : \idxcat \to \Lex_{\infty}^{(\emptyset)}\) a functor. We have
equivalences
\begin{align*}
  & \quad \CAT_{\infty}/\Lex_{\infty}^{(\emptyset)}((\idxcat, \cat),
  (\Fun(\tth^{\rcone}, \Cat_{\infty}), \ev_{\bas})) \\
  \simeq & \qquad \text{\{transposition\}} \\
  & \quad \{\bas\}/\CAT_{\infty}((\tth^{\rcone}, \{\bas\}),
    (\Fun(\idxcat, \Cat_{\infty}), \cat)) \\
  \simeq & \qquad \text{\{adjunction of join and slice\}} \\
  & \quad \CAT_{\infty}(\tth, \Fun(\idxcat, \Cat_{\infty})/\cat).
\end{align*}

\begin{proposition}
  \label{models-ump}
  Let \(\tth\) be an \(\infty\)-type theory and
  \(\cat : \idxcat \to \Lex_{\infty}^{(\emptyset)}\) a functor. A
  functor \(\fun : \idxcat \to \Fun(\tth^{\rcone}, \Cat_{\infty})\)
  over \(\Lex_{\infty}^{(\emptyset)}\) factors through \(\Mod(\tth)\)
  if and only if its transpose
  \(\widetilde{\fun} : \tth \to \Fun(\idxcat, \Cat_{\infty})/\cat \)
  factors through \(\Fun(\idxcat, \RFib)_{\cat}\) and is a morphism of
  \(\infty\)-categories with representable maps. Consequently, we have
  an equivalence
  \[
    \CAT_{\infty}/\Lex_{\infty}^{(\emptyset)}((\idxcat, \cat),
    (\Mod(\tth), \ev_{\bas})) \simeq \RMCAT_{\infty}(\tth,
    \Fun(\idxcat, \RFib)_{\cat}).
  \]
\end{proposition}

\begin{proof}
  Immediate from the definition of models of \(\tth\).
\end{proof}

\begin{corollary}
  \label{presentation-of-Mod}
  \(\Mod : \TTh_{\infty}^{\op} \to
  \CAT_{\infty}/\Lex_{\infty}^{(\emptyset)}\) preserves limits.
\end{corollary}

\subsection{Slice \(\infty\)-type theories}
\label{sec:slice-infty-type}

\begin{definition}
  For an \(\infty\)-category with representable maps \(\cat\) and an
  object \(\obj \in \cat\), we regard the slice \(\cat/\obj\) as an
  \(\infty\)-category with representable maps in which an arrow is
  representable if it is representable in \(\cat\).
\end{definition}

The goal of this subsection is to show that \(\cat/\obj\) is the
\(\infty\)-category with representable maps obtained from \(\cat\) by
freely adjoining a global section of \(\obj\) (\cref{slice-rep}). We
first recall a universal property of a slice of a left exact
\(\infty\)-category.

\begin{proposition}
  \label{lex-slice}
  Let \(\cat\) be a left exact \(\infty\)-category and
  \(\obj \in \cat\) an object. We denote by
  \(\obj^{*} : \cat \to \cat/\obj\) the pullback functor along
  \(\obj \to \terminal\) and
  \(\diagonal_{\obj} : \terminal \to \obj^{*}\obj\) the arrow in
  \(\cat/\obj\) which is the diagonal arrow
  \(\obj \to \obj \times \obj\) in \(\cat\). For a left exact
  \(\infty\)-category \(\catI\) and a left exact functor
  \(\fun : \cat \to \catI\), the map
  \begin{equation}
    \label{eq:5}
    \cat/\LEX_{\infty}(\cat/\obj, \catI) \ni \funI \mapsto
    \funI(\diagonal_{\obj}) \in \catI(\terminal, \fun\obj)
  \end{equation}
  is an equivalence of spaces.
\end{proposition}
\begin{proof}
  An object of \(\cat/\LEX_{\infty}(\cat/\obj, \catI)\) is a left
  exact functor \(\funI : \cat/\obj \to \catI\) equipped with an
  invertible natural transformation
  \(\trans : \funI \circ \obj^{*} \To \fun\). By the adjunction
  \(\obj_{!} \adj \obj^{*}\), such a natural transformation \(\trans\)
  corresponds to a natural transformation
  \(\widetilde{\trans} : \funI \To \fun \circ \obj_{!}\). One can
  check that \(\trans : \funI \circ \obj^{*} \To \fun\) is invertible
  if and only if
  \(\widetilde{\trans} : \funI \To \fun \circ \obj_{!}\) is a
  cartesian natural transformation, that is, any naturality square is
  a pullback. Therefore, the statement is equivalent to that, given a
  global section \(\arr : \terminal \to \fun\obj\), the space of pairs
  \((\funI, \trans)\) consisting of a left exact functor
  \(\funI : \cat/\obj \to \catI\) and a cartesian natural
  transformation \(\trans : \funI \To \fun \circ \obj_{!}\) extending
  \(\arr\) is contractible.

  Since \(\catI\) has finite limits, the evaluation at the terminal
  object of \(\cat/\obj\) defines a cartesian fibration
  \(\Fun(\cat/\obj, \catI) \to \catI\) in which a natural
  transformation \(\trans : \funI_{1} \To \funI_{2}\) is cartesian if
  and only if the naturality square
  \[
    \begin{tikzcd}
      \funI_{1}\objI
      \arrow[r, "\trans"]
      \arrow[d] &
      \funI_{2}\objI
      \arrow[d] \\
      \funI_{1}\terminal
      \arrow[r, "\trans"'] &
      \funI_{2}\terminal
    \end{tikzcd}
  \]
  is a pullback for any object \(\objI \in \cat/\obj\), which
  is equivalent to that \(\trans\) is a cartesian natural
  transformation. Therefore, given a functor
  \(\funI_{2} : \cat/\obj \to \catI\) and an arrow
  \(\arr : \obj' \to \funI_{2}\terminal\), the space of pairs
  \((\funI_{1}, \trans)\) consisting of a functor
  \(\funI_{1} : \cat/\obj \to \catI\) and a cartesian natural
  transformation \(\trans\) extending \(\arr\) is contractible. When
  \(\funI_{2} = \fun \circ \obj_{!}\), the functor \(\funI_{1}\) must
  preserve pullbacks since \(\funI_{2}\) does. If, in addition,
  \(\obj' \simeq \terminal\), then \(\funI_{1}\) preserves terminal
  objects, and thus it is left exact. We conclude that, given an arrow
  \(\arr : \terminal \to \fun\obj \simeq \fun(\obj_{!}\terminal)\),
  the space of pairs \((\funI, \trans)\) consisting of a left exact
  functor \(\funI : \cat/\obj \to \catI\) and a cartesian natural
  transformation \(\trans : \funI \To \fun \circ \obj_{!}\) extending
  \(\arr\) is contractible, as we have a unique cartesian lift
  \(\funI \To \fun \circ \obj_{!}\) and \(\funI\) must be left exact.
\end{proof}

From this proof, we can extract the inverse of the map
\labelcref{eq:5}: it is given by
\[
  \catI(\terminal, \fun\obj) \ni \arr \mapsto \arr^{*} \circ
  \fun/\obj \in \cat/\LEX_{\infty}(\cat/\obj, \catI)
\]
where \(\fun/\obj : \cat/\obj \to \catI/\fun\obj\) is the functor
induced by \(\fun\).

\begin{definition}
  We denote by \(\RMCAT_{\infty}\) the \(\infty\)-category of large
  \(\infty\)-categories with representable maps and their morphisms.
\end{definition}

\begin{proposition}
  \label{slice-rep}
  Let \(\cat\) be an \(\infty\)-category with representable maps and
  \(\obj \in \cat\) an object.
  \begin{enumerate}
  \item \label{item:-1} The functor \(\obj^{*} : \cat \to \cat/\obj\)
    is a morphism of \(\infty\)-categories with representable maps.
  \item \label{item:-2} For an \(\infty\)-category with representable
    maps \(\catI\) and a morphism \(\fun : \cat \to \catI\), the map
    \[
      \cat/\RMCAT_{\infty}(\cat/\obj, \catI) \ni \funI \mapsto
      \funI(\diagonal_{\obj}) \in \catI(\terminal, \fun\obj)
    \]
    is an equivalence of spaces.
  \end{enumerate}
\end{proposition}
\begin{proof}
  \Cref{item:-1} is because pullbacks preserve all limits and all pushforwards. \Cref{item:-2} follows from
  \cref{lex-slice}, because
  \(\arr^{*} \circ \fun/\obj : \cat/\obj \to \catI\) is a morphism of
  \(\infty\)-categories with representable maps for every global
  section \(\arr : \terminal \to \fun\obj\).
\end{proof}

\Cref{slice-rep} can be reformulated as follows. Let \(\free{\Box}\)
be the free \(\infty\)-type theory generated by one object \(\Box\)
and \(\free{\widetilde{\Box} : \terminal \to \Box}\) the free
\(\infty\)-type theory generated by one object \(\Box\) and one global
section \(\widetilde{\Box}\) of \(\Box\). By definition, a morphism
\(\free{\Box} \to \cat\) corresponds to an object of \(\cat\), and a
morphism \(\free{\widetilde{\Box} : \terminal \to \Box} \to \cat\)
corresponds to a pair \((\obj, \arr)\) consisting of an object
\(\obj\) of \(\cat\) and a global section
\(\arr : \terminal \to \obj\). Then, for an \(\infty\)-category with
representable maps \(\cat\) and an object \(\obj \in \cat\), we can
form a square
\begin{equation}
  \label{eq:-1}
  \begin{tikzcd}
    \free{\Box}
    \arrow[r,"\obj"]
    \arrow[d] &
    \cat
    \arrow[d,"\obj^{*}"] \\
    \free{\widetilde{\Box} : \terminal \to \Box}
    \arrow[r,"\diagonal_{\obj}"'] &
    \cat/\obj.
  \end{tikzcd}
\end{equation}
\Cref{slice-rep} is equivalent to that, for any \(\infty\)-category
with representable maps, the diagram
\[
  \begin{tikzcd}
    \RMCAT_{\infty}(\cat/\obj, \catI)
    \arrow[r]
    \arrow[d] &
    \RMCAT_{\infty}(\free{\widetilde{\Box} : \terminal \to \Box},
    \catI)
    \arrow[d] \\
    \RMCAT_{\infty}(\cat, \catI)
    \arrow[r] &
    \RMCAT_{\infty}(\free{\Box}, \catI)
  \end{tikzcd}
\]
induced by \cref{eq:-1} is a pullback of spaces. In other words:

\begin{proposition}
  \label{slice-rep-2}
  \label{tth-slice}
  For an \(\infty\)-category with representable maps \(\cat\) and an
  object \(\obj \in \cat\), \cref{eq:-1} is a pushout in
  \(\RMCAT_{\infty}\).
\end{proposition}

Using \cref{models-ump} and its corollary, we have the following
description of \(\Mod(\tth / \obj)\) for an \(\infty\)-type theory
\(\tth\) and an object \(\obj \in \tth\). We first observe:
\begin{enumerate}
\item \label{item:14} \(\Mod(\free{\Box}) \simeq \RFib'\) where
  \(\RFib'\) is the base change of \(\RFib\) along the forgetful
  functor \(\Lex_{\infty}^{(\emptyset)} \to \Cat_{\infty}\);
\item \label{item:19}
  \(\Mod(\free{\widetilde{\Box} : \terminal \to \Box}) \simeq
  \RFib'_{\bullet}\) where \(\RFib'_{\bullet}\) is the
  \(\infty\)-category of right fibrations \(\sh \to \cat\) with a
  terminal object in \(\cat\) and a global section
  \(\el : \cat \to \sh\).
\end{enumerate}
These follow from \cref{models-ump}, for example
\begin{align*}
  &\CAT_{\infty} / \Lex_{\infty}^{(\emptyset)}((\idxcat, \cat),
  (\Mod(\free{\Box}), \ev_{\bas})) \\
  \simeq{}& \RMCAT_{\infty}(\free{\Box}, \Fun(\idxcat, \RFib)_{\cat}) \\
  \simeq{}& (\Fun(\idxcat, \RFib)_{\cat})^{\simeq} \\
  \simeq{}& \CAT_{\infty} / \Lex_{\infty}^{(\emptyset)}((\idxcat, \cat), (\RFib', \cod))
\end{align*}
for any \(\cat : \idxcat \to \Lex_{\infty}^{(\emptyset)}\). Then, by
\cref{presentation-of-Mod,tth-slice}, we have:

\begin{proposition}
  \label{slice-model}
  For any \(\infty\)-type theory \(\tth\) and any object \(\obj \in
  \tth\), we have a pullback
  \[
    \begin{tikzcd}
      \Mod(\tth/\obj)
      \arrow[rr]
      \arrow[d] & &
      \RFib'_{\bullet}
      \arrow[d] \\
      \Mod(\tth)
      \arrow[r,"\obj^{*}"'] &
      \Mod(\free{\Box})
      \arrow[r, "\simeq"'] &
      \RFib'.
    \end{tikzcd}
  \]
\end{proposition}

\section{The theory-model correspondence}
\label{sec:theory-model-corr}

Given an \(\infty\)-type theory \(\tth\), we establish an adjunction
between the \(\infty\)-category of \(\tth\)-theories and the
\(\infty\)-category of models of \(\tth\). The right adjoint assigns
an \emph{internal language} to each model of \(\tth\), and the left
adjoint assigns a \emph{syntactic model} to each \(\tth\)-theory. Not
all models of \(\tth\) are syntactic ones. We give a characterization
of syntactic models.

All the results in this section are \(\infty\)-categorical analogues
of results from the previous work of the second author
\parencite{uemura2019framework}, but proofs are simplified and
improved.
\begin{itemize}
\item In the previous work \((2,1)\)-categorical (co)limits are
  distinguished from \(1\)-categorical (co)limits, but there is no
  such difference in the \(\infty\)-categorical setting.
\item In the previous work the left adjoint of the internal language
  functor is made by hand, but in this work we construct the internal
  language functor inside the \(\infty\)-cosmos \(\kPrR_{\omega}\), so
  it has a left adjoint by definition. Therefore, all we have to do is
  to analyze the unit and counit of the adjunction.
\end{itemize}

Let \(\tth\) be an \(\infty\)-type theory. Since the base
  \(\infty\)-category \(\model(\bas)\) of a model \(\model\) of
  \(\tth\) has a terminal object
  \(\terminal : \Delta^{0} \to \model(\bas)\), we have a natural
  transformation
  \[
    \begin{tikzcd}
      \Mod(\tth)
      \arrow[r]
      \arrow[d] &
      \Fun(\tth^{\rcone}, \Cat_{\infty})
      \arrow[d,"\ev_{\bas}"] \\
      \Delta^{0}
      \arrow[r,"\Delta^{0}"']
      \arrow[ur,Rightarrow,"\terminal",start
      anchor={[xshift=1ex,yshift=1ex]},end
      anchor={[xshift=-1ex,yshift=-1ex]}] &
      \Cat_{\infty}.
    \end{tikzcd}
  \]
  Since \(\Fun(\tth^{\rcone}, \Cat_{\infty})\) is the pullback
  \[
    \begin{tikzcd}
      \Fun(\tth^{\rcone}, \Cat_{\infty})
      \arrow[r]
      \arrow[d,"\ev_{\bas}"'] &
      \Fun(\tth, \Cat_{\infty})^{\to}
      \arrow[d,"\cod"] \\
      \Cat_{\infty}
      \arrow[r,"\diagonal"'] &
      \Fun(\tth, \Cat_{\infty}),
    \end{tikzcd}
  \]
  the functor
  \(\ev_{\bas} : \Fun(\tth^{\rcone}, \Cat_{\infty}) \to
  \Cat_{\infty}\) is a cartesian fibration in
  \(\kPrR_{\omega}\). Thus, the natural transformation \(\terminal\)
  induces an \(\omega\)-accessible right adjoint
  \(\terminal^{*} : \Mod(\tth) \to (\Delta^{0})^{*}\Fun(\tth^{\rcone},
  \Cat_{\infty}) \simeq \Fun(\tth, \Cat_{\infty})\). By the definition
  of a model of \(\tth\), the functor
  \(\terminal^{*} : \Mod(\tth) \to \Fun(\tth, \Cat_{\infty})\) factors
  through
  \(\Th(\tth) = \Lex(\tth, \Space) \subset \Fun(\tth,
  \Cat_{\infty})\). We denote this functor
  \(\Mod(\tth) \to \Th(\tth)\) by \(\IL\). By definition,
  \(\IL(\model)\) is the composite
  \[
    \begin{tikzcd}
      \tth
      \arrow[r,"\model"] &
      \RFib_{\model(\bas)}
      \arrow[r,"\terminal^{*}"] &
      \Space
    \end{tikzcd}
  \]
  where \(\terminal^{*}\sh\) is the fiber over
  \(\terminal \in \model(\bas)\) for a right fibration \(\sh\) over
  \(\model(\bas)\). As the functor \(\IL : \Mod(\tth) \to \Th(\tth)\)
  lies in \(\kPrR_{\omega}\), it has a left adjoint
  \(\SM : \Th(\tth) \to \Mod(\tth)\).

\begin{definition}
For a model \(\model\) of
  \(\tth\), the \(\tth\)-theory \(\IL(\model)\) is called the
  \emph{internal language of \(\model\)}.
For a theory \(K\) over \(\tth\), we call \(\SM(\theory)\) the
  \emph{syntactic model} generated by a \(\tth\)-theory \(\theory\).
\end{definition}

In this section, we prove the following:
\begin{enumerate}
\item \label{item:1} the unit of the adjunction \(\SM \adj \IL\) is
  invertible, so the functor \(\SM : \Th(\tth) \to \Mod(\tth)\) is
  fully faithful;
\item \label{item:2} the essential image of
  \(\SM : \Th(\tth) \to \Mod(\tth)\) is the class of democratic models
  of \(\tth\) defined below.
\end{enumerate}
Consequently, the adjunction \(\SM \adj \IL\) induces an equivalence
between \(\Th(\tth)\) and the full subcategory of \(\Mod(\tth)\)
spanned by the democratic models of \(\tth\). We define the notion of
a democratic model in \cref{sec:democratic-models}. The components of
the unit \(\unit : \id \To \IL\SM\) are completely determined by the
components at the representable \(\tth\)-theories
\({\tth}(\obj, {-})\), because \(\Th(\tth)\) is the cocompletion of
\(\tth^{\op}\) under filtered colimits and the right adjoint
\(\IL : \Mod(\tth) \to \Th(\tth)\) preserves filtered colimits. We
thus study in details the syntactic model generated by a representable
\(\tth\)-theory. In \cref{sec:initial-model} we concretely describe
the initial model of \(\tth\) which is the syntactic model generated
by the initial \(\tth\)-theory \({\tth}(\terminal, {-})\). We then
generalize it in \cref{sec:synt-models-repr} to a description of the
syntactic model generated by an arbitrary representable
\(\tth\)-theory. Finally we prove the main results in
\cref{sec:equiv-theor-democr}.

\subsection{Democratic models}
\label{sec:democratic-models}

For a model \(\model\) of an \(\infty\)-type theory, we think of an
object \(\ctx \in \model(\bas)\) as a context (see
\cref{natural-model}), but contexts from the syntax of type theory
satisfy an additional property: every context is obtained from the
empty context by context comprehension. A model of an \(\infty\)-type
theory satisfying this property is said to be \emph{democratic},
generalizing the notion of a democratic category with families
\parencite{clairambault2014biequivalence}.

\begin{definition}
  Let \(\model\) be a model of \(\tth\), \(\arr : \obj \to \objI\) a
  representable arrow in \(\tth\), \(\ctx \in \model(\bas)\) an object
  and \(\elI : \model(\bas)/\ctx \to \model(\objI)\) a section. Let
  \(\rarep_{\arr} : \model(\objI) \to \model(\obj)\) be the right
  adjoint of \(\model(\arr)\). Then the counit
  \(\ctxproj_{\arr}(\elI) : \model(\arr)(\rarep_{\arr}(\elI)) \to
  \elI\) is a pullback square
  \[
    \begin{tikzcd}
      \model(\bas)/\{\elI\}_{\arr}
      \arrow[r,"\rarep_{\arr}(\elI)"]
      \arrow[d,"\ctxproj_{\arr}(\elI)"'] &
      \model(\obj)
      \arrow[d,"\model(\arr)"] \\
      \model(\bas)/\ctx
      \arrow[r,"\elI"'] &
      \model(\objI).
    \end{tikzcd}
  \]
  We refer to the object \(\{\elI\}_{\arr}\) the \emph{context
    comprehension of \(\elI\) with respect to \(\arr\)}.
\end{definition}

\begin{definition}
  Let \(\model\) be a model of \(\tth\). The class of \emph{contextual
    objects of \(\model\)} is the smallest replete class of objects of
  \(\model(\bas)\) containing the terminal object and closed under
  context comprehension.
\end{definition}

In other words, the contextual objects of
  \(\model\) are inductively defined as follows:
  \begin{itemize}
  \item the terminal object \(\terminal \in \model(\bas)\) is
    contextual;
  \item if \(\ctx \in \model(\bas)\) is a contextual object,
    \(\arr : \obj \to \objI\) is a representable arrow in \(\tth\) and
    \(\elI : \model(\bas)/\ctx \to \model(\objI)\) is a section, then
    the context comprehension \(\{\elI\}_{\arr}\) is contextual;
  \item if \(\ctx \in \model(\bas)\) is a contextual object and
    \(\ctx \simeq \ctxI\), then \(\ctxI\) is contextual.
  \end{itemize}

\begin{definition}
  We call a model \(\model\) \emph{democratic} if all the objects of
  \(\model(\bas)\) are contextual. We denote by \(\Mod^{\dem}(\tth)\) the
  full subcategory of \(\Mod(\tth)\) spanned by the democratic models.
\end{definition}

One can always find a largest democratic model contained in an
arbitrary model of \(\tth\).

\begin{definition}
  For a model \(\model\) of \(\tth\), we define a model
  \(\model^{\heartsuit}\) of \(\tth\) called the \emph{heart of
    \(\model\)} as follows:
  \begin{itemize}
  \item the base \(\infty\)-category \(\model^{\heartsuit}(\bas)\) is
    the full subcategory of \(\model(\bas)\) spanned by the contextual
    objects;
  \item the functor
    \(\model^{\heartsuit} : \tth \to \RFib_{\model^{\heartsuit}(*)}\)
    is the composite with the pullback along the inclusion
    \(\model^{\heartsuit}(\bas) \to \model(\bas)\)
    \[
      \begin{tikzcd}
        \tth
        \arrow[r,"\model"] &
        \RFib_{\model(\bas)}
        \arrow[r] &
        \RFib_{\model^{\heartsuit}(\bas)}.
      \end{tikzcd}
    \]
  \end{itemize}
  \(\model^{\heartsuit}\) is indeed a model of \(\tth\), and the
  inclusion \(\model^{\heartsuit} \hookrightarrow \model\) is a
  morphism of models of \(\tth\). By definition, the functor
  \(\model^{\heartsuit} : \tth \to \RFib_{\model^{\heartsuit}(*)}\)
  preserves finite limits. Since \(\model^{\heartsuit}(\bas)\) is
  closed under context comprehension, for a representable arrow
  \(\arr : \obj \to \objI\) in \(\tth\), the composite
  \(\model^{\heartsuit}(\objI) \hookrightarrow \model(\objI)
  \overset{\rarep_{\arr}}{\longrightarrow} \model(\obj)\) factors
  through \(\model^{\heartsuit}(\obj) \hookrightarrow \model(\obj)\)
  \[
    \begin{tikzcd}
      \model^{\heartsuit}(\objI)
      \arrow[r,dotted]
      \arrow[d,hook] &
      \model^{\heartsuit}(\obj)
      \arrow[d,hook] \\
      \model(\objI)
      \arrow[r,"\rarep_{\arr}"'] &
      \model(\obj).
    \end{tikzcd}
  \]
  This means that \(\model^{\heartsuit}(\arr) :
  \model^{\heartsuit}(\obj) \to \model^{\heartsuit}(\objI)\) has a
  right adjoint, and the square
  \[
    \begin{tikzcd}
      \model^{\heartsuit}(\obj)
      \arrow[r,hook]
      \arrow[d,"\model^{\heartsuit}(\arr)"'] &
      \model(\obj)
      \arrow[d,"\model(\arr)"] \\
      \model^{\heartsuit}(\objI)
      \arrow[r,hook] &
      \model(\objI)
    \end{tikzcd}
  \]
  satisfies the Beck-Chevalley condition. Since the pushforward along
  \(\model^{\heartsuit}(\arr)\) is given by the pullback along its right
  adjoint \(\rarep_{\arr}\), we see that
  \(\model^{\heartsuit} : \tth \to \RFib_{\model^{\heartsuit}(\bas)}\)
  preserves pushforwards along representable maps.
\end{definition}

\begin{proposition}
  \label{heart-of-model}
  For a democratic model \(\model\) of \(\tth\) and an arbitrary model
  \(\modelI\) of \(\tth\), the inclusion
  \(\modelI^{\heartsuit} \hookrightarrow \modelI\) induces an
  equivalence of spaces
  \[
    \Mod^{\dem}(\tth)(\model, \modelI^{\heartsuit}) \simeq
    \Mod(\tth)(\model, \modelI).
  \]
  In other words, \(({-})^{\heartsuit}\) is a right adjoint of the
  inclusion \(\Mod^{\dem}(\tth) \hookrightarrow \Mod(\tth)\).
\end{proposition}
\begin{proof}
  Because any morphism of models of \(\tth\) preserves contextual
  objects, any morphism \(\model \to \modelI\) from a democratic
  model \(\model\) factors through \(\modelI^{\heartsuit}\).
\end{proof}

\subsection{The initial model}
\label{sec:initial-model}

\begin{definition}
  Recall that the Yoneda embedding
  \(\yoneda_{\tth} : \tth \to \RFib_{\tth}\) preserves all existing
  limits and pushforwards. Therefore, the pair
  \((\tth, \yoneda_{\tth})\) is regarded as a model of \(\tth\). We
  define the \emph{initial model} \(\IM(\tth)\) to be the heart of the
  model \((\tth, \yoneda_{\tth})\).
\end{definition}

The goal of this subsection is to show that \(\IM(\tth)\) is indeed an
initial object of \(\Mod(\tth)\).

By definition, the model \(\IM(\tth)\) is described as follows:
\begin{itemize}
\item the base \(\infty\)-category is \(\tth_{r}\), the full
  subcategory of \(\tth\) spanned by the objects \(\obj\) such that
  the arrow \(\obj \to \terminal\) is representable;
\item \(\IM(\tth)(\objI) = \tth_{r}/\objI\) defined by the pullback
  \[
    \begin{tikzcd}
      \tth_{r}/\objI
      \arrow[r,hook]
      \arrow[d] &
      \tth/\objI
      \arrow[d] \\
      \tth_{r}
      \arrow[r,hook] &
      \tth
    \end{tikzcd}
  \]
  for \(\objI \in \tth\).
\end{itemize}
Alternatively, the functor \(\IM(\tth) : \tth \to \RFib_{\tth_{r}}\)
is defined as the left Kan extension of the Yoneda embedding
\(\yoneda_{\tth_{r}} : \tth_{r} \to \RFib_{\tth_{r}}\) along the
inclusion \(\tth_{r} \hookrightarrow \tth\).
\[
  \begin{tikzcd}
    \tth_{r}
    \arrow[r,"\yoneda_{\tth_{r}}",""'{name=a0}]
    \arrow[d,hook] &
    \RFib_{\tth_{r}} \\
    \tth
    \arrow[ur,dotted,bend right,"\IM(\tth)"',""{name=a1}]
    \arrow[from=a0,to=a1,dotted,Rightarrow,"\simeq"'{near start}]
  \end{tikzcd}
\]

\begin{theorem}
  For an \(\infty\)-type theory \(\tth\), the model \(\IM(\tth)\) is
  an initial object in \(\Mod(\tth)\).
\end{theorem}

\begin{proof}
  We first note that, since \(\Mod(\tth)\) has finite limits, it
  suffices to show that \(\IM(\tth)\) is an initial object in the
  homotopy category of \(\Mod(\tth)\) \parencite[Proposition
  2.2.2]{nrs2019adjointfunctors}, that is, for any model \(\model\) of
  \(\tth\), there exists a morphism \(\IM(\tth) \to \model\) and any
  two morphisms \(\IM(\tth) \to \model\) are equivalent.

  Let \(\model\) be a model of \(\tth\). Suppose that we have a
  morphism \(\funI : \IM(\tth) \to \model\). It is regarded as a pair
  \((\funI(\bas), \funI)\) consisting of a functor
  \(\funI(\bas) : \tth_{r} \to \model(\bas)\) and a natural
  transformation
  \(\funI : \IM(\tth) \To \funI(\bas)^{*}\model : \tth \to
  \RFib_{\tth_{r}}\).

  The Beck-Chevalley condition for a representable arrow
  \(\arrI : \objI \to \objII\) in \(\tth\) means that, for any object
  \((\arr : \obj \to \objII) \in \tth_{r}/\objII\), the square
  \[
    \begin{tikzcd}
      \model(\bas)/\funI(\bas)(\arr^{*}\objI)
      \arrow[r,"\funI(\arr^{*}\objI)(\arrI^{*}\arr)"]
      \arrow[d,"\arr^{*}\arrI"'] &
      [4ex]
      \model(\objI)
      \arrow[d,"\model(\arrI)"] \\
      \model(\bas)/\funI(\bas)(\obj)
      \arrow[r,"\funI(\bas)(\arr)"'] &
      \model(\objII)
    \end{tikzcd}
  \]
  is a pullback. From the special case when \(\objII\) is the terminal
  object, we see that the canonical map
  \(\funI(\objI)(\id_{\objI}) : \model(\bas)/\funI(\bas)(\objI) \to
  \model(\objI)\) is an equivalence for every object
  \(\objI \in \tth_{r}\). In other words, the diagram
  \[
    \begin{tikzcd}
      \tth_{r}
      \arrow[r,"\funI(\bas)"]
      \arrow[d,hook] &
      \model(\bas)
      \arrow[d,"\yoneda_{\model(\bas)}"] \\
      \tth
      \arrow[r,"\model"'] &
      \RFib_{\model(\bas)}
    \end{tikzcd}
  \]
  commutes up to equivalence, and we have an equivalence
  \[
    \begin{tikzcd}
      \tth_{r}
      \arrow[r,"\funI(\bas)"]
      \arrow[d,hook] &
      \model(\bas)
      \arrow[d,"\yoneda_{\model(\bas)}"] \\
      \tth
      \arrow[r,"\model"]
      \arrow[dr,bend right,"\IM(\tth)"',""{name=a0}] &
      \RFib_{\model(\bas)}
      \arrow[d,"\funI(\bas)^{*}"]
      \arrow[from=a0,Rightarrow,"\funI"] \\
      & \RFib_{\tth_{r}}
    \end{tikzcd}
    \simeq
    \begin{tikzcd}
      \tth_{r}
      \arrow[r,"\funI(\bas)"]
      \arrow[ddr,bend right,"\yoneda_{\tth_{r}}"',""{name=a0}] &
      \model(\bas)
      \arrow[d,"\yoneda_{\model(\bas)}"]
      \arrow[from=a0,Rightarrow,"\funI(\bas)_{1}"] \\
      & \RFib_{\model(\bas)}
      \arrow[d,"\funI(\bas)^{*}"] \\
      & \RFib_{\tth_{r}},
    \end{tikzcd}
  \]
  where \(\funI(\bas)_{1}\) is the natural transformation
  \(\funI(\bas)_{\obj, \objI} : \tth_{r}(\obj, \objI) \to
  \model(\bas)(\funI(\bas)(\obj), \funI(\bas)(\objI))\). Hence,
  \(\funI(\bas) : \tth_{r} \to \model(\bas)\) is uniquely determined,
  and then
  \(\funI : \IM(\tth) \To \funI(\bas)^{*}\model : \tth \to
  \RFib_{\tth_{r}}\) is uniquely determined because
  \(\IM(\tth) : \tth \to \RFib_{\tth_{r}}\) is the left Kan extension
  of the Yoneda embedding \(\tth_{r} : \tth_{r} \to \RFib_{\tth_{r}}\)
  along the inclusion \(\tth_{r} \hookrightarrow \tth\). This shows
  that morphisms \(\IM(\tth) \to \model\) are unique up to
  equivalence.

  It suffices now to construct a morphism
  \(\fun : \IM(\tth) \to \model\) of models of \(\tth\). We first
  construct a functor \(\fun(\bas) : \tth_{r} \to \model(\bas)\). For
  an object \(\obj \in \tth_{r}\), the map
  \(\model(\obj) \to \model(\terminal) \simeq \model(\bas)\) of right
  fibrations over \(\model(\bas)\) is representable. Thus, since
  \(\model(\bas)\) has a terminal object, the right fibration
  \(\model(\obj)\) is representable. Hence, the restriction of
  \(\model : \tth \to \RFib_{\model(\bas)}\) along the inclusion
  \(\tth_{r} \to \tth\) factors as a functor
  \(\fun(\bas) : \tth_{r} \to \model(\bas)\) followed by the Yoneda
  embedding.
  \[
    \begin{tikzcd}
      \tth_{r}
      \arrow[r,dotted,"\fun(\bas)"]
      \arrow[d,hook] &
      \model(\bas)
      \arrow[d,"\yoneda_{\model(\bas)}"] \\
      \tth
      \arrow[r,"\model"'] &
      \RFib_{\model(\bas)}
    \end{tikzcd}
  \]
  We then define a natural transformation
  \(\fun : \IM(\tth) \To \fun(\bas)^{*}\model : \tth \to
  \RFib_{\tth_{r}}\) to be the one whose restriction to \(\tth_{r}\)
  is the natural transformation
  \(\fun(\bas) : \yoneda_{\tth_{r}} \To \fun(\bas)^{*}\model\fun(\bas)
  : \tth_{r} \to \RFib_{\tth_{r}}\).
  \[
    \begin{tikzcd}
      \tth_{r}
      \arrow[r,"\fun(\bas)"]
      \arrow[d,hook] &
      \model(\bas)
      \arrow[d,"\yoneda_{\model(\bas)}"] \\
      \tth
      \arrow[r,"\model"]
      \arrow[dr,bend right,"\IM(\tth)"',""{name=a0}] &
      \RFib_{\model(\bas)}
      \arrow[d,"\fun(\bas)^{*}"]
      \arrow[from=a0,Rightarrow,"\fun"] \\
      & \RFib_{\tth_{r}}
    \end{tikzcd}
    \simeq
    \begin{tikzcd}
      \tth_{r}
      \arrow[r,"\fun(\bas)"]
      \arrow[ddr,bend right,"\yoneda_{\tth_{r}}"',""{name=a0}] &
      \model(\bas)
      \arrow[d,"\yoneda_{\model(\bas)}"]
      \arrow[from=a0,Rightarrow,"\fun(\bas)_{1}"] \\
      & \RFib_{\model(\bas)}
      \arrow[d,"\fun(\bas)^{*}"] \\
      & \RFib_{\tth_{r}}
    \end{tikzcd}
  \]

  In order to show that \(\fun\) is a morphism of models of \(\tth\),
  it remains to prove that \(\fun(\bas) : \tth_{r} \to \model(\bas)\)
  preserves terminal objects and that \(\fun\) satisfies the
  Beck-Chevalley condition for representable arrows. The first claim
  is clear by definition. For the second, we have to show that, for
  any representable arrow \(\arrI : \objI \to \objII\) in \(\tth\),
  the square
  \[
    \begin{tikzcd}
      \tth_{r}/\objI
      \arrow[r,"\fun(\objI)"]
      \arrow[d,"\arrI"'] &
      \model(\objI)
      \arrow[d,"\model(\arrI)"] \\
      \tth_{r}/\objII
      \arrow[r,"\fun(\objII)"'] &
      \model(\objII)
    \end{tikzcd}
  \]
  satisfies the Beck-Chevalley condition. It suffices to show that,
  for any arrow \(\arr : \obj \to \objII\) with \(\obj \in \tth_{r}\),
  the composite of squares
  \begin{equation}
    \label{eq:8}
    \begin{tikzcd}
      \tth_{r}/\arr^{*}\objI
      \arrow[r,"\arrI^{*}\arr"]
      \arrow[d,"\arr^{*}\arrI"'] &
      \tth_{r}/\objI
      \arrow[r,"\fun(\objI)"]
      \arrow[d,"\arrI"] &
      \model(\objI)
      \arrow[d,"\model(\arrI)"] \\
      \tth_{r}/\obj
      \arrow[r,"\arr"'] &
      \tth_{r}/\objII
      \arrow[r,"\fun(\objII)"'] &
      \model(\objII)
    \end{tikzcd}
  \end{equation}
  satisfies the Beck-Chevalley condition. By the definition of
  \(\fun\), \cref{eq:8} is equivalent to
  \begin{equation}
    \label{eq:9}
    \begin{tikzcd}
      \tth_{r}/\arr^{*}\objI
      \arrow[r,"\fun(\bas)"]
      \arrow[d,"\arr^{*}\arrI"'] &
      \model(*)/\fun(\bas)(\arr^{*}\objI)
      \arrow[r,"\simeq"]
      \arrow[d,"\fun(\bas)(\arr^{*}\arrI)"'] &
      \model(\arr^{*}\objI)
      \arrow[r,"\model(\arrI^{*}\arr)"]
      \arrow[d,"\model(\arr^{*}\arrI)"] &
      \model(\objI)
      \arrow[d,"\model(\arrI)"] \\
      \tth_{r}/\obj
      \arrow[r,"\fun(\bas)"'] &
      \model(*)/\fun(\bas)(\obj)
      \arrow[r,"\simeq"'] &
      \model(\obj)
      \arrow[r,"\model(\arr)"'] &
      \model(\objII).
    \end{tikzcd}
  \end{equation}
  The right square of \cref{eq:9} satisfies the Beck-Chevalley
  condition by \cref{rep-pullback}. The middle square satisfies the
  Beck-Chevalley condition as the horizontal maps are
  equivalences. The Beck-Chevalley condition for the left square
  asserts that \(\fun(\bas)\) preserves pullbacks of representable
  arrows in \(\tth_{r}\), which is true by the definition of
  \(\fun(\bas)\).
\end{proof}

\subsection{Syntactic models generated by representable theories}
\label{sec:synt-models-repr}

We describe the model \(\SM(\yoneda(\obj))\) for \(\obj \in \tth\),
where
\(\yoneda : \tth^{\op} \to \Th(\tth) \subset \Fun(\tth,
\Space)\) is the Yoneda embedding.

\begin{proposition}
  \label{models-of-slice}
  For an object \(\obj\) of \(\tth\), we have a pullback
  \[
    \begin{tikzcd}
      \Mod(\tth/\obj)
      \arrow[r]
      \arrow[d] &
      \yoneda(\obj)/\Th(\tth)
      \arrow[d] \\
      \Mod(\tth)
      \arrow[r,"\IL"'] &
      \Th(\tth).
    \end{tikzcd}
  \]
\end{proposition}
\begin{proof}
  Recall (\cref{slice-model}) that we have a pullback
  \[
    \begin{tikzcd}
      \Mod(\tth/\obj)
      \arrow[r]
      \arrow[d] &
      \RFib'_{\bullet}
      \arrow[d] \\
      \Mod(\tth)
      \arrow[r,"\obj^{*}"'] &
      \RFib'.
    \end{tikzcd}
  \]
  For an object \((\sh \to \cat) \in \RFib'\), the fiber of
  \(\RFib'_{\bullet}\) over \((\sh \to \cat)\) is the space of global
  sections of \(\sh\). Since the base \(\infty\)-category \(\cat\) has
  a terminal object \(\terminal\), that space is equivalent to the
  fiber of \(\sh\) over \(\terminal\). In other words, we have a
  pullback
  \[
    \begin{tikzcd}
      \RFib'_{\bullet}
      \arrow[r]
      \arrow[d] &
      \terminal / \Space
      \arrow[d] \\
      \RFib'
      \arrow[r,"\terminal^{*}"'] &
      \Space.
    \end{tikzcd}
  \]
  By the definition of \(\IL\), the composite \(\terminal^{*} \circ
  \obj^{*}\) is equivalent to the composite
  \[
    \begin{tikzcd}
      \Mod(\tth)
      \arrow[r,"\IL"] &
      \Th(\tth)
      \arrow[r,"\ev_{\obj}"] &
      \Space.
    \end{tikzcd}
  \]
  By Yoneda, we have a pullback
  \[
    \begin{tikzcd}
      \yoneda(\obj)/\Th(\tth)
      \arrow[r]
      \arrow[d] &
      \terminal / \Space
      \arrow[d] \\
      \Th(\tth)
      \arrow[r,"\ev_{\obj}"'] &
      \Space,
    \end{tikzcd}
  \]
  and then we get a pullback as in the statement.
\end{proof}

By \cref{models-of-slice}, we get an equivalence
\[
  \Mod(\tth/\obj) \simeq (\yoneda(\obj) \downarrow \IL).
\]
Since \(\SM(\yoneda(\obj))\) is the initial object of
\((\yoneda(\obj) \downarrow \IL)\), it is obtained from the initial
model \(\IM(\tth/\obj)\) of \(\tth/\obj\) by restricting the morphism
of \(\infty\)-categories with representable maps
\(\IM(\tth/\obj) : \tth/\obj \to \RFib_{\IM(\tth/\obj)(\bas)}\) along
\(\obj^{*} : \tth \to \tth/\obj\). We thus have a concrete description
of \(\SM(\yoneda(\obj))\) as follows:
\begin{itemize}
\item the base \(\infty\)-category \(\SM(\yoneda(\obj))(\bas)\) is the
  full subcategory of \(\tth/\obj\) spanned by the representable
  arrows over \(\obj\);
\item for objects \(\objI \in \tth\) and
  \((\arr : \obj' \to \obj) \in \SM(\yoneda(\obj))(\bas)\), the fiber
  of \(\SM(\yoneda(\obj))(\objI)\) over \(\arr\) is
  \(\tth/\obj(\arr, \obj^{*}\objI) \simeq \tth(\obj', \objI)\).
\end{itemize}

\subsection{The equivalence of theories and democratic models}
\label{sec:equiv-theor-democr}

\begin{proposition}
  \label{unit-invertible}
  The unit of the adjunction
  \(\SM \adj \IL : \Th(\tth) \to \Mod(\tth)\) is
  invertible. Consequently, the left adjoint
  \(\SM : \Th(\tth) \to \Mod(\tth)\) is fully faithful.
\end{proposition}

\begin{proof}
  Since both functors \(\SM\) and \(\IL\) preserves filtered colimits,
  it suffices to show that the unit
  \(\unit_{\theory} : \theory \to \IL(\SM(\theory))\) is invertible
  for every representable functor \(\theory : \tth \to \Space\). From
  the description of \(\SM(\yoneda(\obj))\) in
  \cref{sec:synt-models-repr}, we have that
  \(\IL(\SM(\yoneda(\obj)))(\objI) \simeq \tth(\obj, \objI) =
  \yoneda(\obj)(\objI)\) and \(\unit_{\yoneda(\obj)}\) is just the
  identity.
\end{proof}

\begin{proposition}
  \label{syntactic-model-democratic}
  The functor \(\SM : \Th(\tth) \to \Mod(\tth)\) factors through
  \(\Mod^{\dem}(\tth) \subset \Mod(\tth)\).
\end{proposition}

\begin{proof}
  Since \(\Mod^{\dem}(\tth) \subset \Mod(\tth)\) is a coreflective
  subcategory by \cref{heart-of-model}, it is closed under
  colimits. Thus, it suffices to show that \(\SM(\yoneda(\obj))\) is
  democratic for every \(\obj \in \tth\), but this follows from the
  description of \(\SM(\yoneda(\obj))\) in
  \cref{sec:synt-models-repr}.
\end{proof}

\begin{proposition}
  \label{IL-dem-conservative}
  The restriction of \(\IL : \Mod(\tth) \to \Th(\tth)\) to
  \(\Mod^{\dem}(\tth) \subset \Mod(\tth)\) is conservative.
\end{proposition}

We first show the following lemma.

\begin{lemma}
  \label{IL-dem-conservative-1}
  Let \(\fun : \model \to \modelI\) be a morphism of models of
  \(\tth\) such that \(\IL(\fun) : \IL(\model) \to \IL(\modelI)\) is
  an equivalence of \(\tth\)-theories. Then, the map
  \[
    \fun(\obj)_{\ctx} : \model(\obj)_{\ctx} \to
    \modelI(\obj)_{\fun(\bas)(\ctx)}
  \]
  is an equivalence of spaces for any contextual object
  \(\ctx \in \model(\bas)\) and any object \(\obj \in \tth\).
\end{lemma}

\begin{proof}
  By induction on the contextual object \(\ctx \in \model(\bas)\). When
  \(\ctx = \terminal\), the map \(\fun(\obj)_{\terminal}\) is an
  equivalence by assumption. Suppose that \(\ctx = \{\el\}_{\arr}\)
  for some contextual object \(\ctx' \in \model(\bas)\), representable
  arrow \(\arr : \objI \to \objII\) in \(\tth\) and section
  \(\el : \model(\bas)/\ctx' \to \model(\objII)\). Since
  \(\model : \tth \to \RFib_{\model(\bas)}\) commutes with the polynomial
  functor \(\poly_{\arr}\), the sections
  \(\model(\bas)/\{\el\}_{\arr} \to \model(\obj)\) correspond to the
  sections of \(\model(\poly_{\arr}\obj) \to \model(\objII)\) over
  \(\el : \model(\bas)/\ctx' \to \model(\objII)\). Thus,
  \(\model(\obj)_{\{\el\}_{\arr}}\) is the pullback
  \[
    \begin{tikzcd}
      \model(\obj)_{\{\el\}_{\arr}}
      \arrow[r]
      \arrow[d] &
      \model(\poly_{\arr}\obj)_{\ctx'}
      \arrow[d] \\
      \Delta^{0}
      \arrow[r,"\el"'] &
      \model(\objII)_{\ctx'}.
    \end{tikzcd}
  \]
  By the induction hypothesis, \(\fun(\poly_{\arr}\obj)_{\ctx'}\) and
  \(\fun(\objII)_{\ctx'}\) are equivalences, and thus
  \(\fun(\obj)_{\{\el\}_{\arr}}\) is an equivalence.
\end{proof}

\begin{proof}[Proof of \cref{IL-dem-conservative}]
  Let \(\fun : \model \to \modelI\) be a morphism between democratic
  models of \(\tth\) and suppose that
  \(\IL(\fun) : \IL(\model) \to \IL(\modelI)\) is an equivalence of
  \(\tth\)-theories. We show that \(\fun\) is an equivalence of models
  of \(\tth\). Since the forgetful functor
  \(\Mod(\tth) \to \Fun(\tth^{\rcone}, \Cat_{\infty})\) is
  conservative, it suffices to show that
  \(\fun(\obj) : \model(\obj) \to \modelI(\obj)\) is an equivalence of
  \(\infty\)-categories for every object \(\obj \in
  \tth^{\rcone}\). \Cref{IL-dem-conservative-1} implies that the
  square
  \[
    \begin{tikzcd}
      \model(\obj)
      \arrow[r,"\fun(\obj)"]
      \arrow[d] &
      \modelI(\obj)
      \arrow[d] \\
      \model(\bas)
      \arrow[r,"\fun(\bas)"'] &
      \modelI(*)
    \end{tikzcd}
  \]
  is a pullback for every \(\obj \in \tth\). It remains to show that
  the functor \(\fun(\bas) : \model(\bas) \to \modelI(*)\) is fully
  faithful and essentially surjective.

  We show by induction on \(\ctxI\) that
  \(\fun(\bas) : \model(\bas)(\ctx, \ctxI) \to \modelI(*)(\fun(\bas)(\ctx),
  \fun(\bas)(\ctxI))\) is an equivalence of spaces for any objects
  \(\ctx, \ctxI \in \model(\bas)\). The case when \(\ctxI = \terminal\)
  is trivial. Suppose that \(\ctxI = \{\el\}_{\arr}\) for some object
  \(\ctxI' \in \model(\bas)\), representable arrow
  \(\arr : \obj \to \objI\) in \(\tth\) and section
  \(\el : \model(\bas)/\ctxI' \to \model(\objI)\). By definition, we have
  a pullback
  \[
    \begin{tikzcd}
      \model(\bas)(\ctx, \{\el\}_{\arr})
      \arrow[r]
      \arrow[d] &
      \model(\obj)_{\ctx}
      \arrow[d,"\model(\arr)_{\ctx}"] \\
      \model(\bas)(\ctx, \ctxI')
      \arrow[r,"\ctxmor \mapsto \ctxmor^{*}\el"'] &
      \model(\objI)_{\ctx}.
    \end{tikzcd}
  \]
  Then, by the induction hypothesis and \cref{IL-dem-conservative-1},
  the map
  \(\fun(\bas) : \model(\bas)(\ctx, \{\el\}_{\arr}) \to
  \modelI(*)(\fun(\bas)(\ctx), \fun(\bas)(\{\el\}_{\arr}))\) is an
  equivalence.

  Finally, we show by induction on \(\ctxI\) that, for any object
  \(\ctxI \in \modelI(*)\), there exists an object
  \(\ctx \in \model(\bas)\) such that
  \(\fun(\bas)(\ctx) \simeq \ctxI\). The case when
  \(\ctxI = \terminal\) is trivial. Suppose that
  \(\ctxI = \{\elI\}_{\arr}\) for some object
  \(\ctxI' \in \modelI(*)\), representable arrow
  \(\arr : \obj \to \objI\) in \(\tth\) and section
  \(\elI : \modelI(*)/\ctxI' \to \modelI(\objI)\). By the induction
  hypothesis, we have an object \(\ctx' \in \model(\bas)\) such that
  \(\fun(\bas)(\ctx') \simeq \ctxI'\). By
  \cref{IL-dem-conservative-1}, we have a section
  \(\el : \model(\bas)/\ctx' \to \model(\objI)\) such that
  \(\fun(\objI)_{\ctx'}(\el) \simeq \elI\). Then
  \(\fun(\bas)(\{\el\}_{\arr}) \simeq \{\elI\}_{\arr}\).
\end{proof}

\begin{theorem}
  \label{theory-model-correspondence}
  For an \(\infty\)-type theory, the restriction of
  \(\IL : \Mod(\tth) \to \Th(\tth)\) to
  \(\Mod^{\dem}(\tth) \subset \Mod(\tth)\) is an equivalence
  \[
    \Mod^{\dem}(\tth) \simeq \Th(\tth).
  \]
\end{theorem}
\begin{proof}
  By \cref{syntactic-model-democratic}, the functor
  \(\IL : \Mod^{\dem}(\tth) \to \Th(\tth)\) has the left adjoint
  \(\SM\). By \cref{unit-invertible}, the unit of this adjunction is
  invertible. By \cref{IL-dem-conservative} and the triangle
  identities, the counit is also invertible.
\end{proof}

\section{Correspondence between type-theoretic structures and categorical structures}
\label{sec:corr-betw-type}

We discuss a correspondence between type-theoretic structures and
categorical structures. Given an \(\infty\)-category \(\cat\) whose
objects are small \(\infty\)-categories equipped with a certain
structure and morphisms are structure-preserving functors, we try to
find an \(\infty\)-type theory \(\tth\) such that
\(\Th(\tth) \simeq \cat\). Such an \(\infty\)-type theory \(\tth\) can
be understood in a couple of ways. Type-theoretically, \(\tth\)
provides \emph{internal languages} for \(\infty\)-categories in
\(\cat\). We will find type-theoretic structures corresponding to
categorical structures like finite limits and
pushforwards. Categorically, \(\tth\) gives a \emph{presentation} of
the \(\infty\)-category \(\cat\) as a localization of a presheaf
\(\infty\)-category. Such a presentation has the advantage that the
\(\infty\)-type theory \(\tth\) often has a simple universal property
from which one can derive a universal property of \(\cat\) (see
\cref{Lex-ump} for example).

The fundamental example of such an \(\infty\)-category \(\cat\) is
\(\cat = \Lex_{\infty}\), the \(\infty\)-category of small left exact
\(\infty\)-categories. In \cref{sec:left-exact-infty}, we introduce an
\(\infty\)-type theory \(\etth_{\infty}\) which is an
\(\infty\)-analogue of Martin-L{\"o}f type theory with extensional identity
types. The main result of this section is to establish an equivalence
\(\Th(\etth_{\infty}) \simeq \Lex_{\infty}\), and this is a higher analogue of the
result of \textcite{clairambault2014biequivalence}. To do this, we
need two preliminaries: one is the \emph{representable map classifier}
of right fibrations over a left exact \(\infty\)-category
(\cref{sec:repr-map-class}) which is used for constructing a
democratic model of \(\etth_{\infty}\) out of a left exact
\(\infty\)-category; the other is the notion of \emph{univalence} in
\(\infty\)-categories with representable maps (\cref{sec:univ-repr-arrows})
which for example makes a type constructor unique up to contractible
choice. We also give two other examples \(\cat = \LCCC_{\infty}\), the
\(\infty\)-category of small locally cartesian closed
\(\infty\)-categories (\cref{sec:locally-cart-clos}), and
\(\cat = \TTh_{\infty}\), the \(\infty\)-category of \(\infty\)-type theories
(\cref{sec:infty-type-theories-1}). The latter example shows that the
notion of \(\infty\)-type theories itself can be written in the
\(\infty\)-type-theoretic language.

\subsection{The representable map classifier}
\label{sec:repr-map-class}

In this preliminary subsection, we review a \emph{representable map
  classifier} over a left exact \(\infty\)-category \(\cat\), that is,
a classifying object for the class of representable maps of right
fibrations over \(\cat\).

\begin{definition}
  Let \(\sect\) denote the category
  \[
    \begin{tikzcd}
      & 0
      \arrow[d] \\
      1
      \arrow[ur]
      \arrow[r,equal] &
      1.
    \end{tikzcd}
  \]
  The inclusion \(\Delta^{1} = \{0 \to 1\} \to \sect\)
  induces a functor
  \(\sect \cotensor \cat \to \Delta^{1} \cotensor \cat = \cat^{\to}\)
  for an \(\infty\)-category \(\cat\). Note that
  \(\sect \cotensor \cat\) is the \(\infty\)-category of sections in
  \(\cat\).
\end{definition}

\begin{definition}
  \label{def-rmcls}
  Let \(\cat\) be a left exact \(\infty\)-category. We define
  \(\rmcls_{\cat}\) to be the largest right fibration over \(\cat\)
  contained in the cartesian fibration \(\cod : \cat^{\to} \to \cat\)
  and \(\genrm_{\cat} : \ptrmcls_{\cat} \to \rmcls_{\cat}\) to be the
  pullback
  \[
    \begin{tikzcd}
      \ptrmcls_{\cat}
      \arrow[r,hook]
      \arrow[d,"\genrm_{\cat}"'] &
      \sect \cotensor \cat
      \arrow[d] \\
      \rmcls_{\cat}
      \arrow[r,hook] &
      \cat^{\to}.
    \end{tikzcd}
  \]
  That is, \(\rmcls_{\cat}\) is the wide subcategory of \(\cat^{\to}\)
  whose morphisms are the pullback squares. We refer to
  \(\rmcls_{\cat}\) as the \emph{representable map classifier} over
  \(\cat\) and \(\genrm_{\cat}\) as the \emph{generic representable
    map} of right fibrations over \(\cat\) because of the following
  proposition.
\end{definition}

\begin{proposition}
  \label{rep-map-classifier}
  Let \(\cat\) be a left exact \(\infty\)-category.
  \begin{enumerate}
  \item \(\genrm_{\cat} : \ptrmcls_{\cat} \to \rmcls_{\cat}\) is a
    representable map of right fibrations over \(\cat\).
  \item For any right fibration \(\sh\) over \(\cat\), the map
    \begin{equation*}
      \RFib_{\cat}(\sh, \rmcls_{\cat}) \to (\RFib_{\cat} / \sh)_{r}
    \end{equation*}
    defined by the pullback of \(\genrm_{\cat}\) is an equivalence,
    where \((\RFib_{\cat} / \sh)_{r}\) denotes the subspace of
    \((\RFib_{\cat} / \sh)^{\simeq}\) spanned by the representable maps
    over \(\sh\).
  \end{enumerate}
\end{proposition}
\begin{proof}
  We first observe that
  \(\genrm_{\cat} : \ptrmcls_{\cat} \to \rmcls_{\cat}\) is a right
  fibration, that is, the functor
  \begin{equation*}
    (\ev_{1}, (\genrm_{\cat})_{*}) : \Delta^{1} \cotensor \ptrmcls_{\cat} \to
    \ptrmcls_{\cat} \times_{\rmcls_{\cat}} \Delta^{1} \cotensor \rmcls_{\cat}
  \end{equation*}
  is an equivalence. Since \(\Delta^{1} \cotensor \rmcls_{\cat}\) is a
  subcategory of \((\Delta^{1} \times \Delta^{1}) \cotensor \cat\) whose objects are
  the pullback squares, this follows from the universal property of
  pullbacks.

  For the representability of \(\genrm_{\cat}\), we use
  \cref{representable-map-1}. Let
  \(\kappa_{\objI} : \cat / \obj \to \rmcls_{\cat}\) be a section which
  corresponds to an arrow \(\objI \to \obj\) in \(\cat\). We show that
  \(\kappa_{\objI}^{*} \ptrmcls_{\cat}\) is representable by
  \(\objI\). Since the diagonal map
  \(\objI \to \objI \times_{\obj} \objI\) is a section of the first
  projection, it determines a section
  \(\diagonal : \cat / \objI \to \ptrmcls_{\cat}\) such that the diagram
  \begin{equation*}
    \begin{tikzcd}
      \cat / \objI
      \arrow[r, "\diagonal"]
      \arrow[d] &
      \ptrmcls_{\cat}
      \arrow[d, "\genrm_{\cat}"] \\
      \cat / \obj
      \arrow[r, "\kappa_{\objI}"'] &
      \rmcls_{\cat}
    \end{tikzcd}
  \end{equation*}
  commutes. This square is a pullback. Indeed, for an object
  \((\arr : \objII \to \obj) \in \cat / \obj\), the fiber of
  \(\genrm_{\cat}\) over \(\kappa_{\objI}(\arr)\) is the space of sections
  of \(\objII \times_{\obj} \objI \to \objII\) which is equivalent to the
  space of sections of \(\objI \to \obj\) over \(\arr\).

  For the second claim, observe that
  \((\RFib_{\cat} / \colim_{\idx \in \idxcat}\sh_{\idx})_{r} \simeq
  \lim_{\idx \in \idxcat}(\RFib_{\cat} / \sh_{\idx})_{r}\) for any
  diagram \((\sh_{\idx})_{\idx \in \idxcat}\) in
  \(\RFib_{\cat}\). Indeed, since \(\RFib_{\cat}\) is an
  \(\infty\)-topos, we have
  \((\RFib_{\cat} / \colim_{\idx \in \idxcat} \sh_{\idx})^{\simeq} \simeq
  \lim_{\idx \in \idxcat}(\RFib_{\cat} / \sh_{\idx})^{\simeq}\), and this
  equivalence is restricted to representable maps by
  \cref{representable-map-1}. Then it is enough to show that the map
  in the statement is an equivalence in the case when \(\sh\) is
  representable by some \(\obj \in \cat\). By definition
  \(\RFib_{\cat}(\cat / \obj, \rmcls_{\cat}) \simeq (\cat / \obj)^{\simeq}\). By
  \cref{representable-map-1}, \((\RFib_{\cat} / (\cat / \obj))_{r}\)
  is the space of arrows \(\objI \to \obj\) of which the pullback along
  an arbitrary arrow \(\objII \to \obj\) exists, but since \(\cat\) has
  pullbacks this is \((\cat / \obj)^{\simeq}\).
\end{proof}

\begin{remark}
  The representable map classifier in \(\RFib_{\cat}\) exists even
  when \(\cat\) is not left exact. In the above proof, we have seen
  that
  \((\RFib_{\cat} / \colim_{\idx \in \idxcat} \sh_{\idx})_{r} \simeq
  \lim_{\idx \in \idxcat}(\RFib_{\cat} / \sh_{\idx})_{r}\) and that
  \((\RFib_{\cat} / \sh)_{r}\) is essentially small. Then, by
  \parencite[Proposition 5.5.2.2]{lurie2009higher}, the functor
  \(\RFib_{\cat}^{\op} \ni \sh \mapsto (\RFib_{\cat} / \sh)_{r} \in \Space\) is
  representable, and the representing object is the representable map
  classifier. From the concrete construction given in
  \cref{def-rmcls}, the construction of the representable map
  classifier in the case when \(\cat\) is left exact is moreover
  functorial: any left exact functor \(\fun : \cat \to \catI\) induces a
  map of right fibrations \(\rmcls_{\cat} \to \rmcls_{\catI}\) over
  \(\fun\).
\end{remark}

\subsection{Univalent representable arrows}
\label{sec:univ-repr-arrows}

In this preliminary subsection, we extend the notion of a univalent
map in a (presentable) locally cartesian closed \(\infty\)-category
\parencite{gepner2017univalence,rasekh2018objects,rasekh2021univalence} to a notion of a
univalent representable arrow in an \(\infty\)-categories with
representable maps.

\begin{definition}
  For objects \(\obj\) and \(\objI\) of an \(\infty\)-category
  \(\cat\) with finite products, let \(\Map(\obj, \objI) \to \cat\)
  denote the right fibration whose fiber over \(\objII\) is
  \(\cat/\objII(\obj \times \objII, \objI \times \objII) \simeq
  \cat(\obj \times \objII, \objI)\). It is defined by the pullback
  \[
    \begin{tikzcd}
      \Map(\obj, \objI)
      \arrow[r]
      \arrow[d] &
      \cat/\objI
      \arrow[d] \\
      \cat
      \arrow[r,"({-} \times \obj)"'] &
      \cat.
    \end{tikzcd}
  \]
  If \(\Map(\obj, \objI)\) is representable, we denote by
  \(\intern{\Map}(\obj, \objI)\) the representing object. We define
  \(\Equiv(\obj, \objI)\) to be the subfibration of
  \(\Map(\obj, \objI)\) spanned by the equivalences
  \(\obj \times \objII \simeq \objI \times \objII\). If
  \(\Equiv(\obj, \objI)\) is representable, we denote by
  \(\intern{\Equiv}(\obj, \objI)\) the representing object.
\end{definition}

\begin{definition}
  Let \(\arr : \obj \to \objI\) be an arrow in a left exact
  \(\infty\)-category \(\cat\). We regard
  \(\arr \times \objI : \obj \times \objI \to \objI \times \objI\) and
  \(\objI \times \arr : \objI \times \obj \to \objI \times \objI\) as
  objects of \(\cat/\objI \times \objI\) and denote by
  \(\Equiv(\arr)\) the right fibration
  \(\Equiv(\arr \times \objI, \objI \times \arr) \to \cat/\objI \times
  \objI\). If \(\Equiv(\arr)\) representable, we denote by
  \(\intern{\Equiv}(\arr)\) the representing object.
\end{definition}

By definition, an
  arrow \(\objII \to \intern{\Equiv}(\arr)\) corresponds to a triple
  \((\arrI_{1}, \arrI_{2}, \arrII)\) consisting of arrows
  \(\arrI_{1}, \arrI_{2} : \objII \to \objI\) and an equivalence
  \(\arrII : \arrI_{1}^{*}\obj \simeq \arrI_{2}^{*}\obj\) over
  \(\objII\).

\begin{definition}
  Let \(\arr : \obj \to \objI\) be an arrow in a left exact
  \(\infty\)-category \(\cat\) such that \(\Equiv(\arr)\) is
  representable. We have a section
  \(|\id| : \objI \to \intern{\Equiv}(\arr)\) over the diagonal
  \(\diagonal : \objI \to \objI \times \objI\) corresponding to the
  identity \(\id : \obj \to \obj\). We say \(\arr\) is
  \emph{univalent} if the arrow
  \(|\id| : \objI \to \intern{\Equiv}(\arr)\) is an equivalence.
\end{definition}

\begin{proposition}
  \label{univalent-arrow-1}
  Let \(\arr : \obj \to \objI\) be an arrow in a left exact
  \(\infty\)-category \(\cat\) such that \(\Equiv(\arr)\) is
  representable. Let \(\kappa_{\arr} : \cat/\objI \to \rmcls_{\cat}\)
  be the section corresponding to \(\arr\) by Yoneda. The following
  are equivalent:
  \begin{enumerate}
  \item \(\arr\) is univalent:
  \item the square
    \[
      \begin{tikzcd}
        \cat/\objI
        \arrow[r,"\kappa_{\arr}"]
        \arrow[d,"\diagonal"'] &
        \rmcls_{\cat}
        \arrow[d,"\diagonal"] \\
        \cat/\objI \times \objI
        \arrow[r,"\kappa_{\arr} \times \kappa_{\arr}"'] &
        \rmcls_{\cat} \times_{\cat} \rmcls_{\cat}
      \end{tikzcd}
    \]
    is a pullback;
  \item \label{item:10} \(\kappa_{\arr} : \cat/\objI \to \rmcls_{\cat}\) is a
    \((-1)\)-truncated map of right fibrations over
    \(\cat\). Equivalently, for any object \(\objII \in \cat\), the map
    \(\cat(\objII, \objI) \to (\cat / \objII)^{\simeq}\) defined by the
    pullback of \(\arr\) is \((-1)\)-truncated.
  \end{enumerate}
\end{proposition}
\begin{proof}
  The same proof as \parencite[Proposition 3.8
  (1)--(3)]{gepner2017univalence} works only assuming the
  representability of \(\Equiv(\arr)\).
\end{proof}

\begin{example}
  \label{univalent-gen-rep-map}
  For any left exact \(\infty\)-category \(\cat\), the generic
  representable map
  \(\genrm_{\cat} : \ptrmcls_{\cat} \to \rmcls_{\cat}\) is a univalent
  representable map in \(\RFib_{\cat}\) by \cref{rep-map-classifier}.
\end{example}

\begin{proposition}
  \label{equiv-repr}
  Let \(\obj\) and \(\objI\) be objects in a left exact
  \(\infty\)-category \(\cat\) and suppose that \(\obj \times \objII\)
  and \(\objI \times \objII\) are exponentiable in \(\cat/\objII\) for
  any object \(\objII \in \cat\).
  \begin{enumerate}
  \item The right fibration \(\Equiv(\obj, \objI) \to \cat\) is
    representable.
  \item Let \(\catI\) be a left exact \(\infty\)-category and
    \(\fun : \cat \to \catI\) a left exact functor. If \(\fun\) sends
    \(\obj \times \objII\) and \(\objI \times \objII\) to
    exponentiable objects over \(\fun\objII\) and commutes with
    exponentiation by \(\obj \times \objII\) and
    \(\objI \times \objII\) for any \(\objII \in \cat\), then the
    canonical arrow
    \(\fun(\intern{\Equiv}(\obj, \objI)) \to
    \intern{\Equiv}(\fun(\obj), \fun(\objI))\) is an equivalence.
  \end{enumerate}
\end{proposition}
\begin{proof}
  The right fibration \(\Equiv(\obj, \objI)\) is equivalent to the
  right fibration \(\BiInv(\obj, \objI)\) of bi-invertible arrows
  whose fiber over \(\objII \in \cat\) is the space of tuples
  \((\arr, \arrI, \unit, \arrII, \counit)\) consisting of arrows
  \(\arr : \obj \times \objII \to \objI \times \objII\) and
  \(\arrI, \arrII : \objI \times \objII \to \obj \times \objII\) over
  \(\objII\) and homotopies \(\unit : \arrI\arr \simeq \id\) and
  \(\counit : \arr\arrII \simeq \id\) over \(\objII\). The right
  fibration \(\BiInv(\obj, \objI)\) is representable by the
  exponentiability of \(\obj \times \objII\) and
  \(\objI \times \objII\). The second assertion is clear from the
  construction of the representing object for \(\BiInv(\obj, \objI)\).
\end{proof}

\begin{corollary}
  Let \(\arr : \obj \to \objI\) be a representable arrow in an
  \(\infty\)-category with representable maps \(\cat\).
  \begin{enumerate}
  \item The right fibration
    \(\Equiv(\arr) \to \cat/\objI \times \objI\) is representable.
  \item If \(\arr\) is univalent, so is \(\fun\arr\) for any morphism
    of \(\infty\)-categories with representable maps
    \(\fun : \cat \to \catI\).
  \end{enumerate}
\end{corollary}

\subsection{Left exact \(\infty\)-categories}
\label{sec:left-exact-infty}

We define an \(\infty\)-type theory \(\etth_{\infty}\) whose theories
are equivalent to small left exact \(\infty\)-categories.

\begin{definition}
  Let \(\cat\) be an \(\infty\)-category with representable maps and
  \(\typeof : \El \to \Ty\) a representable arrow in \(\cat\).
  \begin{itemize}
  \item A \emph{\(\Unit\)-type structure on \(\typeof\)} is a
    pullback square of the form
    \begin{equation}
      \label{eq:1}
      \begin{tikzcd}
        \terminal
        \arrow[r, dotted, "\elUnit"]
        \arrow[d,equal] &
        \El
        \arrow[d,"\typeof"] \\
        \terminal
        \arrow[r, dotted, "\Unit"'] &
        \Ty.
      \end{tikzcd}
    \end{equation}
  \item A \emph{\(\dSum\)-type structure on \(\typeof\)} is a
    pullback square of the form
    \begin{equation}
      \label{eq:2}
      \begin{tikzcd}
        \dom(\typeof \otimes \typeof)
        \arrow[r, dotted, "\pair"]
        \arrow[d,"\typeof \otimes \typeof"'] &
        \El
        \arrow[d,"\typeof"] \\
        \cod(\typeof \otimes \typeof)
        \arrow[r, dotted, "\dSum"'] &
        \Ty
      \end{tikzcd}
    \end{equation}
    where \(\otimes\) is the composition of polynomials
    (\cref{sec:exponentiable-arrows}).
  \item An \emph{\(\Id\)-type structure on \(\typeof\)} is a
    pullback square of the form
    \begin{equation}
      \label{eq:3}
      \begin{tikzcd}
        \El
        \arrow[r, dotted, "\refl"]
        \arrow[d,"\diagonal"'] &
        \El
        \arrow[d,"\typeof"] \\
        \El \times_{\Ty} \El
        \arrow[r, dotted, "\Id"'] &
        \Ty.
      \end{tikzcd}
    \end{equation}
  \end{itemize}
\end{definition}

\begin{proposition}
  \label{univalent-type-constructors}
  Let \(\typeof : \El \to \Ty\) be a univalent representable arrow in
  an \(\infty\)-category with representable maps \(\cat\). Then
  \(\Unit\)-type structures, \(\dSum\)-type structures and
  \(\Id\)-type structures are unique up to contractible
  choice. Moreover, we have the following:
  \begin{enumerate}
  \item \label{item:-3} \(\typeof\) has a \(\Unit\)-type structure if
    and only if all the identity arrows are pullbacks of \(\typeof\);
  \item \label{item:-4} \(\typeof\) has a \(\dSum\)-type structure if
    and only if pullbacks of \(\typeof\) are closed under composition;
  \item \label{item:-5} \(\typeof\) has an \(\Id\)-type structure if
    and only if pullbacks of \(\typeof\) are closed under equalizers:
    if \(\arr : \obj \to \objI\) is a pullback of \(\typeof\) and
    \(\arrI_{1}, \arrI_{2} : \obj' \to \obj\) are arrows such that
    \(\arr\arrI_{1} \simeq \arr\arrI_{2}\), then the equalizer
    \(\obj'' \to \obj'\) of \(\arrI_{1}\) and \(\arrI_{2}\) in
    \(\cat/\objI\) is a pullback of \(\typeof\).
  \end{enumerate}
\end{proposition}
\begin{proof}
  The uniqueness follows from \cref{item:10} of
  \cref{univalent-arrow-1}. The rests are straightforward.
\end{proof}

\begin{definition}
  By a \emph{left exact universe} in an \(\infty\)-category with
  representable maps \(\cat\) we mean a univalent representable arrow
  \(\typeof : \El \to \Ty\) equipped with a \(\Unit\)-type structure,
  a \(\dSum\)-type structure and an \(\Id\)-type structure. We denote
  by \(\etth_{\infty}\) the initial \(\infty\)-type theory containing
  a left exact universe \(\typeof : \El \to \Ty\).
\end{definition}

\begin{theorem}
  \label{etth-dem-lex}
  The functor
  \(\ev_{\bas} : \Mod^{\dem}(\etth_{\infty}) \to \Cat_{\infty}\)
  factors through \(\Lex_{\infty}\) and induces an equivalence
  \[
    \Mod^{\dem}(\etth_{\infty}) \simeq \Lex_{\infty}.
  \]
\end{theorem}

\begin{lemma}
  \label{etth-representable-arrows}
  An arrow in \(\etth_{\infty}\) is representable if and only if it is
  a pullback of \(\typeof : \El \to \Ty\).
\end{lemma}
\begin{proof}
  Let \(\etth_{\infty}'\) denote the \(\infty\)-category with
  representable maps whose underlying \(\infty\)-category is the same
  as \(\etth_{\infty}\) and representable arrows are the pullbacks of
  \(\typeof\). As \(\typeof\) is equipped with a \(\Unit\)-type
  structure and a \(\dSum\)-type structure, the representable arrows in
  \(\etth_{\infty}'\) include all the identities and are closed under
  composition by \cref{univalent-type-constructors}, so \(\etth_{\infty}'\) is
  indeed an \(\infty\)-category with representable maps. By the
  initiality of \(\etth_{\infty}\), the inclusion
  \(\etth_{\infty}' \to \etth_{\infty}\) has a section, and thus
  \(\etth_{\infty}' \simeq \etth_{\infty}\).
\end{proof}

\begin{definition}
  Let \(\model\) be a model of an \(\infty\)-type theory \(\tth\). By
  a \emph{display map} we mean an arrow \(\ctxmor : \ctxI \to \ctx\)
  in \(\model(\bas)\) that is equivalent over \(\ctx\) to
  \(\ctxproj_{\arr} : \{\elI\}_{\arr} \to \ctx\) for some
  representable arrow \(\arr : \obj \to \objI\) in \(\tth\) and
  section \(\elI : \model(\bas)/\ctx \to \model(\objI)\). By
  definition, display maps are stable under pullbacks.
\end{definition}

\begin{lemma}
  \label{etth-model-display-maps-1}
  Let \(\model\) be a model of \(\etth_{\infty}\). An arrow \(\ctxmor
  : \ctx_{1} \to \ctx_{2}\) in \(\model(\bas)\) is a display map if and
  only if there exists a pullback of the form
  \[
    \begin{tikzcd}
      \model(\bas)/\ctx_{1}
      \arrow[r,dotted]
      \arrow[d,"\ctxmor"'] &
      \model(\El)
      \arrow[d,"\model(\typeof)"] \\
      \model(\bas)/\ctx_{2}
      \arrow[r,dotted] &
      \model(\Ty).
    \end{tikzcd}
  \]
\end{lemma}
\begin{proof}
  By \cref{etth-representable-arrows}.
\end{proof}

\begin{lemma}
  \label{disp-map-id-cancel}
  Let \(\disp\) be a pullback-stable class of arrows in an
  \(\infty\)-category \(\cat\). Suppose that \(\disp\) contains all
  the identity and is closed under composition and equalizers. Then,
  for arrows \(\arr : \obj \to \objI\) and
  \(\arrI : \objI \to \objII\), if \(\arrI\) and \(\arrI\arr\) are in
  \(\disp\), so is \(\arr\).
\end{lemma}
\begin{proof}
  The assumption implies that the full subcategory of \(\cat/\objII\)
  spanned by the arrows \(\obj \to \objII\) in \(\disp\) is an
  \(\infty\)-category of fibrant objects in which the weak
  equivalences are the equivalences and the fibrations are the arrows
  in \(\disp\). Therefore, any arrow between fibrations is equivalent
  to a fibration.
\end{proof}

\begin{lemma}
  \label{etth-dem-model-disp}
  Let \(\model\) be a democratic model of \(\etth_{\infty}\). Then
  every arrow in \(\model(\bas)\) is a display map.
\end{lemma}
\begin{proof}
  By
  \cref{univalent-type-constructors,etth-model-display-maps-1,disp-map-id-cancel}.
\end{proof}

\begin{proof}[Proof of \cref{etth-dem-lex}]
  Since display maps are stable under pullbacks and morphisms of
  models commute with pullbacks of display maps,
  \cref{etth-dem-model-disp} implies that the base \(\infty\)-category
  of a democratic model of \(\etth_{\infty}\) has all finite limits
  and that any morphism between democratic models of
  \(\etth_{\infty}\) commutes with finite limits in the base
  \(\infty\)-categories. In other words, the functor
  \(\ev_{\bas} : \Mod^{\dem}(\etth_{\infty}) \to \Cat_{\infty}\)
  factors through \(\Lex_{\infty}\).

  Let \(\cat\) be a left exact \(\infty\)-category. We define a model
  \(\rmcls_{\cat}\) of \(\etth_{\infty}\) by setting
  \(\rmcls_{\cat}(\bas)\) to be \(\cat\) and
  \(\rmcls_{\cat}(\typeof) : \rmcls_{\cat}(\El) \to \rmcls_{\cat}(\Ty)\)
  to be the generic representable map of right fibrations over
  \(\cat\). We have seen in \cref{univalent-gen-rep-map} that the
  generic representable map is univalent. Since \(\cat\) has finite
  limits, representable maps of right fibrations over \(\cat\) are
  closed under equalizers. Thus, by
  \cref{univalent-type-constructors}, \(\rmcls_{\cat}\) is indeed a
  model of \(\etth_{\infty}\). Since the construction of the generic
  representable map for a left exact \(\infty\)-category is functorial, the
  assignment \(\cat \mapsto \rmcls_{\cat}\) is part of a functor
  \[
    \rmcls : \Lex_{\infty} \to \Mod(\etth_{\infty}).
  \]
  The model \(\rmcls_{\cat}\) is democratic as the map
  \(\cat/\obj \to \cat/\terminal\) is representable for every object
  \(\obj \in \cat\).

  We show that the functor
  \(\rmcls : \Lex_{\infty} \to \Mod^{\dem}(\etth_{\infty})\) is an
  inverse of
  \(\ev_{\bas} : \Mod^{\dem}(\etth_{\infty}) \to \Lex_{\infty}\). By
  definition, \(\ev_{\bas} \circ \rmcls \simeq \id\). To show the other
  equivalence \(\rmcls \circ \ev_{\bas} \simeq \id\), let \(\model\) be
  a democratic model of \(\etth_{\infty}\). Since
  \(\rmcls_{\model(\bas)}(\typeof)\) is the generic representable map,
  we have a unique pullback
  \[
    \begin{tikzcd}
      \model(\El)
      \arrow[r, dotted]
      \arrow[d,"\model(\typeof)"'] &
      \rmcls_{\model(\bas)}(\El)
      \arrow[d,"\rmcls_{\model(\bas)}(\typeof)"] \\
      \model(\Ty)
      \arrow[r, dotted, "\map"'] &
      \rmcls_{\model(\bas)}(\Ty).
    \end{tikzcd}
  \]
  It suffices to show that \(\map\) is an equivalence of right
  fibrations over \(\model(\bas)\). Since \(\model(\typeof)\) is
  univalent, the map \(\map\) is \((-1)\)-truncated
  \parencite[Corollary 3.10]{gepner2017univalence}. Recall that the
  objects of \(\rmcls_{\model(\bas)}(\Ty)\) are the arrows of
  \(\model(\bas)\). \Cref{etth-model-display-maps-1} implies that the
  essential image of \(\map\) is the class of display maps in
  \(\model(\bas)\). Then, by \cref{etth-dem-model-disp}, the map
  \(\map\) is essentially surjective and thus an equivalence.
\end{proof}

Consider the image of the arrow \(\typeof : \El \to \Ty\) by the
inclusion
\(\etth_{\infty} \to \Th(\etth_{\infty})^{\op} \simeq
\Mod^{\dem}(\etth_{\infty})^{\op} \simeq \Lex_{\infty}^{\op}\). For a
left exact \(\infty\)-category \(\cat\), we have
\begin{gather*}
  \Th(\etth_{\infty})(\yoneda(\Ty), \IL(\rmcls_{\cat})) \simeq
  \rmcls_{\cat}(\Ty)_{\terminal} \simeq \cat^{\simeq} \\
  \Th(\etth_{\infty})(\yoneda(\El), \IL(\rmcls_{\cat})) \simeq
  \rmcls_{\cat}(\El)_{\terminal} \simeq (\terminal/\cat)^{\simeq}.
\end{gather*}
Hence, the object \(\Ty\) corresponds to the free left exact
\(\infty\)-category \(\free{\Box}\) generated by an object \(\Box\),
the object \(\El\) corresponds to the free left exact
\(\infty\)-category \(\free{\widetilde{\Box} : \terminal \to \Box}\)
generated by an object \(\Box\) and a global section
\(\widetilde{\Box} : \terminal \to \Box\), and the arrow
\(\typeof : \El \to \Ty\) corresponds to the inclusion
\(\inc : \free{\Box} \to \free{\widetilde{\Box} : \terminal \to
  \Box}\). Since
\(\yoneda(\obj)/\Th(\etth_{\infty}) \simeq \Th(\etth_{\infty}/\obj)\),
we see that the inclusion \(\inc\) becomes an exponentiable arrow in
\(\Lex_{\infty}^{\op}\). This makes \(\Lex_{\infty}^{\op}\) an
\(\infty\)-category with representable maps in which the representable
arrows are the pullbacks of \(\inc\), and \(\inc\) is a left exact
universe in \(\Lex_{\infty}^{\op}\). Since \(\Th(\etth_{\infty})\) is
the \(\omega\)-free cocompletion of \(\etth_{\infty}^{\op}\), the
universal property of \(\etth_{\infty}\) gives the following universal
property of \(\Lex_{\infty}\).

\begin{corollary}
  \label{Lex-ump}
  Let \(\cat\) be an \(\infty\)-category with representable
  maps that has all small limits and \(\arr : \obj \to \objI\) a left exact universe in
  \(\cat\). Then there exists a unique morphism of
  \(\infty\)-categories with representable maps
  \(\fun : \Lex_{\infty}^{\op} \to \cat\) that sends \(\inc\) to
  \(\arr\) and preserves small limits.
\end{corollary}
\begin{proof}
  By the definition of \(\etth_{\infty}\), we have a
  unique morphism of \(\infty\)-categories with representable maps
  \(\overline{\fun} : \etth_{\infty} \to \cat\) that sends \(\typeof\)
  to \(\arr\), which uniquely extends to a limit-preserving functor
  \(\fun : \Lex_{\infty}^{\op} \simeq \Th(\etth_{\infty})^{\op} \to
  \cat\). The functor \(\fun\) sends pushforwards along \(\inc\) to
  pushforwards along \(\arr\) because the pushforward functors
  preserve limits and every object of \(\Th(\etth_{\infty})^{\op}\)
  is a limit of objects from \(\etth_{\infty}\).
\end{proof}

\subsection{Locally cartesian closed \(\infty\)-categories}
\label{sec:locally-cart-clos}

We define an \(\infty\)-type theory \(\etth_{\infty}^{\dProd}\) whose
theories are equivalent to small locally cartesian closed
\(\infty\)-categories.

\begin{definition}
  Let \(\cat\) be an \(\infty\)-category with representable maps and
  \(\typeof : \El \to \Ty\) a representable arrow in \(\cat\). A
  \emph{\(\dProd\)-type structure on \(\typeof\)} is a pullback square
  of the form
  \begin{equation}
    \label{eq:10}
    \begin{tikzcd}
      \poly_{\typeof}\El
      \arrow[r, dotted, "\lam"]
      \arrow[d,"\poly_{\typeof}\typeof"'] &
      \El
      \arrow[d,"\typeof"] \\
      \poly_{\typeof}\Ty
      \arrow[r, dotted, "\dProd"'] &
      \Ty.
    \end{tikzcd}
  \end{equation}
\end{definition}

The following is straightforward.

\begin{proposition}
  \label{univalent-pi-types}
  Let \(\typeof : \El \to \Ty\) be a univalent representable arrow in
  an \(\infty\)-category with representable maps. Then \(\dProd\)-type
  structures on \(\typeof\) are unique up to contractible
  choice. Moreover, there exists a \(\dProd\)-type structure on
  \(\typeof\) if and only if pullbacks of \(\typeof\) are closed under
  pushforwards along pullbacks of \(\typeof\).
\end{proposition}

\begin{definition}
  Let \(\etth_{\infty}^{\dProd}\) denote the initial \(\infty\)-type
  theory containing a left exact universe \(\typeof : \El \to \Ty\)
  with a \(\dProd\)-type structure.
\end{definition}

\begin{theorem}
  \label{EPi-dem-lccc}
  The functor
  \(\ev_{\bas} : \Mod^{\dem}(\etth_{\infty}^{\dProd}) \to \Cat_{\infty}\)
  factors through the \(\infty\)-category \(\LCCC_{\infty}\) of small locally
  cartesian closed \(\infty\)-categories and induces an equivalence
  \[
    \Mod^{\dem}(\etth_{\infty}^{\dProd}) \simeq \LCCC_{\infty}.
  \]
\end{theorem}

\begin{proof}
  \Cref{etth-representable-arrows} holds for
  \(\etth_{\infty}^{\dProd}\): an arrow in \(\etth_{\infty}^{\dProd}\)
  is representable if and only if it is a pullback of \(\typeof\). It
  follows from this that the restriction of a democratic model of
  \(\etth_{\infty}^{\dProd}\) along the inclusion
  \(\etth_{\infty} \to \etth_{\infty}^{\dProd}\) is a democratic model
  of \(\etth_{\infty}\). Thus, by \cref{etth-dem-lex}, it suffices to
  show that the composite
  \(\Mod^{\dem}(\etth_{\infty}^{\dProd}) \to
  \Mod^{\dem}(\etth_{\infty}) \overset{\ev_{\bas}}{\longrightarrow}
  \Lex_{\infty}\) factors through \(\LCCC_{\infty}\) and gives rise a
  pullback square
  \[
    \begin{tikzcd}
      \Mod^{\dem}(\etth_{\infty}^{\dProd})
      \arrow[r,dotted]
      \arrow[d] &
      \LCCC_{\infty}
      \arrow[d] \\
      \Mod^{\dem}(\etth_{\infty})
      \arrow[r,"\ev_{\bas}"',"\simeq"] &
      \Lex_{\infty}.
    \end{tikzcd}
  \]
  It suffices to show the following:
  \begin{enumerate}
  \item \label{item:-6} an object \(\model\) in
    \(\Mod^{\dem}(\etth_{\infty})\) is in
    \(\Mod^{\dem}(\etth_{\infty}^{\dProd})\) if and only if
    \(\model(\bas)\) is in \(\LCCC_{\infty}\);
  \item \label{item:-7} for objects
    \(\model, \modelI \in \Mod^{\dem}(\etth_{\infty}^{\dProd})\), a
    morphism \(\fun : \model \to \modelI\) in
    \(\Mod^{\dem}(\etth_{\infty})\) is in
    \(\Mod^{\dem}(\etth_{\infty}^{\dProd})\) if and only if
    \(\fun(\bas) : \model(\bas) \to \modelI(\bas)\) is in
    \(\LCCC_{\infty}\).
  \end{enumerate}

  \Cref{item:-6}. By \cref{univalent-pi-types}, an object
  \(\model \in \Mod^{\dem}(\etth_{\infty})\) is in
  \(\Mod^{\dem}(\etth_{\infty}^{\dProd})\) if and only if
  representable maps of right fibrations over \(\model(\bas)\) are
  closed under pushforwards along representable maps. This is
  equivalent to that the \(\infty\)-category \(\model(\bas)\) is
  locally cartesian closed.

  \Cref{item:-7}. A morphism \(\fun : \model \to \modelI\) in
  \(\Mod^{\dem}(\etth_{\infty})\) between objects from
  \(\Mod^{\dem}(\etth_{\infty}^{\dProd})\) is in
  \(\Mod^{\dem}(\etth_{\infty}^{\dProd})\) if and only if it commutes
  with \(\dProd\)-type structures. Observe that
  \(\model(\dProd) : \model(\poly_{\typeof}\Ty) \to \model(\Ty)\)
  sends a pair of composable arrows \(\arr : \obj \to \objI\) and
  \(\arrI : \objI \to \objII\) in \(\model(\bas)\) to the pushforward
  \(\arrI_{*}\arr : \arrI_{*}\obj \to \objII\). Thus, \(\fun\)
  commutes with \(\dProd\)-type structures if and only if
  \(\fun(\bas) : \model(\bas) \to \modelI(\bas)\) commutes with
  pushforwards.
\end{proof}

\subsection{\(\infty\)-type theories}
\label{sec:infty-type-theories-1}

We define an \(\infty\)-type theory \(\tthR_{\infty}\) whose theories
are equivalent to \(\infty\)-type theories.

\begin{definition}
  Let \(\typeof_{1} : \El_{1} \to \Ty_{1}\),
  \(\typeof_{2} : \El_{2} \to \Ty_{2}\) and
  \(\typeof_{3} : \El_{3} \to \Ty_{3}\) be representable arrows in an
  \(\infty\)-category with representable maps. A
  \emph{\((\typeof_{1}, \typeof_{2}, \typeof_{3})\)-\(\dProd\)-type
    structure} is a pullback of the form
  \begin{equation}
    \label{eq:11}
    \begin{tikzcd}
      \poly_{\typeof_{1}}\El_{2}
      \arrow[r, dotted, "\lam"]
      \arrow[d,"\poly_{\typeof_{1}}\typeof_{2}"'] &
      \El_{3}
      \arrow[d,"\typeof_{3}"] \\
      \poly_{\typeof_{1}}\Ty_{2}
      \arrow[r, dotted, "\dProd"'] &
      \Ty_{3}.
    \end{tikzcd}
  \end{equation}
  Note that if \(\typeof_{3}\) is univalent, then
  \((\typeof_{1}, \typeof_{2}, \typeof_{3})\)-\(\dProd\)-type
  structure are unique up to contractible choice, and there exists a
  \((\typeof_{1}, \typeof_{2}, \typeof_{3})\)-\(\dProd\)-type
  structure if and only if the pushforward of a pullback of
  \(\typeof_{2}\) along a pullback of \(\typeof_{1}\) is a pullback of
  \(\typeof_{3}\).
\end{definition}

\begin{definition}
  We denote by \(\tthR_{\infty}\) the initial \(\infty\)-type theory
  containing the following data:
  \begin{itemize}
  \item a left exact universe \(\typeof : \El \to \Ty\);
  \item a \((-1)\)-truncated arrow \(\Rep \hookrightarrow \Ty\). We
    denote by \(\typeof_{\Rep}\) the pullback of \(\typeof\) along the
    inclusion \(\Rep \hookrightarrow \Ty\);
  \item a \(\Unit\)-type structure and a \(\dSum\)-type
    structure on \(\typeof_{\Rep}\);
  \item a \((\typeof_{\Rep}, \typeof, \typeof)\)-\(\dProd\)-type
    structure.
  \end{itemize}
  Note that the inclusion \(\Rep \hookrightarrow \Ty\) automatically
  commutes with \(\Unit\)-type structures and \(\dSum\)-type
  structures because of univalence.
\end{definition}

\begin{definition}
  Let \(\model\) be a model of \(\tthR_{\infty}\). We say an arrow in
  \(\model(\bas)\) is \emph{representable} if it is a context
  comprehension with respect to \(\typeof_{\Rep}\). Using the
  \((\typeof_{\Rep}, \typeof, \typeof)\)-\(\dProd\)-type structure, we
  see that the pushforward of a display map along a representable map
  exists and is a display map. In particular, if \(\model\) is
  democratic, then \(\model(\bas)\) is an \(\infty\)-type theory and,
  for any morphism \(\fun : \model \to \modelI\) between democratic
  models, \(\fun(\bas) : \model(\bas) \to \modelI(\bas)\) is a
  morphism of \(\infty\)-type theories. Hence, we have a functor
  \[
    \ev_{\bas} : \Mod^{\dem}(\tthR_{\infty}) \to \TTh_{\infty}
  \]
\end{definition}

\begin{theorem}
  \label{tthR-dem-TTh}
  The functor
  \(\ev_{\bas} : \Mod^{\dem}(\tthR_{\infty}) \to \TTh_{\infty}\) is an
  equivalence.
\end{theorem}

\begin{proof}
  Similar to \cref{etth-dem-lex}. For an \(\infty\)-type theory
  \(\cat\), the representable map classifier \(\rmcls_{\cat}(\Ty)\)
  has the full subfibration
  \(\rmcls_{\cat}(\Rep) \subset \rmcls_{\cat}(\Ty)\) spanned by the
  representable arrows in \(\cat\), which defines a democratic model
  of \(\tthR_{\infty}\).
\end{proof}

\section{Internal languages for left exact $\infty$-categories}
\label{sec:intern-lang-left}

In this section, we show \citeauthor{kapulkin2018homotopy}'s
conjecture that the \(\infty\)-category of small left exact
\(\infty\)-categories is a localization of the category of theories
over Martin-L{\"o}f type theory with intensional identity types
\parencite{kapulkin2018homotopy}.

We first introduce a structure of \emph{intensional identity types} in
the context of \(\infty\)-type theory.

\begin{definition}
  Let \(\cat\) be an
  \(\infty\)-category with representable maps and \(\typeof : \El \to
  \Ty\) a representable arrow in \(\cat\). An \emph{\(\Id^{+}\)-type
    structure on \(\typeof\)} is a commutative square of the form
  \begin{equation}
    \label{eq:16}
    \begin{tikzcd}
      \El
      \arrow[r, dotted ,"\refl"]
      \arrow[d, "\diagonal"'] &
      \El
      \arrow[d, "\typeof"] \\
      \El \times_{\Ty} \El
      \arrow[r, dotted, "\Id"'] &
      \Ty
    \end{tikzcd}
  \end{equation}
  equipped with a section \(\elim_{\Id^{+}}\) of the induced arrow
  \[
    (\refl^{*}, \typeof_{*}) : (\Id^{*}\El \To_{\Ty} \Ty^{*} \El) \to
    (\El \To_{\Ty} \Ty^{*} \El) \times_{(\El \To_{\Ty} \Ty^{*} \Ty)}
    (\Id^{*}\El \To_{\Ty} \Ty^{*} \Ty),
  \]
  where \(\To_{\Ty}\) is the exponential in the slice \(\cat /
  \Ty\).
\end{definition}

The codomain of the arrow \((\refl^{*}, \typeof_{*})\) classifies
\emph{lifting problems} for \(\refl\) against \(\typeof\), and the
section \(\elim_{\Id^{+}}\) is considered as a \emph{uniform solution}
to the lifting problems. See
\parencite{awodey2009homotopy,awodey2018natural,kapulkin2021simplicial}
for how this definition is related to syntactically presented
intensional identity types. We note that for an \(\Id\)-type structure
\((\Id, \refl)\), the arrow \((\refl^{*}, \typeof_{*})\) is
invertible, and thus any \(\Id\)-type structure is uniquely extended
to an \(\Id^{+}\)-type structure.

Let \(\itth\) denote the \(1\)-type theory freely generated by a
representable arrow \(\typeof : \El \to \Ty\) equipped with a
\(\Unit\)-type structure, a \(\dSum\)-type structure, and an
\(\Id^{+}\)-type structure. \Textcite{kapulkin2018homotopy}
conjectured that the \(\infty\)-category \(\Lex_{\infty}\) is a
localization of the \(1\)-category \(\Th(\itth)\). Strictly, they work
with contextual categories with a unit type, \(\dSum\)-types, and
intensional identity types instead of theories over \(\itth\) in our
sense, but it is straightforward to see that those contextual
categories are equivalent to democratic models of \(\itth\). They also
gave a specific functor \(\Th(\itth) \to \Lex_{\infty}\) and
conjectured that it is a localization functor. We prove their
conjecture using the theory of \(\infty\)-type theories and the
equivalence \(\Th(\etth_{\infty}) \simeq \Lex_{\infty}\).

We construct the functor
\(\Th(\itth) \to \Lex_{\infty} \simeq \Th(\etth_{\infty})\)
differently from \citeauthor{kapulkin2018homotopy}. A first attempt is
to construct a morphism between \(\itth\) and \(\etth_{\infty}\), but
this fails: since the generating representable arrow \(\typeof\) is
not univalent in \(\itth\), we do not have a morphism
\(\etth_{\infty} \to \itth\); since \(\typeof\) is not \(0\)-truncated
in \(\etth_{\infty}\), we do not have a morphism
\(\itth \to \etth_{\infty}\). We thus introduce an intermediate
\(\infty\)-type theory \(\itth_{\infty}\) defined as the free
\(\infty\)-type theory generated by the same data as \(\itth\) but without
truncatedness. Then
\(\itth\) is the universal \(1\)-type theory under \(\itth_{\infty}\),
and \(\etth_{\infty}\) is the universal \(\infty\)-type theory under
\(\itth_{\infty}\) inverting the morphisms
\(\refl : \El \to \Id^{*}\El\) and
\(|\id| : \Ty \to \intern{\Equiv}(\typeof)\). We thus have a span of
\(\infty\)-type theories
\begin{equation}
  \label{eq:6}
  \begin{tikzcd}
    \itth &
    \itth_{\infty}
    \arrow[l, "\truncmap"']
    \arrow[r, "\locmap"] &
    \etth_{\infty}.
  \end{tikzcd}
\end{equation}
Since any morphism \(\fun : \tth \to \tth'\) between \(\infty\)-type
theories induces an adjunction \(\fun_{!} \adj \fun^{*} : \Th(\tth)
\to \Th(\tth')\) as \(\Th(\tth) = \Lex(\tth, \Space)\), we have a
functor
\begin{equation}
  \label{eq:7}
  \begin{tikzcd}
    \Th(\itth)
    \arrow[r, "\truncmap^{*}"] &
    \Th(\itth_{\infty})
    \arrow[r, "\locmap_{!}"] &
    \Th(\etth_{\infty}).
  \end{tikzcd}
\end{equation}
We define the \emph{weak equivalences} in \(\Th(\itth)\) to be the
morphisms inverted by the functor \(\locmap_{!} \truncmap^{*}\) and
write \(\Loc(\Th(\itth))\) for the localization by the weak
equivalences.

\begin{theorem}
  \label{itth-etth-infty}
  The functor \(\locmap_{!} \truncmap^{*} : \Th(\itth) \to
  \Th(\etth_{\infty})\) induces an equivalence of
  \(\infty\)-categories
  \[
    \Loc(\Th(\itth)) \simeq \Th(\etth_{\infty}).
  \]
  Moreover, the composite \(
  \begin{tikzcd}
    \Th(\itth)
    \arrow[r, "\locmap_{!} \truncmap^{*}"] &
    \Th(\etth_{\infty})
    \arrow[r, "\simeq"] &
    \Lex_{\infty}
  \end{tikzcd}
  \) coincides with the functor considered by \textcite[Conjecture
  3.7]{kapulkin2018homotopy}.
\end{theorem}

\begin{remark}
  The construction of the functor
  \(\locmap_{!} \truncmap^{*} : \Th(\itth) \to \Th(\etth_{\infty})\)
  is easily generalized to extensions with type-theoretic structures
  such as \(\dProd\)-types, (higher) inductive types, and
  universes. For example, if we extend \(\itth\) with
  \(\dProd\)-types, then we have a span
  \[
    \begin{tikzcd}
      \itth^{\dProd} &
      \itth_{\infty}^{\dProd}
      \arrow[l, "\truncmap"']
      \arrow[r, "\locmap"] &
      \etth_{\infty}^{\dProd}
    \end{tikzcd}
  \]
  by extending \(\itth_{\infty}\) with \(\dProd\)-types. We expect
  that similar results to \cref{itth-etth-infty} can be proved for a
  wide range of extensions of \(\itth\), which is left as future
  work. See \cref{sec:generalizations} for discussion.
\end{remark}

\subsection{Proof of the theorem}
\label{sec:proof-theorem}

This subsection is devoted to the proof of
\cref{itth-etth-infty}. We use \citeauthor{cisinski2019higher}'s
results on localizations of \(\infty\)-categories \parencite[Chapter
7]{cisinski2019higher}. We first give the category \(\Th(\itth)\) the structure of a \emph{category
  with weak equivalences and cofibrations} (we recall the definition below) and show that the functor
\(\locmap_{!} \truncmap^{*}\) is \emph{right exact}. We then show that the functor \(\locmap_{!} \truncmap^{*}\) satisfies the \emph{left approximation property} (also recalled below), which implies that the induced functor on localization is an equivalence.

\begin{definition}
  \label{cof-cat}
  A \emph{category with weak equivalences and cofibrations} is a
  category \(\cat\) equipped with two classes of arrows called
  \emph{weak equivalences} and \emph{cofibrations} satisfying the
  conditions below. An object \(\obj \in \cat\) is \emph{cofibrant} if
  the arrow \(\initial \to \obj\) is a cofibration. An arrow is a
  \emph{trivial cofibration} if it is both a weak equivalence and a
  cofibration.
  \begin{enumerate}
  \item \label{item:cof-cat-1} \(\cat\) has an initial object.
  \item \label{item:cof-cat-2} All the identities are trivial cofibrations, and weak
    equivalences and cofibrations are closed under composition.
  \item \label{item:cof-cat-3} The weak equivalences satisfy the \emph{2-out-of-3} property:
    if \(\arr\) and \(\arrI\) are a composable pair of arrows and if
    two of \(\arr\), \(\arrI\), and \(\arrI \arr\) are weak
    equivalences, then so is the rest.
  \item \label{item:cof-cat-4} (Trivial) cofibrations are stable under pushouts along
    arbitrary arrows between cofibrant objects: if
    \(\obj, \obj' \in \cat\) are cofibrant objects,
    \(\cof : \obj \to \objI\) is a (trivial) cofibration, and
    \(\arr : \obj \to \obj'\) is an arbitrary arrow, then the pushout
    \(\arr_{!} \objI\) exists and the arrow
    \(\obj' \to \arr_{!}  \objI\) is a (trivial) cofibration.
  \item \label{item:cof-cat-5} Any arrow \(\arr : \obj \to \objI\) with cofibrant domain
    factors into a cofibration \(\obj \to \objI'\) followed by a weak
    equivalence \(\objI' \to \objI\).
  \end{enumerate}
\end{definition}

\begin{definition}
  Let \(\cat\) be a category with weak equivalences and cofibrations
  and \(\catI\) an \(\infty\)-category with finite colimits. A functor
  \(\fun : \cat \to \catI\) is \emph{right exact} if it sends trivial
  cofibrations between cofibrant objects to invertible arrows and
  preserves initial objects and pushouts of cofibrations along arrows
  between cofibrant objects. A right exact functor \(\fun : \cat \to
  \catI\) has the \emph{left approximation property} if the following
  conditions hold:
  \begin{enumerate}
  \item an arrow in \(\cat\) is a weak equivalence if and only if it
    becomes invertible in \(\catI\);
  \item for any cofibrant object \(\obj \in \cat\) and any arrow
    \(\arr : \fun(\obj) \to \objI\) in \(\catI\), there exists an
    arrow \(\arr' : \obj \to \objI'\) in \(\cat\) such that
    \(\fun(\objI') \simeq \objI\) under \(\fun(\obj)\).
  \end{enumerate}
\end{definition}

\begin{proposition}
  \label{derived-equiv}
  Any right exact functor \(\fun : \cat \to \catI\) with the left
  approximation property induces an equivalence
  \(\Loc(\cat) \simeq \catI\).
\end{proposition}
\begin{proof}
  By \parencite[Proposition 7.6.15]{cisinski2019higher}.
\end{proof}

Our first task will be to show that \(\Th(\itth)\) admits the structure of a category with weak equivalences and cofibrations. We have already defined the weak equivalences in \(\Th(\itth)\) as those
morphisms inverted by \(\locmap_{!} \truncmap^{*}\). We define the
cofibrations in \(\Th(\itth)\) as follows. Recall that
\(\yoneda : \itth^{\op} \to \Th(\itth)\) is the Yoneda embedding and
\(\poly_{\typeof} : \itth \to \itth\) is the polynomial functor
associated with \(\typeof : \El \to \Ty\).

\begin{definition}
  The \emph{generating cofibrations in \(\Th(\itth)\)} are the
  following morphisms:
  \begin{itemize}
  \item \(\yoneda (\poly_{\typeof}^{\nat}(\terminal)) \to \yoneda
    (\poly_{\typeof}^{\nat}(\Ty))\) for \(\nat \ge 0\);
  \item \(\yoneda (\poly_{\typeof}^{\nat}(\typeof)) : \yoneda
    (\poly_{\typeof}^{\nat}(\Ty)) \to \yoneda
    (\poly_{\typeof}^{\nat}(\El))\) for \(\nat \ge 0\).
  \end{itemize}
  The class of \emph{cofibrations in \(\Th(\itth)\)} is the closure of
  the generating cofibrations under retracts, pushouts along arbitrary
  morphisms, and transfinite composition. Cofibrations in
  \(\Th(\itth_{\infty})\) and \(\Th(\etth_{\infty})\) are defined in
  the same way. Note that the functors \(\truncmap_{!}\) and
  \(\locmap_{!}\) preserve generating cofibrations.
\end{definition}

\begin{remark}
  \label{remark-itth-cof}
  Our choice of generating cofibrations in \(\Th(\itth)\) coincides
  with the choice by \textcite{kapulkin2018homotopy}. That is,
  \(\yoneda (\poly_{\typeof}^{\nat}(\terminal))\) is the theory freely
  generated by a context of length \(\nat\),
  \(\yoneda (\poly_{\typeof}^{\nat}(\Ty))\) is the theory freely
  generated by a type over a context of length \(\nat\), and
  \(\yoneda (\poly_{\typeof}^{\nat}(\El))\) is the theory freely
  generated by a term over a context of length \(\nat\). This is
  verified as follows. Let \(\theory\) be an \(\itth\)-theory and let
  \(\model\) be the democratic model of \(\itth\) corresponding to
  \(\theory\) via the equivalence
  \(\Th(\itth) \simeq \Mod^{\dem}(\itth)\). By construction, a
  morphism \(\yoneda \obj \to \theory\) correspond to a global section
  \(\model(\bas) \to \model(\obj)\) for any object \(\obj \in
  \itth\). Then, by the universal property of \(\poly_{\typeof}\), a
  morphism
  \(\ctx : \yoneda (\poly_{\typeof}^{\nat}(\terminal)) \to \theory\)
  corresponds to a list of sections
  \begin{equation}
    \label{eq:13}
    \begin{split}
      \ctx_{1}
      &: \model(\bas) / \{\ctx_{0}\} \to \model(\Ty) \\
      \ctx_{2}
      &: \model(\bas) / \{\ctx_{1}\} \to \model(\Ty) \\
      \vdots \\
      \ctx_{\nat}
      &: \model(\bas) / \{\ctx_{\nat - 1}\} \to \model(\Ty)
    \end{split}
  \end{equation}
  where \(\{\ctx_{0}\} = \terminal\) and
  \(\model(\bas) / \{\ctx_{\idx + 1}\} \simeq \ctx_{\idx}^{*}
  \model(\El)\). Since we think of sections of \(\model(\Ty)\) as
  types, such a list of sections can be regarded as a context of
  length \(\nat\). Under this identification, an extension
  \(\yoneda (\poly_{\typeof}^{\nat}(\Ty)) \to \theory\) of \(\ctx\)
  corresponds to a section
  \(\model(\bas) / \{\ctx_{\nat}\} \to \model(\Ty)\), that is, a type
  over \(\ctx\). Similarly, an extension
  \(\yoneda (\poly_{\typeof}^{\nat}(\El)) \to \theory\) of \(\ctx\)
  corresponds to a term
  \(\model(\bas) / \{\ctx_{\nat}\} \to \model(\El)\), that is, a term
  over \(\ctx\). Hence, morphisms from
  \(\yoneda (\poly_{\typeof}^{\nat}(\terminal))\),
  \(\yoneda (\poly_{\typeof}^{\nat}(\Ty))\), and
  \(\yoneda (\poly_{\typeof}^{\nat}(\El))\) correspond to contexts,
  types, and terms, respectively. In this view, a cofibration in
  \(\Th(\itth)\) is an extension by types and terms, but without any
  equation. In particular, cofibrant \(\itth\)-theories are those
  freely generated by types and terms.
\end{remark}

\begin{theorem}
  \label{itth-cof-cat}
  The classes of cofibrations and weak equivalences endow \(\Th(\itth)\) with the structure of a category with weak equivalences and cofibrations.
\end{theorem}

By definition, \(\Th(\itth)\) satisfies
\cref{item:cof-cat-1,item:cof-cat-2,item:cof-cat-3} of \cref{cof-cat},
and cofibrations are stable under arbitrary pushouts. To make
\(\Th(\itth)\) a category with weak equivalences and cofibrations, it
remains to verify the stability of trivial cofibrations under pushouts
and the factorization axiom. The former is true if the functor
\(\locmap_{!}  \truncmap^{*} : \Th(\itth) \to \Th(\etth_{\infty})\)
preserves initial objects and pushouts of cofibrations along morphisms
between cofibrant objects. Note that this also implies that the
functor \(\locmap_{!}  \truncmap^{*}\) must be right exact. Since
\(\locmap_{!}\) preserves all colimits, it suffices to show that
\(\truncmap^{*}\) has this property. For the latter, we introduce the notion of \emph{trivial fibration}.

\begin{definition}
  A morphism \(\map : \theory \to \theoryI\) in \(\Th(\itth)\) is a
  \emph{trivial fibration} if it has the right lifting property
  against cofibrations: for any commutative square
  \[
    \begin{tikzcd}
      \sh
      \arrow[r, "\mapI"]
      \arrow[d, "\cof"'] &
      \theory
      \arrow[d, "\map"] \\
      \shI
      \arrow[r, "\mapII"'] &
      \theoryI
    \end{tikzcd}
  \]
  in which \(\cof\) is a cofibration, there exists a morphism
  \(\mapIII : \shI \to \theory\) such that \(\map \mapIII = \mapII\)
  and \(\mapIII \cof = \mapI\). Trivial fibrations in
  \(\Th(\itth_{\infty})\) and \(\Th(\etth_{\infty})\) are defined in
  the same way. By a standard argument in model category theory,
  \(\map\) is a trivial fibration if and only if it has the right
  lifting property against generating cofibrations.
\end{definition}

By the \emph{small object argument}, we know that any morphism in
\(\Th(\itth)\) factors into a cofibration followed by a trivial
fibration. Thus, to show that \(\Th(\itth)\) satisfies the factorization axiom it is enough to show that trivial fibrations are
inverted by \(\locmap_{!} \truncmap^{*}\). In conclusion, theorem \ref{itth-cof-cat} will follow from the following two propositions.

\begin{proposition}
  \label{itth-cof-obj}
  The functor \(\truncmap^{*} : \Th(\itth) \to \Th(\itth_{\infty})\)
  preserves initial objects and pushouts of cofibrations along
  morphisms between cofibrant objects.
\end{proposition}

\begin{proposition}
  \label{itth-triv-fib}
  Trivial fibrations in \(\Th(\itth)\) are inverted by the functor
  \(\locmap_{!}  \truncmap^{*} : \Th(\itth) \to
  \Th(\etth_{\infty})\).
\end{proposition}

We begin by proving \cref{itth-triv-fib}. It can be broken into the
following two lemmas.

\begin{lemma}
  \label{etth-triv-fib}
  In \(\Th(\etth_{\infty})\), the trivial fibrations are precisely the
  invertible morphisms. Equivalently, all the morphisms are
  cofibrations.
\end{lemma}

\begin{lemma}
  \label{itth-unit-triv-fib}
  For any \(\itth\)-theory \(\theory\), the unit
  \(\truncmap^{*} \theory \to \locmap^{*} \locmap_{!} \truncmap^{*}
  \theory\) is a trivial fibration.
\end{lemma}

\begin{proof}[Proof of \cref{itth-triv-fib}]
  Let \(\map : \theory \to \theoryI\) be a trivial fibration in
  \(\Th(\itth)\). Consider the naturality square
  \[
    \begin{tikzcd}
      \truncmap^{*} \theory
      \arrow[r, "\unit_{\truncmap^{*} \theory}"]
      \arrow[d, "\truncmap^{*} \map"'] &
      \locmap^{*} \locmap_{!} \truncmap^{*} \theory
      \arrow[d, "\locmap^{*} \locmap_{!} \truncmap^{*} \map"] \\
      \truncmap^{*} \theoryI
      \arrow[r, "\unit_{\truncmap^{*} \theoryI}"'] &
      \locmap^{*} \locmap_{!} \truncmap^{*} \theoryI
    \end{tikzcd}
  \]
  where \(\unit\) is the unit of the adjunction
  \(\locmap_{!} \adj \locmap^{*}\). By an adjoint argument,
  \(\truncmap^{*} \map\) is a trivial fibration. By
  \cref{itth-unit-triv-fib}, \(\unit_{\truncmap^{*} \theory}\) and
  \(\unit_{\truncmap^{*} \theoryI}\) are trivial fibrations. Since the
  domains of the generating cofibrations are cofibrant, it follows
  that \(\locmap^{*} \locmap_{!} \truncmap^{*} \map\) is a trivial
  fibration. Then, again by an adjoint argument,
  \(\locmap_{!} \truncmap^{*} \map\) is a trivial fibration and thus
  invertible by \cref{etth-triv-fib}.
\end{proof}

\Cref{etth-triv-fib} is straightforward.

\begin{proof}[Proof of \cref{etth-triv-fib}]
  Since the representable arrow \(\typeof\) in \(\etth_{\infty}\) has
  an \(\Id\)-type structure, the diagonal
  \(\El \to \El \times_{\Ty} \El\) is a pullback of \(\typeof\). This
  implies that the codiagonal
  \(\yoneda (\poly_{\typeof}^{\nat}(\El \times_{\Ty} \El)) \to \yoneda
  (\poly_{\typeof}^{\nat}(\El))\) in \(\Th(\etth_{\infty})\) is a
  cofibration for \(\nat \ge 0\). Similarly, the univalence of
  \(\typeof\) implies that the codiagonal
  \(\yoneda (\poly_{\typeof}^{\nat}(\Ty \times \Ty)) \to \yoneda
  (\poly_{\typeof}^{\nat}(\Ty))\) in \(\Th(\etth_{\infty})\) is a cofibration
  for \(\nat \ge 0\). Hence, for any generating cofibration
  \(\cof : \sh \to \shI\) in \(\Th(\etth_{\infty})\), the codiagonal
  \(\shI +_{\sh} \shI \to \shI\) is a cofibration, and thus
  cofibrations in \(\Th(\etth_{\infty})\) are closed under
  codiagonal. It then follows that cofibrations in
  \(\Th(\etth_{\infty})\) has the right cancellation property: for a
  composable pair of morphisms \(\map\) and \(\mapI\), if \(\map\) and
  \(\mapI \map\) are cofibrations, then so is \(\mapI\). Therefore, it
  suffices to show that all the objects of \(\Th(\etth_{\infty})\) are
  cofibrant. One can show that
  \(\yoneda (\poly_{\typeof}^{\nat}(\Ty))\)'s and
  \(\yoneda (\poly_{\typeof}^{\nat}(\El))\)'s generate
  \(\Th(\etth_{\infty})\) under colimits. Since they are cofibrant,
  all the objects are cofibrant.
\end{proof}

For \cref{itth-cof-obj,itth-unit-triv-fib}, we need analysis of the
functors \(\truncmap^{*}\) and \(\locmap_{!}\). We work with democratic
models instead of theories via the equivalence
\(\Mod^{\dem}(\tth) \simeq \Th(\tth)\)
(\cref{theory-model-correspondence}). We first note that the functors
\(\truncmap^{*} : \Mod(\itth) \to \Mod(\itth_{\infty})\) and
\(\locmap^{*} : \Mod(\etth_{\infty}) \to \Mod(\itth_{\infty})\) are
fully faithful. More precisely, the models of \(\itth\) are the models
\(\model\) of \(\itth_{\infty}\) such that \(\model(\Ty)\) and
\(\model(\El)\) are \(0\)-truncated objects in
\(\RFib_{\model(\bas)}\), and the models of \(\etth_{\infty}\) are the
models \(\model\) of \(\itth_{\infty}\) such that the map
\(\model(\refl) : \model(\El) \to \model(\Id^{*} \El)\) is invertible
and the representable map \(\model(\typeof)\) is a univalent. It is
also clear from this description that the functors \(\truncmap^{*}\)
and \(\locmap^{*}\) preserve democratic models. Hence, we may identify
the functors \(\truncmap^{*} : \Th(\itth) \to \Th(\itth_{\infty})\)
and \(\locmap^{*} : \Th(\etth_{\infty}) \to \Th(\itth_{\infty})\) with
the inclusions
\(\Mod^{\dem}(\itth) \subset \Mod^{\dem}(\itth_{\infty})\) and
\(\Mod^{\dem}(\etth_{\infty}) \subset \Mod^{\dem}(\itth_{\infty})\),
respectively.

To prove \cref{itth-unit-triv-fib}, we concretely describe
\(\locmap_{!} \model \in \Mod^{\dem}(\etth_{\infty})\) for a
democratic model \(\model\) of \(\itth\). By a standard argument in
the categorical semantics of homotopy type theory
\parencite[e.g.][Theorem 3.2.5]{avigad2015homotopy}, the base category
\(\model(\bas)\) is a category of fibrant objects whose fibrations are
the display maps and whose weak equivalences are homotopy equivalences
defined by the identity types. By the result of
\textcite{szumilo2014two}, the localization \(\Loc(\model(\bas))\) has
finite limits, and the localization functor
\(\model(\bas) \to \Loc(\model(\bas))\) is left exact. The
construction \(\model \mapsto \Loc(\model(\bas))\) is the one
considered by \textcite[Conjecture 3.7]{kapulkin2018homotopy}, and
thus the following lemma implies the second assertion of
\cref{itth-etth-infty}.

\begin{lemma}
  \label{itth-model-localization}
  The functor
  \[
    \locmap_{!} : \Mod^{\dem}(\itth) \subset
    \Mod^{\dem}(\itth_{\infty}) \to \Mod^{\dem}(\etth_{\infty}) \simeq
    \Lex_{\infty}
  \]
  is naturally equivalent to the functor
  \(\model \mapsto \Loc(\model(\bas))\).
\end{lemma}
\begin{proof}
  Let \(\cat\) be a left exact \(\infty\)-category and let
  \(\rmcls_{\cat}\) be the corresponding democratic model of
  \(\etth_{\infty}\). We construct an equivalence of spaces
  \begin{equation}
    \label{eq:12}
    \Mod^{\dem}(\itth_{\infty})(\model, \rmcls_{\cat}) \simeq
    \Lex_{\infty}(\model(\bas), \cat).
  \end{equation}
  Then we see that \(\locmap_{!} \model\) and \(\Loc(\model(\bas))\)
  has the same universal property. Given a morphism
  \(\fun : \model \to \rmcls_{\cat}\) of models of \(\itth_{\infty}\),
  since the weak equivalences in \(\model(\bas)\) is defined by the
  intensional identity types, and since the intensional identity types
  become extensional one in \(\cat\), the underlying functor
  \(\fun_{\bas} : \model(\bas) \to \cat\) is left exact. This defines
  one direction of \cref{eq:12}. For the other, let
  \(\fun_{\bas} : \model(\bas) \to \cat\) be a left exact
  functor. Recall that \(\rmcls_{\cat}(\Ty)\) is the right fibration
  of arrows in \(\cat\) and that \(\rmcls_{\cat}(\El)\) is the right
  fibration of sections in \(\cat\). Then we can construct maps
  \(\fun_{\Ty} : \model(\Ty) \to \rmcls_{\cat}(\Ty)\) and
  \(\fun_{\El} : \model(\El) \to \rmcls_{\cat}(\El)\) of right
  fibrations over \(\fun_{\bas}\) by context comprehension followed by
  \(\fun_{\bas}\). It is straightforward to see that these define a
  morphism \(\model \to \rmcls_{\cat}\) of democratic models of
  \(\itth_{\infty}\). The two constructions are mutually
  inverses.
\end{proof}

We characterize trivial fibrations of democratic models of
\(\itth_{\infty}\) in the same way as \textcite{kapulkin2018homotopy}.

\begin{lemma}
  A morphism \(\fun : \model \to \modelI\) in
  \(\Mod^{\dem}(\itth_{\infty})\) is a trivial fibration if and only
  if the following conditions are satisfied:
  \begin{description}
  \item[Type lifting] for any object \(\ctx \in \model(\bas)\)
    and any section
    \(\sh : \modelI(\bas) / \fun(\ctx) \to \modelI(\Ty)\), there
    exists a section \(\sh' : \model(\bas) / \ctx \to \model(\Ty)\)
    such that \(\fun(\sh') \simeq \sh\);
  \item[Term lifting] for any object \(\ctx \in \model(\bas)\), any
    section \(\sh : \model(\bas) / \ctx \to \model(\Ty)\), and any
    section \(\el : \modelI(\bas) / \fun(\ctx) \to \modelI(\El)\) over
    \(\fun(\sh)\), there exists a section
    \(\el' : \model(\bas) / \ctx \to \model(\El)\) over \(\sh\) such
    that \(\fun(\el') \simeq \el\) over \(\fun(\sh)\).
  \end{description}
\end{lemma}
\begin{proof}
  Let \(\theory\) be the \(\itth_{\infty}\)-theory corresponding to
  \(\model\), that is, \(\theory(\obj)\) is the space of global
  sections of \(\model(\obj)\) for \(\obj \in \itth_{\infty}\). As we
  saw in \cref{remark-itth-cof}, a morphism
  \(\ctx : \yoneda (\poly_{\typeof}^{\nat}(\terminal)) \to \theory\)
  corresponds to a list of sections \labelcref{eq:13}, and extensions
  \(\yoneda (\poly_{\typeof}^{\nat}(\Ty)) \to \theory\) and
  \(\yoneda (\poly_{\typeof}^{\nat}(\El)) \to \theory\) of \(\ctx\)
  correspond to sections
  \(\model(\bas) / \{\ctx_{\nat}\} \to \model(\Ty)\) and
  \(\model(\bas) / \{\ctx_{\nat}\} \to \model(\El)\),
  respectively. Then, type lifting and term lifting implies the right
  lifting property against the generating cofibrations. The converse
  is also true because, since \(\model\) is democratic, any object of
  \(\model(\bas)\) is of the form \(\{\ctx_{\nat}\}\) for some list of
  sections \labelcref{eq:13}.
\end{proof}

\begin{proof}[Proof of \cref{itth-unit-triv-fib}]
  We check type lifting and term lifting along the unit
  \(\unit : \model \to \locmap_{!} \model \simeq \Loc(\model(\bas))\)
  for a democratic model \(\model\) of \(\itth\). Type lifting is
  immediate because any object in \(\Loc(\model(\bas)) / \unit(\ctx)\)
  is represented by a fibration \(\sh \to \ctx\) in
  \(\model(\bas)\). For term lifting, we also need the fact that
  \(\model(\bas)\) is not only a category of fibrant objects but also
  a \emph{tribe} \parencite{joyal2017clans} and in particular a
  \emph{path category} \parencite{vandenberg2018exact}. In this
  special case, a section of \(\unit(\sh) \to \unit(\ctx)\) in
  \(\Loc(\model(\bas))\) for a fibration \(\sh \to \ctx\) in
  \(\model(\bas)\) is represented by a section in \(\model(\bas)\) by
  \parencite[Corollary 2.19]{vandenberg2018exact}.
\end{proof}

This concludes the proof that trivial fibrations are weak equivalences in \(\Th(\itth)\). It remains to show \cref{itth-cof-obj}, which follows from the following theorem.

\begin{theorem}
  \label{itth-infty-cof-obj}
  Any cofibrant object of \(\Mod^{\dem}(\itth_{\infty})\) belongs to
  \(\Mod^{\dem}(\itth)\).
\end{theorem}

\begin{proof}[Proof of \cref{itth-cof-obj}]
  Initial objects are cofibrant, and the pushout of a cofibration
  along a morphism between cofibrant objects is cofibrant. Thus, by
  \cref{itth-infty-cof-obj},
  \(\Mod^{\dem}(\itth) \subset \Mod^{\dem}(\itth_{\infty})\) is closed
  under these colimits.
\end{proof}

\Cref{itth-infty-cof-obj} is the hardest
part. We may think of this theorem as a form of \emph{coherence
  problem}. A general democratic model of \(\itth_{\infty}\) may
contain a lot of non-trivial homotopies, but \cref{itth-infty-cof-obj}
says that all the homotopies in a cofibrant democratic model of
\(\itth_{\infty}\) are trivial.

A successful approach to coherence problems in the categorical
semantics of type theory is to replace a non-split model by a split
model
\parencite{hofmann1995interpretation,lumsdaine2015local}. Following
them, we construct, given a democratic model \(\model\) of
\(\itth_{\infty}\), a democratic model \(\Sp \model\) of \(\itth\)
equipped with a trivial fibration \(\counit : \Sp \model \to
\model\). Then \cref{itth-infty-cof-obj} follows from a retract
argument.

The construction of \(\Sp \model\) crucially relies on
\citeauthor{shulman2019toposes}'s result of replacing any
(Grothendieck) \(\infty\)-topos by a well-behaved model category
called a \emph{type-theoretic model topos}
\parencite{shulman2019toposes}. Let \(\model\) be a democratic model
of \(\itth_{\infty}\). Recall that it consists of a base
\(\infty\)-category \(\model(\bas)\), a representable map
\(\model(\typeof) : \model(\El) \to \model(\Ty)\) of right fibrations
over \(\model(\bas)\), and some other structures. Since the
\(\infty\)-category \(\RFib_{\model(\bas)}\) is an \(\infty\)-topos,
it is a localization
\(\locmap_{\catX} : \catX \to \RFib_{\model(\bas)}\) of some
type-theoretic model topos \(\catX\) \parencite[Theorem
11.1]{shulman2019toposes}. Then there exists a fibration
\(\typeof_{\catX} : \El_{\catX} \to \Ty_{\catX}\) between fibrant
objects in \(\catX\) such that
\(\locmap_{\catX}(\typeof_{\catX}) \simeq \model(\typeof)\). We will
choose \(\typeof_{\catX}\) that has a \(\Unit\)-type structure, a
\(\dSum\)-type structure, and an \(\Id^{+}\)-type structure so that it
induces a model of \(\itth\).

We remind the reader that the type-theoretic model topos \(\catX\) has
nice properties by definition \parencite[Definition
6.1]{shulman2019toposes}: the underlying \(1\)-category is a
Grothendieck \(1\)-topos, the cofibrations are the monomorphisms, and
the model structure is right proper. The right properness in
particular implies that the localization functor
\(\locmap_{\catX} : \catX \to \RFib_{\model(\bas)}\) preserves
pullbacks of fibrations and pushforwards of fibrations along
fibrations used in the definitions of \(\Unit\)-type, \(\dSum\)-type,
and \(\Id^{+}\)-type structures.

\begin{lemma}
  \label{lemma-1}
  For any cofibration \(\cof : \sh \to \shI\) between fibrant objects
  in \(\catX\) and for any fibration \(\fib : \shXI \to \shX\) in
  \(\catX\), the induced map
  \[
    (\cof^{*}, \fib_{*}) : \shXI^{\shI} \to \shXI^{\sh}
    \times_{\shX^{\sh}} \shX^{\shI}
  \]
  is a fibration.
\end{lemma}
\begin{proof}
  By an adjoint argument, it suffices to show that for any trivial
  cofibration \(\cof' : \sh' \to \shI'\), the induced map
  \[
    (\cof', \cof) : \sh' \times \shI \amalg_{\sh' \times \sh} \shI'
    \times \sh \to \shI' \times \shI
  \]
  is a trivial cofibration. Since the class of cofibrations are the
  class of monomorphisms in the Grothendieck \(1\)-topos \(\catX\),
  the map \((\cof', \cof)\) is a cofibration. Since \(\sh\) and
  \(\shI\) are fibrant and since the model structure is right proper,
  the maps \(\cof' \times \sh : \sh' \times \sh \to \shI' \times \sh\)
  and \(\cof' \times \shI : \sh' \times \shI \to \shI' \times \shI\)
  are weak equivalences. By 2-out-of-3, the map \((\cof', \cof)\) is a
  weak equivalence.
\end{proof}

\begin{lemma}
  For any choice of \(\typeof_{\catX}\), there exists an
  \(\Id^{+}\)-type structure on \(\typeof_{\catX}\) sent by
  \(\locmap_{\catX} : \catX \to \RFib_{\model(\bas)}\) to the
  \(\Id^{+}\)-type structure on \(\model(\typeof)\).
\end{lemma}
\begin{proof}
  Since all the objects in \(\catX\) are cofibrant and
  \(\typeof_{\catX}\) is a fibration between fibrant objects, the
  commutative square \labelcref{eq:16} for \(\model(\typeof)\) can be
  lifted to one for \(\typeof_{\catX}\). The map
  \(\refl : \El \to \Id^{*} \El\) is a monomorphism in \(\catX\) and
  thus a cofibration. Applying \cref{lemma-1} for the slice
  \(\catX / \Ty\) instead of \(\catX\), we see that the induced map
  \((\refl^{*}, \typeof_{*})\) is a fibration. The codomain of the map
  \((\refl^{*}, \typeof_{*})\) is fibrant by the right
  properness. Hence, the section of \((\refl^{*}, \typeof_{*})\) for
  \(\model(\typeof)\) can be lifted to one for \(\typeof_{\catX}\).
\end{proof}

\begin{lemma}
  \label{split-repl-sigma}
  One can choose \(\typeof_{\catX}\) that has a \(\Unit\)-type
  structure and a \(\dSum\)-type structure sent by
  \(\locmap_{\catX} : \catX \to \RFib_{\model(\bas)}\) to those
  structures on \(\model(\typeof)\).
\end{lemma}
\begin{proof}
  A \(\Unit\)-type structure and a \(\dSum\)-type structure on
  \(\typeof\) are a pullback of the form
  \[
    \begin{tikzcd}
      \dom (\typeof^{\otimes \nat})
      \arrow[r, dotted, "\pair_{\nat}"]
      \arrow[d, "\typeof^{\otimes \nat}"'] &
      \El
      \arrow[d, "\typeof"] \\
      \cod (\typeof^{\otimes \nat})
      \arrow[r, dotted, "\dSum_{\nat}"'] &
      \Ty
    \end{tikzcd}
  \]
  for \(\nat = 0\) and \(\nat = 2\), respectively, where
  \(\typeof^{\otimes \nat}\) is the \(\nat\)-fold composition of the
  polynomial \(\typeof\). Since \(\model\) is a model of
  \(\itth_{\infty}\), the map \(\model(\typeof)\) is equipped with
  such pullbacks in \(\RFib_{\model(\bas)}\). However, they gives rise
  to only \emph{homotopy} pullbacks
  \begin{equation}
    \label{eq:18}
    \begin{tikzcd}
      \dom (\typeof_{\catX}^{\otimes \nat})
      \arrow[r, dotted, "\pair_{\nat}"]
      \arrow[d, "\typeof_{\catX}^{\otimes \nat}"'] &
      \El_{\catX}
      \arrow[d, "\typeof_{\catX}"] \\
      \cod (\typeof_{\catX}^{\otimes \nat})
      \arrow[r, dotted, "\dSum_{\nat}"'] &
      \Ty_{\catX}
    \end{tikzcd}
  \end{equation}
  in \(\catX\), and thus
  \(\typeof_{\catX}\) need not have \(\Unit\)-type and \(\dSum\)-type
  structures.

  The idea of fixing this issue is to replace \(\typeof_{\catX}\) by
  another fibration
  \(\typeof_{\catX}' : \El_{\catX}' \to \Ty_{\catX}'\) between fibrant
  objects such that the pullbacks of \(\typeof_{\catX}'\) are the
  homotopy pullbacks of \(\typeof_{\catX}\). Let \(\card\) be a
  regular cardinal such that \(\typeof_{\catX}\) is
  \(\card\)-small. By \parencite[Theorem 5.22]{shulman2019toposes},
  there exists a fibration
  \(\typeof_{\catX}^{\card} : \El_{\catX}^{\card} \to
  \Ty_{\catX}^{\card}\) between fibrant objects that classifies
  \(\card\)-small fibrations. Moreover, \(\typeof_{\catX}^{\card}\)
  satisfies the univalence axiom with respect to the model
  structure. Since \(\typeof_{\catX}\) is \(\card\)-small, we have a
  pullback
  \[
    \begin{tikzcd}
      \El_{\catX}
      \arrow[r, dotted]
      \arrow[d, "\typeof_{\catX}"'] &
      \El_{\catX}^{\card}
      \arrow[d, "\typeof_{\catX}^{\card}"] \\
      \Ty_{\catX}
      \arrow[r, dotted, "\inc"'] &
      \Ty_{\catX}^{\card}.
    \end{tikzcd}
  \]
  Factor \(\inc\) into a weak equivalence
  \(\inc' : \Ty_{\catX} \to \Ty_{\catX}'\) followed by a fibration
  \(\proj : \Ty_{\catX}' \to \Ty_{\catX}^{\card}\), and define
  \(\typeof_{\catX}' : \El_{\catX}' \to \Ty_{\catX}'\) to be the
  pullback of \(\typeof_{\catX}^{\card}\) along \(\proj\). Since
  \(\typeof_{\catX}^{\card}\) satisfies the univalence axiom, we can
  choose \(\Ty_{\catX}'\) such that the maps \(\sh \to \Ty_{\catX}'\)
  correspond to the triples \((\shI_{1}, \shI_{2}, \map)\) consisting
  of maps \(\shI_{1} : \sh \to \Ty_{\catX}\) and
  \(\shI_{2} : \sh \to \Ty_{\catX}^{\card}\) and a weak equivalence
  \(\map : \shI_{1}^{*} \El_{\catX} \to \shI_{2}^{*}
  \El_{\catX}^{\card}\) over \(\sh\). In particular, we have a
  \emph{generic} homotopy pullback from a \(\card\)-small fibration
  \begin{equation}
    \label{eq:17}
    \begin{tikzcd}
      \El_{\catX}'
      \arrow[r]
      \arrow[d, "\typeof_{\catX}'"'] &
      \El_{\catX}
      \arrow[d, "\typeof_{\catX}"] \\
      \Ty_{\catX}'
      \arrow[r, "\fib"'] &
      \Ty_{\catX}
    \end{tikzcd}
  \end{equation}
  in the sense that any homotopy pullback from a \(\card\)-small
  fibration to \(\typeof_{\catX}\) factors into a strict pullback
  followed by the homotopy pullback \labelcref{eq:17}.

  We now construct \(\Unit\)-type and \(\dSum\)-type structures on
  \(\typeof_{\catX}'\). There are homotopy pullbacks as in
  \cref{eq:18} for \(\nat = 0\) and \(\nat = 2\) sent by
  \(\locmap_{\catX} : \catX \to \RFib_{\model(\bas)}\) to the
  \(\Unit\)-type and \(\dSum\)-type structures, respectively, on
  \(\model(\typeof)\). Since \(\typeof_{\catX}\) is the pullback of
  \(\typeof_{\catX}'\) along the weak equivalence
  \(\inc' : \Ty_{\catX} \to \Ty_{\catX}'\), one can construct a
  commutative square
  \[
    \begin{tikzcd}
      \dom (\typeof_{\catX}^{\otimes \nat})
      \arrow[r, dotted]
      \arrow[d, "\typeof_{\catX}^{\otimes \nat}"'] &
      \dom ((\typeof_{\catX}')^{\otimes \nat})
      \arrow[d, "(\typeof_{\catX}')^{\otimes \nat}"] \\
      \cod (\typeof_{\catX}^{\otimes \nat})
      \arrow[r, dotted] &
      \cod ((\typeof_{\catX}')^{\otimes \nat})
    \end{tikzcd}
  \]
  in which the horizontal maps are weak equivalences. Then we have a
  homotopy pullback from \((\typeof_{\catX}')^{\otimes \nat}\) to
  \(\typeof_{\catX}\), which factors into a strict pullback followed
  by the homotopy pullback \labelcref{eq:17} because the composition
  of polynomials preserves \(\card\)-smallness. By construction, this
  strict pullback is sent by
  \(\locmap_{\catX} : \catX \to \RFib_{\model(\bas)}\) to the
  \(\Unit\)-type structure on \(\model(\typeof)\) when \(\nat = 0\)
  and to the \(\dSum\)-type structure on \(\model(\typeof)\) when
  \(\nat = 2\).
\end{proof}

By the preceding lemmas, we can choose \(\typeof_{\catX}\) that has a
\(\Unit\)-type structure, a \(\dSum\)-type structure, and an
\(\Id^{+}\)-type structure. Then we have a morphism of
\(1\)-categories with representable maps \(\itth \to \catX\), and we
define \(\Sp \model\) to be the heart of the model of \(\itth\)
defined by the composite \(\itth \to \catX \to \RFib_{\catX}\) with
the Yoneda embedding. Concretely, the base category
\((\Sp \model)(\bas)\) is the full subcategory of \(\catX\) spanned by
the objects \(\ctx\) such that the map \(\ctx \to \terminal\) is a
composite of pullbacks of \(\typeof_{\catX}\), and the sections
\((\Sp \model)(\bas) / \ctx \to (\Sp \model)(\Ty)\) and
\((\Sp \model)(\bas) / \ctx \to (\Sp \model)(\El)\) are the maps
\(\ctx \to \Ty_{\catX}\) and \(\ctx \to \El_{\catX}\), respectively,
in \(\catX\).

Since the localization functor
\(\locmap_{\catX} : \catX \to \RFib_{\model(\bas)}\) sends
\(\typeof_{\catX}\) to the representable map \(\model(\typeof)\) and
preserves pullbacks of \(\typeof_{\catX}\) along maps between fibrant
objects, the restriction of \(\locmap_{\catX}\) to
\((\Sp \model)(\bas)\) factors through the Yoneda embedding
\(\model(\bas) \to \RFib_{\model(\bas)}\). Let
\(\counit_{\bas} : (\Sp \model)(\bas) \to \model(\bas)\) be the
induced functor.
\[
  \begin{tikzcd}
    (\Sp \model)(\bas)
    \arrow[r, dotted, "\counit_{\bas}"]
    \arrow[d, hook] &
    \model(\bas)
    \arrow[d, hook, "\yoneda"] \\
    \catX
    \arrow[r, "\locmap_{\catX}"'] &
    \RFib_{\model(\bas)}
  \end{tikzcd}
\]
The functor \(\locmap_{\catX}\) also induces maps
\(\counit_{\Ty} : (\Sp \model)(\Ty) \to \model(\Ty)\) and
\(\counit_{\El} : (\Sp \model)(\El) \to \model(\El)\) of right
fibrations over \(\counit_{\bas}\), and these define a morphism
\(\counit : \Sp \model \to \model\) of models of \(\itth_{\infty}\).

\begin{lemma}
  The morphism \(\counit : \Sp \model \to \model\) is a trivial
  fibration.
\end{lemma}
\begin{proof}
  We verify type lifting and term lifting. To give type lifting, let
  \(\ctx \in (\Sp \model)(\bas)\) be an object and
  \(\sh : \model(\bas) / \counit_{\bas}(\ctx) \to \model(\Ty)\) a
  section. Since
  \(\model(\bas) / \counit_{\bas}(\ctx) \simeq \locmap_{\catX}(\ctx)\)
  and \(\model(\Ty) \simeq \locmap_{\catX}(\Ty_{\catX})\), the section
  \(\sh\) is represented by some map \(\ctx \to \Ty_{\catX}\) in
  \(\catX\), that is, a section
  \((\Sp \model)(\bas) / \ctx \to (\Sp \model)(\Ty)\). Term lifting
  can be checked in the same way.
\end{proof}

\begin{proof}[Proof of \cref{itth-infty-cof-obj}]
  Let \(\model\) be a cofibrant democratic model of
  \(\itth_{\infty}\). Then we have a section of the trivial fibration
  \(\counit : \Sp \model \to \model\). Since
  \(\Mod^{\dem}(\itth) \subset \Mod^{\dem}(\itth_{\infty})\) is closed
  under retracts, \(\model\) belongs to \(\Mod^{\dem}(\itth)\).
\end{proof}

In conclusion we have shown that \(\Th(\itth)\) is a category with weak equivalences and cofibrations. Moreover, \cref{itth-cof-obj} also implies that \(\locmap_{!}\truncmap^{*}\) is left exact. Thus, to show that this map induces an equivalence after localization, it is enough to show the left approximation property. Since the first axiom is
satisfied by definition, we only have to show the second. But this is now an easy task using \cref{etth-triv-fib,itth-unit-triv-fib,itth-infty-cof-obj}.

\begin{lemma}
  \label{itth-left-approximation}
  For any cofibrant \(\itth\)-theory \(\theory\) and any morphism
  \(\map : \locmap_{!} \truncmap^{*} \theory \to \theoryI\) in
  \(\Th(\etth_{\infty})\), there exists a morphism
  \(\map' : \theory \to \theoryI'\) in \(\Th(\itth)\) such that
  \(\locmap_{!}  \truncmap^{*} \theoryI' \simeq \theoryI\) under
  \(\locmap_{!}  \truncmap^{*} \theory\).
\end{lemma}

\begin{proof}
  Let \(\map : \locmap_{!} \truncmap^{*} \theory \to \theoryI\) be a
  morphism in \(\Th(\etth_{\infty})\) where \(\theory\) is a cofibrant
  \(\itth\)-theory. By \cref{etth-triv-fib} and the small object
  argument, \(\map\) is written as a transfinite composite of pushouts
  of generating cofibrations. Thus, it suffices to prove the case when
  \(\map\) is a pushout of a generating cofibration. Let us assume
  that \(\map\) is a pushout of the form
  \begin{equation}
    \label{eq:14}
    \begin{tikzcd}
      \locmap_{!} \sh
      \arrow[r, "\mapI"]
      \arrow[d, "\locmap_{!} \cof"'] &
      \locmap_{!} \truncmap^{*} \theory
      \arrow[d, "\map"] \\
      \locmap_{!} \shI
      \arrow[r, "\mapII"'] &
      \theoryI
    \end{tikzcd}
  \end{equation}
  where \(\cof : \sh \to \shI\) is one of the generating cofibrations
  in \(\Th(\itth_{\infty})\). Since \(\sh\) is cofibrant, the
  transpose \(\sh \to \locmap^{*} \locmap_{!} \truncmap^{*} \theory\)
  of \(\mapI\) factors through the unit
  \(\truncmap^{*} \theory \to \locmap^{*} \locmap_{!} \truncmap^{*}
  \theory\) by \cref{itth-unit-triv-fib}. Let
  \(\mapI' : \truncmap_{!}  \sh \to \theory\) be the transpose of the
  induced morphism \(\sh \to \truncmap^{*} \theory\) and take the
  pushout
  \begin{equation}
    \label{eq:15}
    \begin{tikzcd}
      \truncmap_{!} \sh
      \arrow[r, "\mapI'"]
      \arrow[d, "\truncmap_{!} \cof"'] &
      \theory
      \arrow[d] \\
      \truncmap_{!} \shI
      \arrow[r] &
      \theoryI'.
    \end{tikzcd}
  \end{equation}
  By \cref{itth-infty-cof-obj}, the units
  \(\sh \to \truncmap^{*} \truncmap_{!} \sh\) and
  \(\shI \to \truncmap^{*} \truncmap_{!}  \shI\) are invertible, and
  \(\locmap_{!} \truncmap^{*}\) sends the pushout \labelcref{eq:15} to
  the pushout \labelcref{eq:14} since \(\theory\) is cofibrant. Hence,
  \(\locmap_{!} \truncmap^{*} \theoryI'\) is equivalent to
  \(\theoryI\) under \(\locmap_{!}  \truncmap^{*} \theory\).
\end{proof}

\begin{proof}[Proof of \cref{itth-etth-infty}]
  By \cref{itth-cof-cat}, the category \(\Th(\itth)\) is
  a category with weak equivalences and cofibrations, and \cref{itth-cof-obj} implies that the functor
  \(\locmap_{!} \truncmap^{*} : \Th(\itth) \to \Th(\etth_{\infty})\)
  is right exact. We checked the left approximation property in
  \cref{itth-left-approximation}. Thus, by \cref{derived-equiv},
  \(\locmap_{!} \truncmap^{*}\) induces an equivalence
  \(\Loc(\Th(\itth)) \simeq \Th(\etth_{\infty})\).
\end{proof}

\subsection{Generalizations}
\label{sec:generalizations}

We end this section with discussion about generalizations of
\cref{itth-etth-infty}. Let \(\widetilde{\itth}\) be an extension of
\(\itth\) with some type-theoretic structures such as \(\dProd\)-types
and (higher) inductive types, and we similarly define extensions
\(\widetilde{\itth}_{\infty}\) and \(\widetilde{\etth}_{\infty}\) of
\(\itth_{\infty}\) and \(\etth_{\infty}\), respectively. We have a
span
\[
  \begin{tikzcd}
    \widetilde{\itth} &
    \widetilde{\itth}_{\infty}
    \arrow[l, "\truncmap"']
    \arrow[r, "\locmap"] &
    \widetilde{\etth}_{\infty}
  \end{tikzcd}
\]
and ask if the functor
\(\locmap_{!} \truncmap^{*} : \Th(\widetilde{\itth}) \to
\Th(\widetilde{\etth}_{\infty})\) induces an equivalence
\[
  \Loc(\Th(\widetilde{\itth})) \simeq \Th(\widetilde{\etth}_{\infty}).
\]
Most part of the proof of \cref{itth-etth-infty} works also for this
case, but we have to modify
\cref{itth-model-localization,split-repl-sigma}. For
\cref{itth-model-localization}, we need to find a
\(\infty\)-categorical structure corresponding to
\(\widetilde{\etth}_{\infty}\)-theories and show that the localization
\(\Loc(\model(\bas))\) for a democratic model \(\model\) of
\(\widetilde{\itth}\) has that structure. For example, in the case
when \(\widetilde{\itth}\) is the extension \(\itth^{\dProd}\) of
\(\itth\) with \(\dProd\)-types satisfying function extensionality
in the sense of \parencite[Section 2.9]{hottbook}, we have
\(\Th(\etth_{\infty}^{\dProd}) \simeq \LCCC_{\infty}\)
(\cref{EPi-dem-lccc}), and by the results of
\textcite{kapulkin2017locally}, \(\Loc(\model(\bas))\) is indeed
locally cartesian closed. When we extend \(\itth\) with
(higher) inductive types, the corresponding \(\infty\)-categorical
structure will be some form of pullback-stable initial algebras. For \cref{split-repl-sigma}, we have to choose
the fibration \(\typeof_{\catX}\) such that it also has the
type-theoretic structures that \(\widetilde{\itth}\) has. In the case
of \(\widetilde{\itth} = \itth^{\dProd}\), one might want to choose
the regular cardinal \(\card\) in the proof of \cref{split-repl-sigma} such that \(\card\)-small
fibrations are closed under pushforwards. However, there is no
guarantee of the existence of such a regular cardinal within the same
Grothendieck universe, unless the Grothendieck universe is
\emph{\(1\)-accessible}, that is, there are unboundedly many
inaccessible cardinals \parencite{monaco2019dependent}. Nevertheless,
the existence of \(\Sp \model\) in a larger universe is enough to
prove \cref{itth-infty-cof-obj}, and thus we have the second part of
Conjecture 3.7 of \textcite{kapulkin2018homotopy} under an extra
assumption on universes.

\begin{theorem}
  Suppose that our ambient Grothendieck universe is \(1\)-accessible
  or contained in a larger universe. Then the functor \(\locmap_{!}
  \truncmap^{*} : \Th(\itth^{\dProd}) \to
  \Th(\etth_{\infty}^{\dProd})\) induces an equivalence of
  \(\infty\)-categories
  \[
    \Loc(\Th(\itth^{\dProd})) \simeq \Th(\etth_{\infty}^{\dProd})
    \simeq \LCCC_{\infty}.
  \]
\end{theorem}

The current proof of \cref{split-repl-sigma} has some issues when
generalizing it. As we have seen, it could cause a rise in universe
levels. Furthermore, the same proof does not work when we extend
\(\itth\) with (higher) inductive types, because having (higher)
inductive types is not a closure property. One possible approach to
these issues is to refine the construction of \(\Sp \model\). The
current construction does not depend on the choice of a type-theoretic
model topos \(\catX\) that presents \(\RFib_{\model(\bas)}\), but
there should be a convenient choice to work with. Another approach is
to give a totally different proof of \cref{itth-infty-cof-obj} without
the use of \(\Sp \model\). There has been a syntactic approach to
coherence problems initiated by \textcite{curien1993substitution}. In
this approach, coherence problems are solved by rewriting techniques,
and we expect that it works for a wide range of type-theoretic
structures without a rise of universe levels. Of course, we first have
to develop nice syntax for \(\infty\)-type theories, and this is not
obvious.

\section{Conclusion and future work}
\label{sec:concl-future-work}

We introduced \(\infty\)-type theories as a higher dimensional
generalization of type theories and as an application proved
\citeauthor{kapulkin2018homotopy}'s conjecture that the
\(\infty\)-category of small left exact \(\infty\)-categories is a localization
of the category of theories over Martin-L{\"o}f type theory with
intensional identity types \parencite{kapulkin2018homotopy}. The
technique developed in this paper also works for the internal language
conjecture for locally cartesian closed \(\infty\)-categories, but further
generalization including (higher) inductive types is left as future
work.

\subsection{Syntax for \(\infty\)-type theories}
\label{sec:syntax-infty-type}

Coherence problems are often solved by syntactic arguments
\parencite{curien1993substitution}. Therefore, syntactic presentations
of \(\infty\)-type theories will be helpful for solving internal
language conjectures for structured \(\infty\)-categories. We have not
figured out syntax for \(\infty\)-type theories. Here we consider one
possibility based on \emph{logical frameworks}.

In the previous work \parencite{uemura2019framework}, the author
introduced a logical framework to define type theories. For every
signature \(\sig\) in that logical framework, the syntactic category
\(\synrm(\sig)\) is naturally equipped with a structure of a category
with representable maps and satisfies a certain universal property. To
define \(\infty\)-type theories syntactically, we modify the logical
framework as follows:
\begin{itemize}
\item the new logical framework has intensional identity types instead
  of extensional identity types;
\item dependent product types indexed over representable types satisfy
  the function extensionality axiom.
\end{itemize}

\begin{remark}
  A similar kind of framework is used by \textcite[Section
  7]{bocquet2020coherence} to represent space-valued models of a type
  theory.
\end{remark}

\begin{proposition}
  Let \(\sig\) be a signature in this new logical framework.
  \begin{enumerate}
  \item \label{item:12} The syntactic category \(\synrm(\sig)\) is
    equipped with a structure of a fibration category.
  \item \label{item:13} The localization \(\Loc(\synrm(\sig))\) is
    equipped with a structure of an \(\infty\)-category with
    representable maps.
  \end{enumerate}
\end{proposition}
\begin{proof}
  It is known \parencite{avigad2015homotopy} that the syntactic
  category of a type theory with intensional identity types is a
  fibration category. The second claim is proved in the same way as
  the fact that the localization of a locally cartesian closed
  fibration category is a locally cartesian closed \(\infty\)-category
  \parencite{kapulkin2017locally,cisinski2019higher}.
\end{proof}

We expect that the syntactic \(\infty\)-category with representable maps
\(\Loc(\synrm(\sig))\) satisfies a universal property analogous to
\parencite[Theorem 5.17]{uemura2019framework} so that the logical
framework with intensional identity types provides syntactic
presentations of \(\infty\)-type theories.
